\documentclass[a4paper]{amsart}
\date{\today}
\usepackage[dvipdfmx]{graphicx}
\usepackage{latexsym}
\usepackage{amssymb}
\usepackage{amsmath}
\usepackage{amscd}
\usepackage{amsthm}
\usepackage{amsfonts}
\usepackage{mathrsfs}
\usepackage{enumerate}
\usepackage{enumitem}
\usepackage{bm}
\usepackage[usenames]{color}
\usepackage[active]{srcltx}
\usepackage{txfonts}
\title[Singularities of spacelike mean curvature one surfaces]{%
Singularities of spacelike 
mean curvature one surfaces in de Sitter space}
\author[A.~Honda]{Atsufumi Honda}
\address{%
   Department of Applied Mathematics, 
   Yokohama National University, 
   Yokohama 240-8501, Japan
}
\email{honda-atsufumi-kp@ynu.ac.jp}
\author[H.~Sato]{Himemi Sato}
\address{%
   Graduate School of Engineering Science, 
   Yokohama National University, 
   Yokohama 240-8501, Japan
}
\email{sato-himemi-hr@ynu.jp}
\curraddr{%
    Asahi Mutual Life Insurance Company,
    Tokyo 160-8570, Japan
}
\thanks{The first author is supported by 
JSPS KAKENHI Grant Number 19K14526 and 20H01801.}
\usepackage{amsthm}
\theoremstyle{plain}
 \newtheorem{theorem}{Theorem}[section]
 \newtheorem{introtheorem}{Theorem}

 \newtheorem{proposition}[theorem]{Proposition}
 \newtheorem{fact}[theorem]{Fact}
 
 \newtheorem*{fact*}{Fact}
 \newtheorem{lemma}[theorem]{Lemma}
 \newtheorem{corollary}[theorem]{Corollary}
 \theoremstyle{remark}
 \newtheorem{definition}[theorem]{Definition}
 \newtheorem{remark}[theorem]{Remark}
 \newtheorem*{acknowledgements}{Acknowledgements}
 \newtheorem{example}[theorem]{Example}
\numberwithin{equation}{section}

 \makeatletter
    
    \@addtoreset{equation}{section}
  \makeatother

\newcommand{\Z}{\boldsymbol{Z}}

\newcommand{\R}{\boldsymbol{R}}
\newcommand{\C}{\boldsymbol{C}}
\newcommand{\Herm}{\operatorname{Herm}}

\newcommand{\SL}{\operatorname{SL}}

\newcommand{\SU}{\operatorname{SU}}
\newcommand{\SO}{\operatorname{SO}}

\newcommand{\bmath}[1]{\mbox{\boldmath $#1$}}
\renewcommand{\L}{{\bmath{L}}}

\newcommand{\trace}{\operatorname{trace}}
\newcommand{\inner}[2]{\langle{#1},{#2}\rangle}
\renewcommand{\Re}{\operatorname{Re}}
\renewcommand{\Im}{\operatorname{Im}}
\newcommand{\vect}[1]{\boldsymbol{#1}}

\newcommand{\ep}{E}
\newcommand{\calE}{\mathcal{E}}

\newcommand{\bz}{\bar{z}}
\newcommand{\const}{\kappa}
\newcommand{\constant}{c}
\newcommand{\Imphi}{c_1}
\newcommand{\reDphi}{c_2}
\newcommand{\DD}{\beta}
\newcommand{\vf}{V}

\newcommand{\redl}{\textcolor{Black}{\lambda}}
\newcommand{\homega}{h}
\newcommand{\ce}{{f}_{\rm ce}}
\newcommand{\sw}{{f}_{\rm sw}}
\newcommand{\ccr}{{f}_{\rm ccr}}
\newcommand{\cone}{{f}_{\rm cone}}
\newcommand{\fold}{{f}_{\rm fold}}
\newcommand{\rce}{{f}_{\rm rce}}

\newcommand{\csk}{{f}_{{\rm cs},k}}
\newcommand{\csone}{{f}_{{\rm cs},1}}
\usepackage{multirow}
\usepackage{array}
\newcommand\innerrule{%
   \omit{\let\strut\relax
     \hrulefill\vspace*{-\arrayrulewidth}%
   }%
}
\newcommand\vcentertext[1]{%
   \let\strut\relax
   \vbox to0pt{\vss \hbox{#1} \vss}%
}

\begin{document}
\begin{abstract}
%
In this paper, we study the singularities 
of spacelike constant mean curvature one (CMC $1$) surfaces 
in the de Sitter $3$-space.
We prove 
the duality between 
generalized conelike singular points and $5/2$-cuspidal edges
on spacelike CMC $1$ surfaces.
To describe the duality 
between $A_{k+3}$ singularities and
cuspidal $S_k$ singularities,
we introduce two invariants,
called the {\it $\alpha$-invariant} and {\it $\sigma$-invariant},
of spacelike CMC $1$ surfaces
at their singular points.
Moreover, we give a classification of non-degenerate singular points
on spacelike CMC $1$ surfaces.
\end{abstract}
\maketitle


\section{Introduction}

\subsection{Background and Motivation}
We denote by $\L^3$ the Lorentz-Minkowski $3$-space
$\L^3=(\R^3,~\inner{~}{~}=dx^2+dy^2-dz^2)$.
A spacelike surface with zero mean curvature in $\L^3$
is called a {\it spacelike maximal surface}.
While the local properties of spacelike maximal surfaces in $\L^3$
are similar to those of minimal surfaces in the Euclidean space $\R^3$,
their global behaviors are not the same.
It is known that {\it complete spacelike maximal surfaces must be spacelike planes}
(Calabi \cite{C}).
So it is natural to investigate the spacelike maximal surfaces with singular points.
Umehara-Yamada \cite{UY_complete} introduced the notion of {\it maxfaces}
which is the class of spacelike maximal surfaces with admissible singular points in $\L^3$,
and proved several global properties, such as the Osserman-type inequality
\cite[Theorem 4.11]{UY_complete}.

As in the case of minimal surfaces in $\R^3$,
maxfaces in $\L^3$ have the Weierstrass-type representation formula
(\cite{UY_complete}, \cite{Kobayashi}):
\[\begin{minipage}{0.8\linewidth}
{\it 
Let $D$ be a simply connected domain of $\C$,
and $(g,\omega)$ be a pair of 
meromorphic function $g$,
and holomorphic $1$-form $\omega$ on $D$
such that $(1+|g|^2)^2|\omega|^2$ is a Riemannian metric on $D$
and $(1-|g|^2)^2$ does not vanish identically.
Then $f := \Re F$ defines a maxface $f:D\to \L^3$,
where $F : D \to \C^3$ is given by
$$
  F(z) := \int_{z_0}^z (1+g^2,\sqrt{-1}(1-g^2),-2g)\omega
  \quad (z\in D),
$$
and $z_0\in D$ is a base point.
Conversely, 
any maxface is locally obtained in this manner.
}\end{minipage}\]
The pair $(g,\omega)$ is called the {\it Weierstrass data}.
As in the case of minimal surfaces in $\R^3$,
the maxface $f^\sharp : D\to \L^3$ defined by 
$
  f^\sharp := \Im F
$
is called the {\it conjugate} of $f$.
The corresponding Weierstrass data is $(g,-\sqrt{-1}\omega)$.

We say that maxfaces have the 
{\it duality between singular points of type $X$ and type $Y$},
if the following holds:
A maxface $f$ has a singular point of type $X$ 
(resp.\ $Y$)
at a point $p$ if and only if 
the conjugate maxface $f^\sharp$ has a singular point of type $Y$
(resp.\ $X$) at $p$.
The following dualities (I)--(III) are shown in 
\cite{UY_maxface, FSUY, KimYang, FRUYY, FKKRSUYY, OT}:
\begin{itemize}
\item[(I)]
The cuspidal edge singularity is self-dual.
The duality between swallowtail singularity
and cuspidal cross cap singularity
(\cite{UY_maxface, FSUY}).
\item[(II)]
The duality between 
generalized conelike singularity and 
fold singularity
(\cite{KimYang, FRUYY, FKKRSUYY}).
\item[(III)]
The duality between cuspidal $S_1^-$ singularity and 
cuspidal butterfly singularity
(\cite{OT}).
\end{itemize}
Such a duality is also known for several classes of surfaces
\cite{Honda, Honda2, INS, IS, Takahashi, Yasumoto}.

\begin{figure}[htb]
\begin{center}
 \begin{tabular}{{c@{\hspace{8mm}}c}}
  \resizebox{4.2cm}{!}{\includegraphics{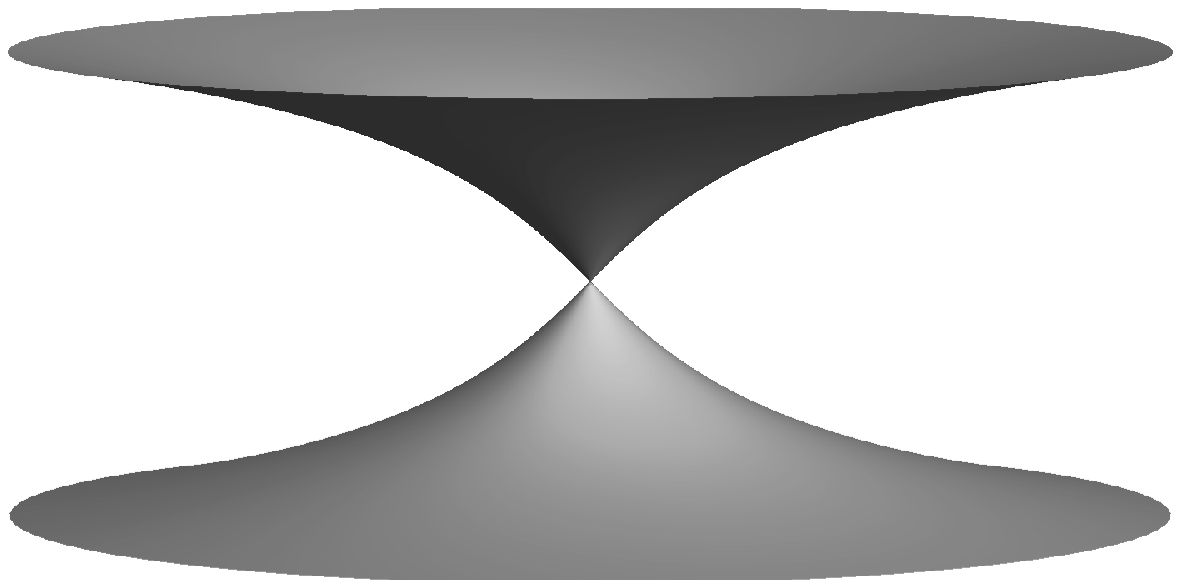}}&
  \resizebox{3.3cm}{!}{\includegraphics{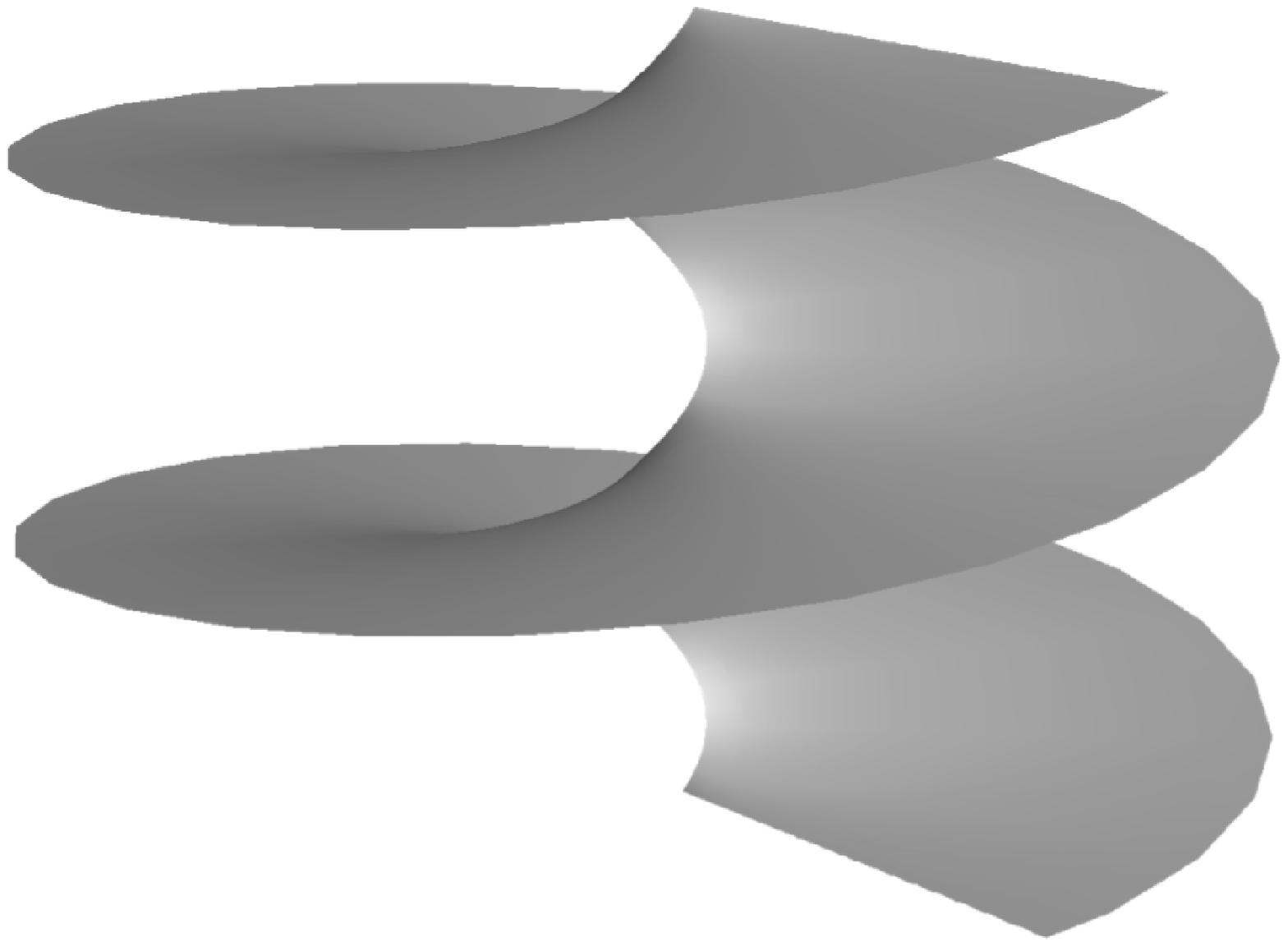}}
 \end{tabular}
\end{center}
\caption{Left: maximal catenoid, which is 
a maxface having generalized conelike singularities.
This maximal catenoid is given by the Weierstrass data 
$(g,\omega)=(e^z , e^{-z}dz)$.
Right: maximal helicoid, which is 
the conjugate of the maximal catenoid.
The maximal helicoid has fold singular points.
These maxfaces verify the duality (II).
}
\label{fig:maxfaces}
\end{figure}

We remark that the singularities appearing in the above 
dualities (I)--(III)
have several important properties.
With respect to (I),
Fujimori-Saji-Umehara-Yamada \cite{FSUY}
proved that
the singularities of maxfaces in $\L^3$ generically consist of 
cuspidal edges, swallowtails and cuspidal cross caps.
On the duality (II),
Kim-Yang \cite{KimYang} showed that,
if a maxface admits a generalized conelike singular point,
then the maxface has a point symmetry
(the reflection principle, \cite[Theorem 4.3]{KimYang}).
Moreover, it is known that, 
if a maxface admits a fold singular point,
then it can be extended to a timelike minimal surface analytically
(\cite{KKSY, FKKRSUYY}).
With respect to the duality (III),
we remark that
cuspidal edges $(A_2)$, swallowtails $(A_3)$ and cuspidal butterflies $(A_4)$
are in the series of the {\it $A_k$ singularities} 
(cf.\ \cite{SUY} and Section \ref{sec:AkSk}),
and that
the cuspidal cross cap is 
in the series of the {\it cuspidal $S_k$ singularities} 
(cf.\ \cite{S}).
It is also proved that
maxfaces in $\L^3$ do not admit
cuspidal $S_1^+$ singularities \cite{OT}
and cuspidal $S_k$ singularities for $k\geq2$ \cite{MNT}.

On the other hand,
let $S^3_1=S^3_1(1)$ be the {\it de Sitter $3$-space},
namely, 
$S^3_1$ is the complete simply-connected and connected 
Lorentzian $3$-manifold with constant
sectional curvature $1$.
There is a Lawson-type isometric correspondence 
(cf.\ \cite{Lawson})
between spacelike maximal surfaces in $\L^3$
and spacelike constant mean curvature $1$ 
(spacelike CMC $1$) surfaces in $S^3_1$,
where $S^3_1$ is the {\it de Sitter $3$-space}
of constant sectional curvature $1$.
It should be remarked that,
as in the relationship between
spacelike maximal surfaces in $\L^3$
and minimal surfaces in $\R^3$,
spacelike CMC $1$ surfaces in $S^3_1$
correspond to CMC $1$ surfaces in the hyperbolic $3$-space $H^3$
(cf.\ \cite{B, UY_complete}).
As Calabi's theorem \cite{C}, it holds that
{\it complete spacelike CMC $1$ surfaces in $S^3_1$ 
must be totally umbilic}
(Akutagawa \cite{Akutagawa}, Ramanathan \cite{Ramanathan}).
So, to discuss the non-trivial global properties,
we need to consider the class of spacelike CMC $1$ surfaces with singular points.
Fujimori \cite{Fujimori} introduced {\it CMC $1$ faces} in $S^3_1$,
and gave the representation formula 
(\cite[Theorem 1.9]{Fujimori}, \cite{AA}, cf.\ Fact \ref{fact:BW-representation}).
CMC $1$ faces in $S^3_1$ can be regarded 
as a corresponding class of maxfaces in $\L^3$.
In \cite{FSUY}, it is proved that
the singularities of CMC $1$ faces in $S^3_1$ generically consist of 
cuspidal edges, swallowtails and cuspidal cross caps.
Moreover, the duality of singularities 
for cuspidal edges, swallowtails and cuspidal cross caps
on CMC $1$ faces which is corresponding to (I)
in \cite{FSUY} (see also Fact \ref{fact:duality-generic-CMC1}).
So the following questions naturally arise:
{\it 
What kind of singularities appear on CMC $1$ faces in $S^3_1$?
Does the duality of singularities 
corresponding to {\rm (II)} or {\rm (III)}
also hold in the case of CMC $1$ faces in $S^3_1$?
}


\subsection{Results}
In this paper, 
we investigate the singularities of CMC $1$ faces 
in the de Sitter $3$-space $S^3_1$.
First, we prove the criterion for $5/2$-cuspidal edges
on CMC $1$ faces
in terms of Weierstrass data (Theorem \ref{thm:to-prove}),
which yields the following.

\begin{introtheorem}\label{thm:A}
CMC $1$ faces in $S^3_1$
have the duality between 
generalized conelike singularity
and $5/2$-cuspidal edge singularity.
\end{introtheorem}

We also show that 
CMC $1$ faces in $S^3_1$ do not admit fold singular points
(Corollary \ref{cor:CMC1-foldsing}).
Hence Theorem \ref{thm:A} can be regarded 
as the de Sitter counterpart of the duality (III).

\begin{figure}[htb]
\begin{center}
 \begin{tabular}{{c@{\hspace{8mm}}c}}
  \resizebox{4cm}{!}{\includegraphics{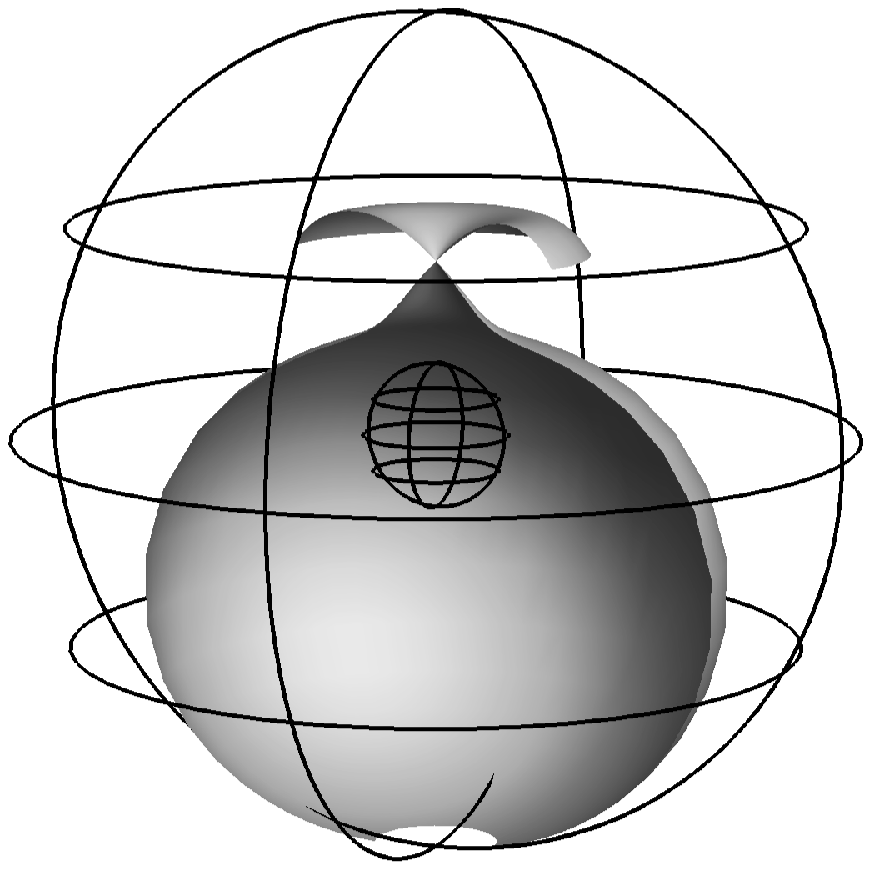}}&
  \resizebox{4cm}{!}{\includegraphics{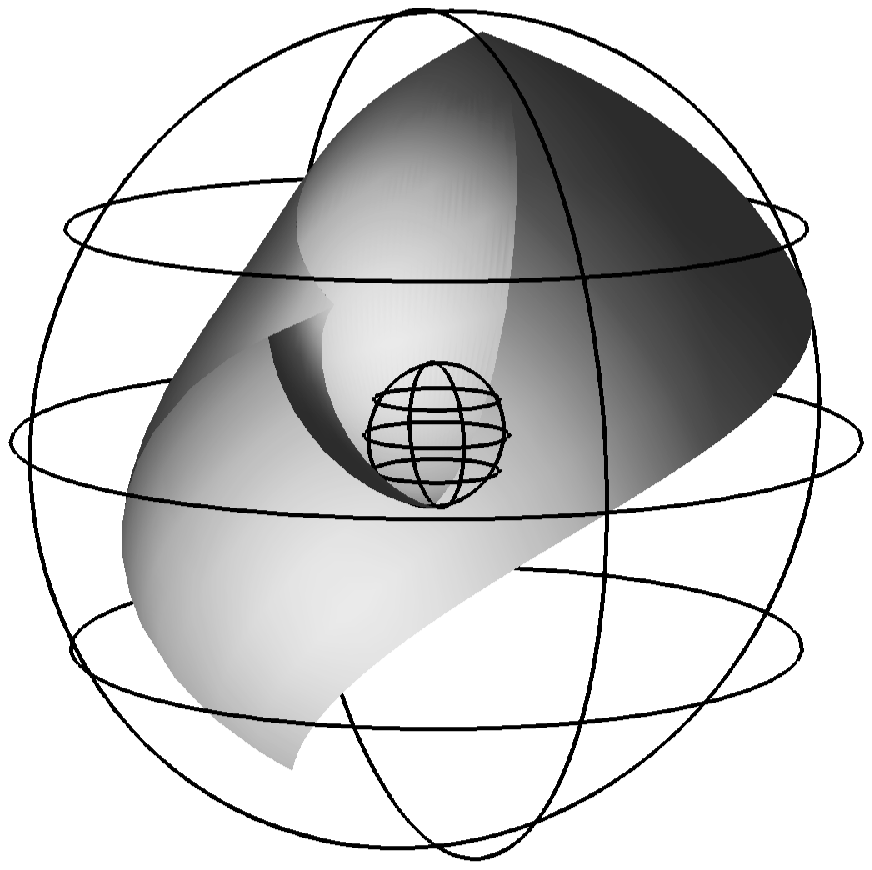}}
 \end{tabular}
\end{center}
\caption{Left: a CMC $1$ face $f:\C\to S^3_1$ 
with the Weierstrass data 
$(g,\omega)=(e^z , e^{-z}dz)$.
The singular set of $f$ 
consists of generalized conelike singular points.
Right: the conjugate CMC $1$ face $f^\sharp$.
In Example \ref{ex:catenoid-relative},
we see that $f^\sharp$ has $5/2$-cuspidal edge singular points.
These surfaces correspond to the maximal catenoid
and the maximal helicoid shown in Figure \ref{fig:maxfaces}, respectively.
To visualize the figures, we use the hollow ball model \cite{FNSSY}, cf.\ \eqref{eq:hollow}.}
\label{fig:catenoid-cousin}
\end{figure}

Next, we show the criteria 
for cuspidal $S_k$ singularities,
$A_4$-type and generalized $A_k$ singularities
on CMC $1$ faces
in terms of Weierstrass data
(Theorems \ref{thm:Ak-criteria-CMC1} and \ref{thm:Sk-criteria-CMC1}).
To derive the duality between cuspidal $S_k$ singularity
and $A_{k+3}$ singularity as in (III),
we introduce two invariants of CMC $1$ faces
$$
  \alpha=\alpha(f,p),\qquad
  \sigma=\sigma(f,p)
$$
(Definition \ref{def:invariants}).
These $\alpha$- and $\sigma$-invariants
are defined at singular points satisfying 
the conditions \eqref{eq:condition-A} and \eqref{eq:condition-S},
respectively (cf.\ Definition \ref{def:invariants}).
For example, for $k\geq1$,
the $A_{k+3}$ singularities satisfies the condition \eqref{eq:condition-A},
and the cuspidal $S_k$ singularities satisfies the condition \eqref{eq:condition-S}.
See Theorem \ref{thm:invariants} for details 
(cf.\ Remark \ref{rem:rhamphoid-invariant}).
Then we have the following.

\begin{introtheorem}\label{thm:B}
CMC $1$ faces in $S^3_1$ have
the following dualities.
\begin{itemize}
\item[{\rm (1)}]
The duality between 
admissible cuspidal $S_1$ singularity
and admissible cuspidal butterfly.
\item[{\rm (2)}] 
The duality between 
cuspidal $S_k$ singularity
and 
generalized $A_{k+3}$ singularity,
where $k$ is an integer satisfying $k\geq 2$.
\end{itemize}
\end{introtheorem}

While the assertion (1) of Theorem \ref{thm:B} seems to 
correspond to the duality (III) for maxfaces in $\L^3$, 
but there are differences:
\begin{itemize}
\item
The first differences is that we need to assume the {\it admissibility},
which is defined in terms of $\alpha$- and $\sigma$-invariants
(see Definition \ref{def:generic-A4}).
The reason is that, if a CMC $1$ face $f$ has 
a cuspidal $S_1$ singular point $p$,
then the $\sigma$-invariant $\sigma(f,p)$ is not equal to $12$
(Corollary \ref{cor:cS1-sigma-invariant}).
\item
The second difference is that
there are CMC $1$ faces which admit 
cuspidal $S_1^+$ singular points 
(Example \ref{ex:cuspidal-butterfly}, Figure \ref{fig:cS1-plus}),
while maxfaces in $\L^3$ do not admit them \cite{OT}.
\end{itemize}
Moreover, the assertion (2) of Theorem \ref{thm:B}
have no counterparts in the case of maxfaces in $\L^3$,
since maxfaces cannot have cuspidal $S_k$ singularities for $k\geq2$
\cite{MNT}.
See Definition \ref{def:generalized-Ak} for the definition of 
generalized $A_{k}$ singularities.

Finally, using our criteria
(Theorems \ref{thm:to-prove}, 
\ref{thm:Ak-criteria-CMC1} 
and \ref{thm:Sk-criteria-CMC1}),
we obtain the classification of non-degenerate singular points 
on CMC $1$ faces in $S^3_1$.

\begin{introtheorem}\label{thm:C}
A non-degenerate singular point on CMC $1$ faces in $S^3_1$
must be one of the followings{\rm :}
\begin{enumerate}
\item an $A_k$-type singular point $(k=2,3,4)$,
\item a generalized $A_k$ singular point $(k\geq5)$, 
\item a generalized conelike singular point,
\item a cuspidal $S_k$ singular point $(k\geq0)$, 
\item a singular point satisfying 
the condition \eqref{eq:condition-S} with $\sigma=12$, or 
\item a $5/2$-cuspidal edge singular point.
\end{enumerate}
\end{introtheorem}

We remark that 
singular points of type $A_2$ (resp.\ $A_3$, $A_4$)  
are cuspidal edges (resp.\ swallowtails, cuspidal butterflies),
and
cuspidal $S_0$ singular points
are cuspidal cross caps.
As a corollary,
non-existence results for several singularities on CMC $1$ faces 
are obtained
(cf.\ Corollaries \ref{cor:CMC1-foldsing} and \ref{cor:CMC1-odd-ce}).
Here, a singular point is called {\it non-degenerate}
if the exterior derivative of the signed area density function $\lambda$
(see \eqref{eq:signed-area-density})
does not vanish.
While the set of non-degenerate singular points
forms a regular curve in the source domain,
every degenerate singular point is isolated.
This is because degenerate singular points
occur at the branch point of the meromorphic function $g$
(cf.\ Fact \ref{fact:singular-FSUY}).
On the other hand, the images of degenerate lightlike points 
on zero mean curvature surfaces in $\L^3$
are lightlike line segments
(so-called the {\it line theorem}, 
\cite{Klyachin}, \cite{UY_lightlike1}, \cite{UY_lightlike2}).

This paper is organized as follows.
First, in Section \ref{sec:prelim},
we review the fundamental properties of 
CMC $1$ faces in $S^3_1$
and the criteria for singular points,
such as cuspidal edges, swallowtails and cuspidal cross caps
(Fact \ref{fact:singular-FSUY}). 
Then, in Section \ref{sec:invariants}, 
we consider the equivalence relation on Weierstrass data,
which yields the $\alpha$- and {\it $\sigma$-invariants}
(Definition \ref{def:invariants}).
In Section \ref{sec:rhamphoid},
we prove Theorem \ref{thm:A}.
For the proof, we use the criterion for $5/2$-cuspidal edges
(\cite[Theorem 4.1]{HKS}, cf.\ Fact \ref{fact:rhamphoid}).
Since the criterion is for frontals in $\R^3$,
we cannot apply it to CMC $1$ faces in $S^3_1$ directly.
Hence,
we use the orthogonal projection from $S^3_1$ to the tangent spaces,
see \eqref{eq:orthogonal-proj}.
Here we remark that a duality as in Theorem \ref{thm:A} is observed 
in the case of spacelike Delaunay surfaces
(i.e., generalized CMC surfaces of revolution)
in $\L^3$ \cite[Theorem 1.2]{HKS}.
In Section \ref{sec:AkSk},
we show criteria 
for $A_k$  and cuspidal $S_k$ singularities on CMC $1$ faces
in terms of Weierstrass data
(Theorems \ref{thm:Ak-criteria-CMC1} and \ref{thm:Sk-criteria-CMC1}),
which yield Theorem \ref{thm:B}.
Finally, in Section \ref{sec:final}, we prove Theorem \ref{thm:C}.
In the appendix, we give a proof of Lemma \ref{lem:T5p-body},
which is the key lemma for the proofs of our criteria.

\section{Singularities of CMC $1$ surfaces in de Sitter 3-space}
\label{sec:prelim}

\subsection{Structure of de Sitter space}
We denote by $\L^4$ the Lorentz-Minkowski $4$-space with 
the Lorentz metric 
$$
  \inner{\vect{x}}{\vect{y}}
  = -x_0y_0 + x_1 y_1 + x_2 y_2 + x_3 y_3,
$$
where 
$\vect{x}=(x_0,x_1,x_2,x_3)^T$,
$\vect{y}=(y_0,y_1,y_2,y_3)^T$,
and ${}^T$ means the transpose.
Then the de Sitter $3$-space is given by
\[
  S^3_1=S^3_1(1)=
  \{\vect{x}\in \L^4 \, ; \, \inner{\vect{x}}{\vect{x}}=1\}
\]
with metric induced from $\L^4$, which is 
a complete simply-connected and connected 
Lorentzian $3$-manifold with constant
sectional curvature $1$.
We identify 
the set of $2\times 2$ 
Hermitian matrices $\Herm(2)=\{X^*=X\}$
$(X^*:=\bar{X}^T)$ via
\begin{equation}\label{eq:Psi}
  \Phi:
  \Herm(2) \ni 
  \begin{pmatrix}
    x_0+x_3 & x_1+\sqrt{-1} x_2 \\
    x_1-\sqrt{-1} x_2 & x_0-x_3 
  \end{pmatrix}
  \longmapsto 
  (x_0,x_1,x_2,x_3)^T 
  \in \L^4
\end{equation}
with the metric
$$
  \inner{X}{Y}
  = -\frac1{2}\trace(X \, {\rm adj}(Y))\qquad
  \left(
  {\rm adj} :
  \begin{pmatrix}
    a & b \\
    c & d 
  \end{pmatrix}
  \mapsto
  \begin{pmatrix}
    d & -b \\
    -c & a 
  \end{pmatrix}
  \right).
$$
In particular, $\inner{X}{X}=-\det X$.
Under the identification \eqref{eq:Psi},
the de Sitter $3$-space $S^3_1$
is represented as
$$ 
   S^3_1=\{X\in \Herm(2)\,;\,\det X=-1\}.
$$
Moreover,
under the identification \eqref{eq:Psi},
the basis 
\begin{equation}\label{eq:pauli}
   e_0=\begin{pmatrix}1&0\\0&1\end{pmatrix},
  \quad
   e_1=\begin{pmatrix}0&1\\1&0\end{pmatrix},
  \quad
   e_2=\begin{pmatrix}0&\sqrt{-1}\\-\sqrt{-1}&0\end{pmatrix},
  \quad
   e_3=\begin{pmatrix}1&0\\0&-1\end{pmatrix}
\end{equation}
of $\Herm(2)$
corresponds to the canonical orthonormal basis
$$
  (\varepsilon_0,\varepsilon_1,\varepsilon_2,\varepsilon_3)
  =(\Phi e_0,\Phi e_1,\Phi e_2,\Phi e_3)
$$
of $\L^4$.
The special linear group $\SL(2,\C)$ acts 
isometrically and transitively on $\Herm(2)$ by 
\begin{equation}\label{eq:action}
  \Herm(2) \ni X \longmapsto AXA^* \in \Herm(2),
\end{equation}
where $A\in \SL(2,\C)$.
Since the isotropy subgroup of $\SL(2,\C)$
at $ e_3\in S^3_1$ is $\SU(1,1)$,
we can identify
$$
  S^3_1
  =\SL(2,\C)/\SU(1,1)
  =\{A e_3A^*\,;\,A\in\SL(2,\C)\}.
$$
Since the action \eqref{eq:action}
preserves the metric, orientation, and the time-orientation
of $\Herm(2)$,
for each $A\in \SL(2,\C)$,
there exists a unique element $\check{A}$ 
of the restricted Lorentz group $\SO^+(1,3)$ such that
\begin{equation}\label{eq:double-cover}
  \Phi (AXA^*)= \check{A}\,\Phi (X)\qquad
  (X\in \Herm(2))
\end{equation}
holds,
where $\Phi$ is the identification \eqref{eq:Psi}.
This map $\SL(2,\C)\ni A \mapsto \check{A}\in \SO^+(1,3)$
gives the universal covering of $\SO^+(1,3)$.

To visualize the graphics of surfaces in $S^3_1$,
we use the {\it hollow ball model} of $S^3_1$
introduced in \cite{FNSSY}.
We set 
\begin{equation}\label{eq:hollow}
  \mathcal{H}:=\{ \vect{y}\in \R^3\,;\, \sqrt{2}-1 < ||\vect{y}|| <\sqrt{2}+1 \}
  \quad
  \left( ||\vect{y}||=\sqrt{(y_1)^2+(y_2)^2+(y_3)^2} \right),
\end{equation}
where $\vect{y}=(y_1,y_2,y_3)^T$.
We identify $S^3_1$ and $\mathcal{H}$
via the map:
$$
  S^3_1\ni 
  (x_0,x_1,x_2,x_3)^T
  \longmapsto
  \frac1{x_0+\sqrt{(x_0)^2+(x_1)^2+(x_2)^2+(x_3)^2}}
  (x_1,x_2,x_3)^T\in \mathcal{H}.
$$
The ideal boundary 
$\partial\mathcal{H}$ of $S^3_1$
consists of two components
$\partial\mathcal{H}=\partial \mathcal{H}_+\cup \partial \mathcal{H}_-$,
where we set
$\partial \mathcal{H}_{\pm}:=\{ \vect{y}\in \R^3\,;\, ||\vect{y}|| =\sqrt{2}\pm1\}$.

\subsection{CMC $1$ faces}

Let $D$ be a Riemann surface.
A holomorphic map $F:D\to \SL(2,\C)$ is called {\it null\/}
if $\det dF=0$ holds on $D$.
The projection (cf.\ \eqref{eq:cmc1-face}) 
of null holomorphic immersions 
gives spacelike constant mean curvature one (CMC $1$, for short) 
surfaces with singularities, 
called {\it CMC $1$ faces} in $S^3_1$ 
(see \cite{Fujimori} for details).

\begin{fact}[\cite{Fujimori}]
\label{fact:BW-representation}
Let $D$ be a simply connected domain of $\C$,
and $(g,\omega)$ be a pair of 
meromorphic function $g$,
and holomorphic $1$-form $\omega$ on $D$
such that $(1+|g|^2)^2|\omega|^2$ is a Riemannian metric on $D$
and $(1-|g|^2)^2$ does not vanish identically.
We fix a point $p\in D$.
Take a holomorphic null immersion $F:D\to \SL(2,\C)$ satisfying 
\begin{equation}\label{eq:F^-1dF}
  F^{-1}dF=
   \begin{pmatrix}
    g & -g^2 \\ 1 & -g\hphantom{^2}
   \end{pmatrix} \omega,
   \qquad
   F(p)= e_0.
\end{equation}
Then 
\begin{equation}\label{eq:cmc1-face}
   f:=F e_3F^*
\end{equation}
defines a CMC $1$ face $f:D\to S^3_1$.
Moreover, any CMC $1$ face is locally obtained in this manner.
\end{fact}

We call $F$ the {\it holomorphic null lift} of $f$.
We remark that $FF^*$ gives a conformal CMC $1$ immersion into 
the hyperbolic $3$-space $H^3$.
Moreover, conformal CMC $1$ immersions 
are always given locally in such a manner
(see \cite{B} and \cite{UY_complete} for details).

The pair $(g,\omega)$ is called 
the {\it Weierstrass data\/} of $f$.
It is known that 
$F$ and $(g,\omega)$
are not uniquely determined from $f$
(\cite{UY_complete}, \cite{Fujimori}, 
cf.\ Section \ref{sec:invariants}).
On the other hand, the {\it hyperbolic Gauss map} $G:D\to \hat{\C}$
$(\hat{\C}:=\C\cup\{\infty\})$ defined by 
$$
  G=dF_{11}/dF_{21}=dF_{12}/dF_{22}
$$
is uniquely determined from $f$,
where $F=(F_{jl})_{jl=1,2}$.
We denote by $Q=\omega dg$ the Hopf differential of $f$.
By \cite[Equation (2.6)]{UY_complete},
it holds that
\begin{equation}\label{eq:QgG}
2Q=S(g)-S(G),
\end{equation}
where $S(g)$ is the Schwarzian differential of $g$
defined by
$$
  S(g)=\hat{S}(g)dz^2 \qquad
  \left(
  \hat{S}(g)=
  \left( \frac{g_{zz}}{g_z} \right)_z
  -\frac{1}{2}\left( \frac{g_{zz}}{g_z} \right)^2
  \right).
$$
Here, $z$ is a local complex coordinate of $D$,
and $g_z=dg/dz$.

Small \cite{Small} gave the following expression:
\begin{equation}
\label{eq:Small}
F=\begin{pmatrix}
\vspace{2mm}
  G\dfrac{da}{dG}-a&G\dfrac{db}{dG}-b \\
\vspace{2mm}
  \dfrac{da}{dG} & \dfrac{db}{dG}
  \end{pmatrix}\quad
\left(
  a=\sqrt{\dfrac{dG}{dg}},\quad
  b=-ga
\right),
\end{equation}
which is called {\it Small's formula}.
See \cite{KUY} for an alternative proof.

\subsection{Generic singularities of CMC $1$ faces}
Let $D$ be a $2$-manifold,
and $f:D\to M$ be a map into 
a $3$-manifold $M=M^3$.
A point $p\in D$ is said to be a {\it singular point}
if $f$ is not an immersion at $p$.
We denote by $\Sigma(f)\,(\subset D)$ 
the set of singular points of $f$.

For $j=1,2$,
let $f_j:D_j\to M_j$ be maps
having singular points $p_j\in D_j$.
Then, 
the map germ $f_1:(D_1,p_1)\to (M_1,f(p_1))$ 
is {\it $\mathcal{A}$-equivalent} 
to $f_2:(D_2,p_2)\to (M_2,f(p_2))$
if there exist diffeomorphism germs
$\phi : (D_1,p_1)\to (D_2,p_2)$,
$\Phi : (M_1,f(p_1))\to (M_2,f(p_2))$
such that $\Phi\circ f_1=f_2\circ \phi$.

A map $f:D\to M$ 
is said to have a {\it cuspidal edge singular point} at $p$
if the map germ $f$ at $p$
is $\mathcal{A}$-equivalent to 
$\ce$ at the origin,
where 
$\ce: \R^2\to \R^3$
is given by $\ce(u,v) = (u,v^2,v^3)$.
Similarly, 
if the map germ $f$ at $p$
is $\mathcal{A}$-equivalent to 
$\sw$ 
(resp.\ $\ccr$)
at the origin,
$f$ is said to have a {\it swallowtail}
(resp.\ {\it cuspidal cross cap}) singular point at $p$,
where 
${\sw}(u,v) = (u, - 4 v^3 -2 u v, 3 v^4+u v^2)$,
${\ccr}(u,v) = (u, v^2, uv^3)$.

\begin{figure}[htb]
\begin{center}
 \begin{tabular}{{c@{\hspace{10mm}}c@{\hspace{8mm}}c}}
  \resizebox{2.4cm}{!}{\includegraphics{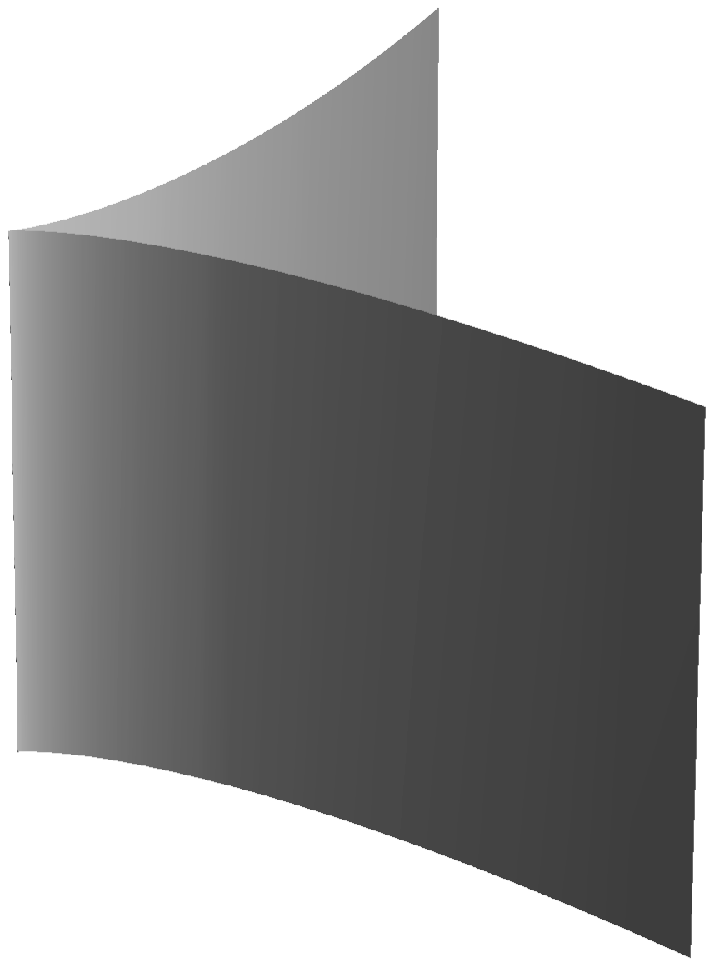}}&
  \resizebox{2.8cm}{!}{\includegraphics{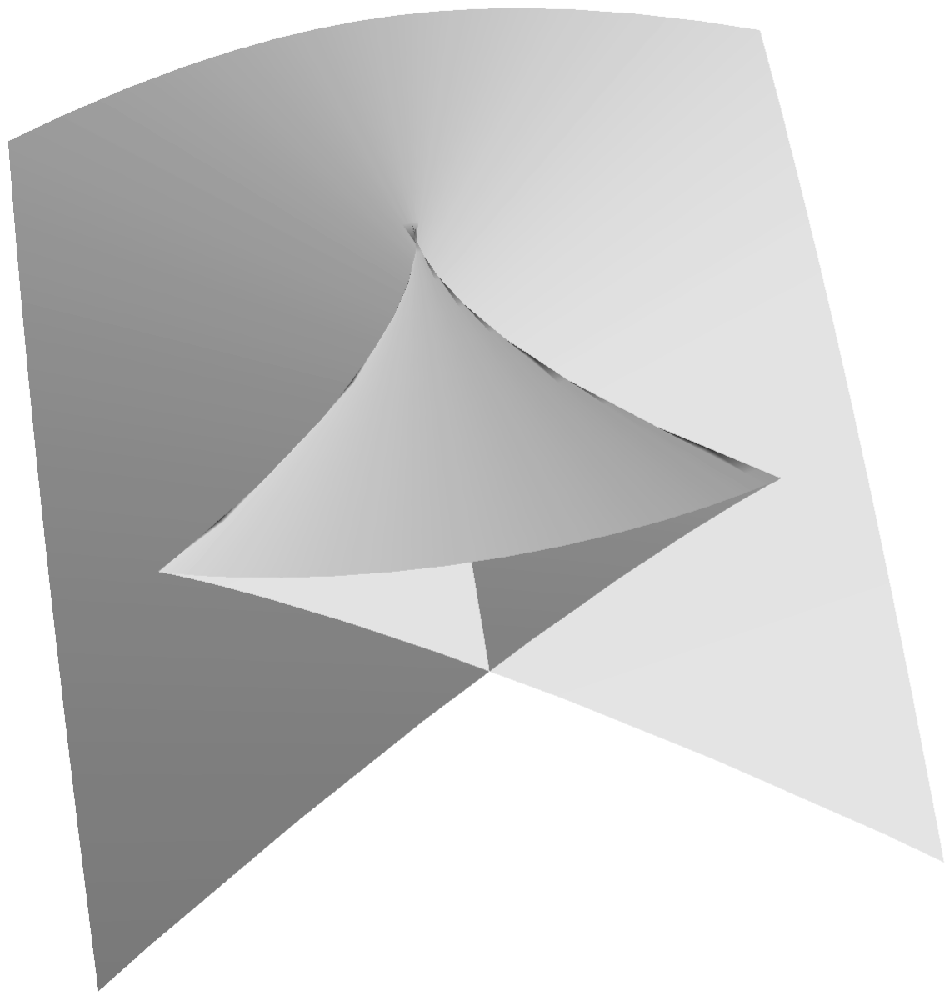}}&
  \resizebox{3.5cm}{!}{\includegraphics{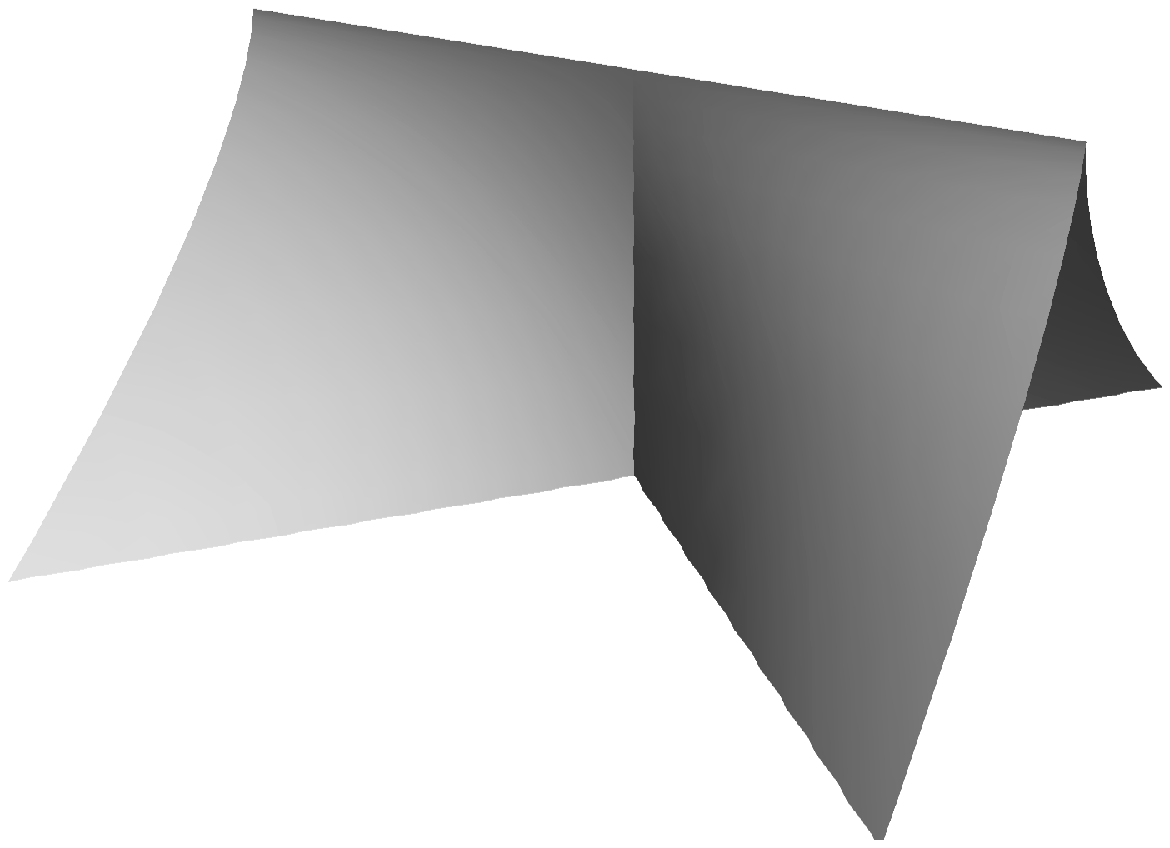}}
 \end{tabular}
\end{center}
\caption{Standard cuspidal edge $\ce(u,v)$,
standard swallowtail $\sw(u,v)$,
standard cuspidal cross cap $\ccr(u,v)$.}
\end{figure}

Such the singular points are shown to be generic singular points
of CMC $1$ faces in $S^3_1$ \cite{FSUY}.
For the proof, criteria for 
cuspidal edge, swallowtail, and cuspidal cross cap singularities
on CMC $1$ faces were shown in \cite{FSUY}.
Here we review them briefly.

We assume that $M$ is equipped with 
a Riemannian metric $d\sigma^2$.
A map $f:D\to M$ 
is said to be a {\it frontal} if,
for each point $p\in D$,
there exist an open neighborhood $(D;u,v)$ of $p$
and a unit vector field $L$ of $M$ along $f$
such that 
$$
  d\sigma^2(\nu,df(W))=0\qquad
  (W\in TD)
$$
on $D$,
where we set $L=(f,\nu)$.
If we can choose $L$ to be an immersion, 
then $f$ is called a {\it wave front}
(or a {\it front}, for short).
For a front (resp.\ frontal) $f$
and a diffeomorphism $\Phi$ of $M$,
the composition $\Phi\circ f$
is also a front (resp.\ frontal),
and hence 
the notions of frontals or fronts 
are independent of the choice
of the Riemannian metric $d\sigma^2$ on $M$.
See \cite{SUY_annals} for details.

Let $f:D\to M$ be a frontal.
We denote by $d\mu$ be the Riemannian volume form
of $(M,d\sigma^2)$.
We set a $C^{\infty}$ function $\lambda$ as
\begin{equation}\label{eq:signed-area-density}
  \lambda=d\mu(f_u,f_v,\nu),
\end{equation}
which is called the {\it signed area density function}.
Then $p\in D$ is a singular point if and only if
$\lambda(p)=0$.
If $d\lambda(p)\ne0$,
then $p$ is said to be a {\it non-degenerate singular point}.
If $p$ is a non-degenerate singular point,
the implicit function theorem yields that 
the singular set $\Sigma(f)$
is a $1$-dimensional submanifold of $D$ near $p$.
That is,
there exist $\varepsilon>0$
and a regular curve 
$\gamma:(-\varepsilon,\varepsilon)\to D$
such that 
the image $\gamma((-\varepsilon,\varepsilon))$
is a subset of $\Sigma(f)$ and $\gamma(0)=p$.
Such the curve $\gamma(t)$ is called a {\it singular curve} at $p$.
A non-vanishing vector field $\xi$ defined on a neighborhood of $p$
is said to be a {\it singular directional vector field}, 
if $\xi_{\gamma(t)}$ is parallel to the tangent vector field $\gamma' (t)$. 
On the other hand,
a non-vanishing vector field $\eta$ defined on a neighborhood of $p$
is said to be a {\it null vector field}, 
if $df_q(\eta)=\vect{0}$ for each $q\in \Sigma(f)$.
The restriction $\eta(t):=\eta_{\gamma(t)}$
is called a {\it null vector field along $\gamma(t)$}.
A non-degenerate singular point $p\in \Sigma(f)$ 
is said to be {\it of the first kind}
if $\xi_p$ and $\eta_p$ are linearly independent.
If a front $f:D\to M$ has a singular point of the first kind $p\in D$,
then $f$ has cuspidal edge at $p$ (\cite{KRSUY}).

\begin{definition}\label{def:ch-pair}
Let $f:D\to S^3_1$ be a CMC $1$ face
with Weierstrass data $(g,\omega)$,
where $D$ is a domain of $\C$.
We set a function $\varphi$
and a vector field $\vf$ as
\begin{equation}\label{eq:varphi}
  \varphi:=\frac{dg}{g^2\omega},
  \qquad
  \vf:= g\,\frac{d}{dg},
\end{equation}
respectively.
We call $(\varphi, \vf)$ the {\it characteristic pair} associated with 
the Weierstrass data $(g,\omega)$.
\end{definition}

The characteristic pair $(\varphi, \vf)$ plays an important role
in the criteria for singularities:

\begin{fact}[{\cite{FSUY}}]
\label{fact:singular-FSUY}
Let $f:D\to S^3_1$ be a CMC $1$ face
with Weierstrass data $(g,\omega)$,
and let $(\varphi, \vf)$ be the characteristic pair 
associated with $(g,\omega)$.
Then, the singular set is given by 
$\Sigma(f)=\{p\in D\,;\, |g(p)|=1\}$.
The vector fields $\xi$, $\eta$
defined by
\begin{equation}\label{eq:xi-eta}
  \xi:=\sqrt{-1}\,\overline{\left( \frac{g_z}{g}\right)} \frac{\partial}{\partial z}
  - \sqrt{-1}\,\left( \frac{g_z}{g}\right)\frac{\partial}{\partial \bar{z}},\quad
  \eta
  :=\frac{\sqrt{-1}}{g\homega}\frac{\partial}{\partial z}
    -\frac{\sqrt{-1}}{\bar{g}\bar{\homega}} \frac{\partial}{\partial \bar{z}}
\end{equation}
give a singular directional vector field and a null vector field
along $f$, respectively.
Here $z$ is a complex coordinate of $D$,
and $\omega=\homega\,dz$.
Moreover, 
\begin{enumerate}
\item
a CMC $1$ face $f$ is a frontal.
\item
a singular point $p \in \Sigma(f)$
is non-degenerate
if and only if $dg(p)\neq0$.
\item
$f$ is a front on a neighborhood of 
$p\in \Sigma(f)$ if and only if
$\Re \varphi(p)\neq 0$.
In particular, $p$ is a non-degenerate singular point.
\item
$p \in \Sigma(f)$ is a singular point of the first kind 
if and only if $\Im \varphi(p)\neq 0$.
\item
$f$ has cuspidal edge at $p\in \Sigma(f)$ 
if and only if 
$\Re \varphi(p)\neq 0$,
$\Im \varphi(p)\neq 0$.
\item
$f$ has swallowtail at $p\in \Sigma(f)$ 
if and only if 
$\varphi(p)\in \R\setminus \{0\}$,
$\Re \vf \varphi(p) \neq 0$.
\item
$f$ has cuspidal cross cap at $p\in \Sigma(f)$ 
if and only if 
$\varphi(p)\in \sqrt{-1}\R\setminus \{0\}$,
$\Im \vf \varphi(p) \neq 0$.
\end{enumerate}
\end{fact}

Let $D$ be a simply connected Riemann surface
and $f:D\to S^3_1$ be a CMC $1$ face
with Weierstrass data $(g,\omega)$.
Then a CMC $1$ face $f^\sharp:D\to S^3_1$
given by the Weierstrass data $(g,-\sqrt{-1}\omega)$
is called the {\it conjugate} CMC $1$ face of $f$.

We say that CMC $1$ faces have the 
{\it duality between singular points of type $X$ and type $Y$},
if the following holds:
A CMC $1$ face $f$ has a singular point of type $X$ 
(resp.\ $Y$)
at a point $p$ if and only if 
the conjugate CMC $1$ face $f^\sharp$ has a singular point of type $Y$
(resp.\ $X$) at $p$.
As in the case of maxfaces in $\L^3$ (\cite{UY_maxface, FSUY}, cf.\ the duality (I) in the introduction),
the following holds.

\begin{fact}[{\cite{FSUY}}]\label{fact:duality-generic-CMC1}
The cuspidal edge singularity is self-dual
on CMC $1$ faces in $S^3_1$.
Moreover, CMC $1$ faces in $S^3_1$ have the duality 
between swallowtail singularity
and cuspidal cross cap singularity.
More precisely,
let $f:D\to S^3_1$ be a CMC $1$ face
defined on a simply connected domain $D$,
and $p\in D$ be a singular point.
Then
$f$ has cuspidal edge 
{\rm (}resp.\ swallowtail, cuspidal cross cap{\rm )}
at $p$
if and only if 
$f^\sharp$ has cuspidal edge 
{\rm (}resp.\ cuspidal cross cap, swallowtail{\rm )}
at $p$.
\end{fact}

\section{Invariants via criteria}
\label{sec:invariants}

It is known that 
a holomorphic null lift $F$ and
Weierstrass data $(g,\omega)$
are not uniquely determined from a CMC $1$ face $f$
(\cite{UY_complete}, \cite{Fujimori}).
In this section, by calculating quantities related to the characteristic pair $(\varphi, \vf)$,
we introduce two invariants $\alpha$, $\sigma$ of CMC $1$ faces at their singular points
(cf.\ Definition \ref{def:invariants}).

Let $f:D\to S^3_1$ be a CMC $1$ face,
and let $F : D\to \SL(2,\C)$ be a holomorphic null lift of $f$
with Weierstrass data $(g,\omega)$.
For a constant matrix $B\in \SU(1,1)$,
the map $\tilde{F}:D\to \SL(2,\C)$
defined by $\tilde{F}:=FB$
is also a holomorphic null lift of $f$.
The Weierstrass data $(\tilde{g},\tilde{\omega})$
associated with $\tilde{F}$ is described as
\begin{equation}\label{eq:W-related}
  \tilde{g}
  =\frac{ag+b}{\bar{b}g+\bar{a}},
  \qquad
  \tilde{\omega}
  =(\bar{b}g+\bar{a})^2\omega,
\end{equation}
where we set
\begin{equation}\label{eq:matB}
B=\begin{pmatrix}
  \bar{a}&-b \\ -\bar{b} & a
  \end{pmatrix} \qquad \left( |a|^2-|b|^2=1 \right).
\end{equation}
Two Weierstrass data $(g,\omega)$
and $(\tilde{g},\tilde{\omega})$
are said to be {\it equivalent} if they satisfy \eqref{eq:W-related}
(\cite{UY_complete}, \cite{Fujimori}).
Let $f_1, f_2:D\to S^3_1$ be two CMC $1$ faces.
We say that {\it $f_1$ is congruent to $f_2$}
if there exists a constant matrix $A\in \SL(2,\C)$
such that
$
  f_2 = Af_1A^*.
$
If $F$ is a holomorphic null lift of $f_1$,
then $AF$ is that of $f_2$.
Since 
$$
(AF)^{-1}d(AF)=F^{-1}dF,
$$
the Weierstrass data associated with $F$
coincides with that of $AF$.
Hence, the equivalence class of $(g,\omega)$
corresponds to the congruence class of the CMC $1$ face $f$
(cf.\ \cite{UY_complete}).

Let $f:D\to S^3_1$ be a CMC $1$ face with Weierstrass data $(g,\omega)$,
and let $(\varphi, \vf)$ be the characteristic pair associated with $(g,\omega)$.
If $(\tilde{g},\tilde{\omega})$ is a Weierstrass data 
which is equivalent to $(g,\omega)$,
then the characteristic pair $(\tilde{\varphi},\tilde{\vf})$
associated with $(\tilde{g},\tilde{\omega})$
is defined by
\begin{equation}\label{eq:varphi-tilde}
  \tilde{\varphi}
  =\dfrac{d\tilde{g}}{\tilde{g}^2\tilde{\omega}},
  \qquad
  \tilde{\vf}= \tilde{g}\,\frac{d}{d\tilde{g}}.
\end{equation}

\begin{lemma}\label{lem:ch-tilde}
Suppose that 
a Weierstrass data $(\tilde{g},\tilde{\omega})$ 
is equivalent to $(g,\omega)$ as \eqref{eq:W-related}.
Then the characteristic pair
$(\tilde{\varphi},\tilde{\vf})$
associated with $(\tilde{g},\tilde{\omega})$
satisfies 
\begin{equation}\label{eq:varphi-delta2}
  \tilde{\varphi}
  =\dfrac{1}{\Delta^2}\varphi,
  \quad
  \tilde{\vf}=\Delta\,\vf,
  \quad
  \text{where we set}
  \quad
  \Delta
  :=(ag+b) \left(\frac{\bar{a}}{g}+\bar{b}\right).
\end{equation}
\end{lemma}

\begin{proof}
By \eqref{eq:W-related},
we have
$\tilde{g}_z=g_z/(\bar{b}g+\bar{a})^2$,
which yields \eqref{eq:varphi-delta2}.
\end{proof}

We have the following.

\begin{lemma}\label{lem:invariants}
Fix a non-negative integer $m\in \Z$.
Let $f:D\to S^3_1$ be a CMC $1$ face with Weierstrass data $(g,\omega)$.
We take a non-degenerate singular point $p\in \Sigma(f)$.
Suppose that a Weierstrass data $(\tilde{g},\tilde{\omega})$ 
is equivalent to $(g,\omega)$ as \eqref{eq:W-related}.
Let $(\varphi, \vf)$ and $(\tilde{\varphi}, \tilde{\vf})$
be the characteristic pairs 
associated with $(g,\omega)$ and $(\tilde{g},\tilde{\omega})$,
respectively.
Then, there exist real numbers 
$c_0=1$, $c_1,\dots,c_m\in \R$
such that
$$
  |a\,g(p)+b|^{2(m-2)} \vf^m \varphi(p)
  = \sum_{k=0}^m \sqrt{-1}^k c_k \tilde{\vf}^{m-k} \tilde{\varphi}(p).
$$
\end{lemma}

\begin{proof}
By the definition of $\varphi$ in \eqref{eq:varphi},
we have $g_z=g^2h\varphi$,
where $\omega=h\,dz$.
For a positive integer $\ell$, 
we have
\begin{equation}\label{eq:D-ell}
  (\vf^{\ell-1} \varphi )_z = g\homega \varphi\,\vf^\ell \varphi,
\end{equation}
since
$$
  \vf^\ell \varphi 
  = g\,\frac{d}{dg}\vf^{\ell-1} \varphi
  = \frac{g}{g_z}(\vf^{\ell-1} \varphi )_z
  = \frac{1}{g\homega \varphi}(\vf^{\ell-1} \varphi )_z.
$$
Similarly, by \eqref{eq:varphi-delta2} and
\eqref{eq:W-related}, we have
\begin{equation}\label{eq:tilde-D-ell}
  (\tilde{\vf}^{\ell-1} \tilde{\varphi} )_z 
  = \tilde{g}\tilde{\homega} \tilde{\varphi}\,\tilde{\vf}^\ell \tilde{\varphi}
  = \frac{g\homega \varphi}{\Delta}\,\tilde{\vf}^\ell \tilde{\varphi},
\end{equation}
where $\Delta$ is the function defined in Lemma \ref{lem:ch-tilde}.
We set $\rho$, $\tau$ as
$$
  \rho
  =a\bar{b}g-\frac{\bar{a}b}{g},\quad
  \tau
  =a\bar{b}g+\frac{\bar{a}b}{g}.
$$
Then, it holds that
\begin{equation}\label{eq:delta-z}
  \Delta_z
  =\tau_z
  =g\homega \varphi \rho,\quad
  \rho_z
  =g\homega \varphi \tau.
\end{equation}

{\allowdisplaybreaks
In the following, we prove, by induction, that there exist $C_{q,r,s}^{m,k}\in \R$ such that
\begin{gather}
\label{eq:tilde-Dm}
  \Delta^{m-2} \vf^m \varphi
  = \sum_{k=0}^m P(m;k) \tilde{\vf}^{m-k} \tilde{\varphi},\\
 \nonumber
  P(m;k) 
  =\left\{
  \begin{array}{cl}
  1 \vphantom{\dfrac1{6}} & (\text{if} ~k=0),\\
  {\displaystyle \sum_{\substack{q+2r+s=k \\ q,r,s\geq0}}
  C_{q,r,s}^{m,k} \Delta^q \rho^{2r} \tau^s}
  \vphantom{\dfrac1{6}} & (\text{if} ~k~ \text{is even and positive}),\\
  {\displaystyle \sum_{\substack{q+2r+s=k-1 \\ q,r,s\geq0}}
  C_{q,r,s}^{m,k} \Delta^q \rho^{2r+1} \tau^s}
  \vphantom{\dfrac1{6}} & (\text{if} ~k~ \text{is odd}).
  \end{array}
  \right. 
\end{gather}
In the case} of $m=0$, Lemma \ref{lem:ch-tilde} yields \eqref{eq:tilde-Dm}. 

For a non-negative integer $m$, 
suppose that \eqref{eq:tilde-Dm} holds.
Taking the differentials of the both sides of \eqref{eq:tilde-Dm},
we have
$$
  (\Delta^{m-2})_z \vf^m \varphi + \Delta^{m-2} (\vf^m \varphi)_z
  = \sum_{k=0}^m 
     \left\{ P(m;k)_z \tilde{\vf}^{m-k} \tilde{\varphi}
    + P(m;k) (\tilde{\vf}^{m-k} \tilde{\varphi})_z \right\}.
$$
By
\eqref{eq:D-ell},
\eqref{eq:tilde-D-ell},
\eqref{eq:delta-z},
we have
\begin{equation*}
  (m-2)\rho \Delta^{m-2} \vf^m \varphi 
  + \Delta^{m-1} \vf^{m+1} \varphi
  = \sum_{k=0}^m 
     \left\{ \frac{\Delta \,P(m;k)_z }{g\homega \varphi} \tilde{\vf}^{m-k} \tilde{\varphi}
    +P(m;k) \tilde{\vf}^{m-k+1} \tilde{\varphi} \right\}.
\end{equation*}
Then \eqref{eq:tilde-Dm} yields that
\begin{equation*}
  \Delta^{m-1} \vf^{m+1} \varphi\\
  = \sum_{k=0}^m 
     \left[\left\{ -(m-2)\rho \, P(m;k)  +\frac{\Delta \,P(m;k)_z }{g\homega \varphi} \right\}
     \tilde{\vf}^{m-k} \tilde{\varphi}
    +P(m;k) \tilde{\vf}^{m-k+1} \tilde{\varphi} \right].
\end{equation*}
By $P(m;0)=1$,
\begin{multline*}
  \Delta^{m-1} \vf^{m+1} \varphi
  = \tilde{\vf}^{m+1} \tilde{\varphi} 
    + \left\{ P(m;1) - (m-2) \rho\right\}\tilde{\vf}^{m} \tilde{\varphi}\\
  + \sum_{k=2}^m 
     \left\{ -(m-2)\rho \, P(m;k-1)  
     +\frac{\Delta \,P(m;k-1)_z }{g\homega \varphi} 
    +P(m;k) \right\}\tilde{\vf}^{m-k+1} \tilde{\varphi}\\ 
    +\left\{ -(m-2)\rho \, P(m;m)  
     +\frac{\Delta \,P(m;m)_z }{g\homega \varphi} \right\}\tilde{\varphi}
\end{multline*}
holds.
Hence, we have
\begin{gather}
\label{eq:tilde-Dm+1}
  \Delta^{m-1} \vf^{m+1} \varphi
  = \sum_{k=0}^{m+1} P(m+1;k) \tilde{\vf}^{m+1-k} \tilde{\varphi},\\
 \nonumber
  P(m+1;k) 
  =\left\{
  \begin{array}{cl}
  1 \vphantom{\dfrac1{6}} & (\text{if} ~k=0),\\
  \vspace{1mm}
  P(m;1) - (m-2) \rho \vphantom{\dfrac1{6}} & (\text{if} ~k=1),\\
  \vspace{1mm}
  {\displaystyle -(m-2)\rho \, P(m;k-1)  
     +\frac{\Delta \,P(m;k-1)_z }{g\homega \varphi} 
    +P(m;k)}
  \vphantom{\dfrac1{6}} & (\text{if} ~k=2,\dots,m),\\
  {\displaystyle -(m-2)\rho \, P(m;m)  
     +\frac{\Delta \,P(m;m)_z }{g\homega \varphi}}
  \vphantom{\dfrac1{6}} & (\text{if} ~k=m+1)
  \end{array}
  \right. 
\end{gather}

\subsubsection*{$\bullet$ The case of $k=1$}
Since $P(m+1;1)=P(m;1) - (m-2) \rho$,
we have
$$
  P(m+1;1)=-\frac1{2}(m+1)(m-4)\rho.
$$

\subsubsection*{$\bullet$ The case that $k$ is even and $2\leq k\leq m$}
Since $k-1$ is odd,
$P(m;k-1)$ is expressed as
$$
  P(m;k-1)
  = \sum_{\substack{q+2r+s=k-2 \\ q,r,s\geq0}}
  C_{q,r,s}^{m,k-1} \Delta^q \rho^{2r+1} \tau^s.
$$
Differentiating the both sides, we have
\begin{align*}
  P(m;k-1)_z
  = g\homega \varphi 
  \sum_{\substack{q+2r+s=k-2 \\ q,r,s\geq0}}
  C_{q,r,s}^{m,k-1} 
  \left\{
    \begin{array}{ll}
    \vspace{2mm}
  q\Delta^{q-1} \rho^{2r+2} \tau^s
  +  (2r+1) \Delta^q\rho^{2r} \tau^{s+1}\\
  +s\Delta^q \rho^{2r+2} \tau^{s-1}
    \end{array}
  \right\}.
\end{align*}
Hence, we obtain
\begin{align*}
  P(m+1;k)
  &=\sum_{\substack{q+2r+s=k \\ q,r,s\geq0}}
  C_{q,r,s}^{m,k} \Delta^q \rho^{2r} \tau^s \\
  &\hspace{4mm}
  +\sum_{\substack{q+2r+s=k-2 \\ q,r,s\geq0}}
  C_{q,r,s}^{m,k-1} 
  \left\{
    \begin{array}{ll}
    \vspace{2mm}
    (q-m+2)\Delta^{q} \rho^{2r+2} \tau^s\\
    +  (2r+1) \Delta^{q+1}\rho^{2r} \tau^{s+1}
    +s\Delta^{q+1} \rho^{2r+2} \tau^{s-1}
    \end{array}
  \right\}.
\end{align*}

\subsubsection*{$\bullet$ The case that $k$ is even and $k=m+1$}
\begin{align*}
  P(m+1;m+1)
  &= \sum_{\substack{q+2r+s=m-1 \\ q,r,s\geq0}}
  C_{q,r,s}^{m,m}\left\{
    \begin{array}{ll}
    \vspace{2mm}
    (q-m+2) 
    \Delta^{q} \rho^{2r+2} \tau^s\\
    +(2r+1)\Delta^{q+1} \rho^{2r} \tau^{s+1}
    +s\Delta^{q+1} \rho^{2r+2} \tau^{s-1}
    \end{array}
  \right\}.
\end{align*}

\subsubsection*{$\bullet$ The case that $k$ is odd and $2\leq k\leq m$}
Since $k-1$ is even, 
$P(m;k-1)$ is expressed as
$$
  P(m;k-1)
  = \sum_{\substack{q+2r+s=k-1 \\ q,r,s\geq0}}
  C_{q,r,s}^{m,k-1} \Delta^q \rho^{2r} \tau^s.
$$
Differentiating the both sides, we have
\begin{align*}
  P(m;k-1)_z
  = g\homega \varphi 
  \sum_{\substack{q+2r+s=k-1 \\ q,r,s\geq0}}
  C_{q,r,s}^{m,k-1} 
  \left(
    \begin{array}{ll}
    \vspace{2mm}
  q\Delta^{q-1} \rho^{2r+1} \tau^s
  +2r\Delta^q \rho^{2r-1} \tau^{s+1}\\
  +s\Delta^q \rho^{2r+1} \tau^{s-1}
    \end{array}
  \right).
\end{align*}
Hence, we obtain
\begin{align*}
  P(m+1;k)
  &= \sum_{\substack{q+2r+s=k-1 \\ q,r,s\geq0}}\left[
    \begin{array}{ll}
    \vspace{2mm}
    \left\{ C_{q,r,s}^{m,k-1}(q-m+2)+C_{q,r,s}^{m,k}\right\} 
    \Delta^{q} \rho^{2r+1} \tau^s\\
    +2r\Delta^{q+1} \rho^{2r-1} \tau^{s+1}
    +s\Delta^{q+1} \rho^{2r+1} \tau^{s-1}
    \end{array}
  \right].
\end{align*}

\subsubsection*{$\bullet$ The case that $k$ is odd and $k=m+1$}
\begin{align*}
  P(m+1;m+1)
  &= \sum_{\substack{q+2r+s=m \\ q,r,s\geq0}}
  C_{q,r,s}^{m,m}\left\{
    \begin{array}{ll}
    \vspace{2mm}
    (q-m+2) 
    \Delta^{q} \rho^{2r} \tau^s\\
    +2r\Delta^{q+1} \rho^{2r-2} \tau^{s+1}
    +s\Delta^{q+1} \rho^{2r} \tau^{s-1}
    \end{array}
  \right\}.
\end{align*}

Therefore, there exist real numbers 
$C_{q,r,s}^{m+1,k}\in \R$ such that
\begin{gather*}
  \Delta^{m-1} \vf^{m+1} \varphi
  = \sum_{k=0}^{m+1} P(m+1;k) \tilde{\vf}^{m+1-k} \tilde{\varphi},\\
 \nonumber
  P(m+1;k) 
  =\left\{
  \begin{array}{cl}
  1 \vphantom{\dfrac1{6}} & (\text{if} ~k=0),\\
  {\displaystyle \sum_{\substack{q+2r+s=k \\ q,r,s\geq0}}
  C_{q,r,s}^{m+1,k} \Delta^q \rho^{2r} \tau^s}
  \vphantom{\dfrac1{6}} & (\text{if} ~k~ \text{is even and positive}),\\
  {\displaystyle \sum_{\substack{q+2r+s=k-1 \\ q,r,s\geq0}}
  C_{q,r,s}^{m+1,k} \Delta^q \rho^{2r+1} \tau^s}
  \vphantom{\dfrac1{6}} & (\text{if} ~k~ \text{is odd})
  \end{array}
  \right.,
\end{gather*}
and hence \eqref{eq:tilde-Dm} holds.

At a singular point $p\in \Sigma(f)$,
we have
$$
  \Delta(p)
  =\left| a\,g(p)+b \right|^2,\quad
  \rho(p)
  =2\sqrt{-1}\Im \left(a\bar{b}\,g(p) \right),\quad
  \tau(p)
  =2\Re \left(a\bar{b}\,g(p) \right),
$$
and hence, it holds that
$P(m;0)(p)=1$,
$P(m;k)(p)\in \R$ (if $k$ is even and positive), and
$P(m;k)(p)\in \sqrt{-1}\R$ (if $k$ is odd).
Evaluating \eqref{eq:tilde-Dm} at $p$,
we have the desired result.
\end{proof}

\begin{theorem}\label{thm:invariants}
Fix a positive integer $m\in \Z$.
Let $f:D\to S^3_1$ be a CMC $1$ face with Weierstrass data $(g,\omega)$.
We take a non-degenerate singular point $p\in \Sigma(f)$,
and take another Weierstrass data $(\tilde{g},\tilde{\omega})$ 
which is equivalent to $(g,\omega)$ as \eqref{eq:W-related}.
Let $(\varphi, \vf)$ and $(\tilde{\varphi}, \tilde{\vf})$
be the characteristic pairs 
associated with $(g,\omega)$ and $(\tilde{g},\tilde{\omega})$,
respectively.
\begin{enumerate}
\item
If $\Re(\sqrt{-1}^k \vf^{k} \varphi)(p)=0$ $(k=0,\dots,m-1)$, then
$$
  |a\,g(p)+b|^{2(m-2)} \Re(\sqrt{-1}^m \vf^m \varphi)(p)
  = \Re(\sqrt{-1}^m \tilde{\vf}^m \tilde{\varphi})(p)
$$
holds.
In particular,
$$
  \left\{
  \begin{array}{ll}
  \Re(\sqrt{-1}^k \vf^{k} \varphi)(p)=0\quad
  (k=0,\dots,m-1),
  \vphantom{\dfrac1{2}}\\
  \Re(\sqrt{-1}^m \vf^m \varphi)(p)
  = \Re(\sqrt{-1}^m \tilde{\vf}^m \tilde{\varphi})(p)
  \vphantom{\dfrac1{2}}
  \end{array}
  \right.
$$
holds if and only if $m=2$.
\item
If $\Im(\sqrt{-1}^k \vf^{k} \varphi)(p)=0$ $(k=0,\dots,m-1)$, then
$$
  |a\,g(p)+b|^{2(m-2)} \Im(\sqrt{-1}^m \vf^m \varphi)(p)
  = \Im(\sqrt{-1}^m \tilde{\vf}^m \tilde{\varphi})(p)
$$
holds. 
In particular,
$$
  \left\{
  \begin{array}{ll}
  \Im(\sqrt{-1}^k \vf^{k} \varphi)(p)=0\quad
  (k=0,\dots,m-1),
  \vphantom{\dfrac1{2}}\\
  \Im(\sqrt{-1}^m \vf^m \varphi)(p)
  = \Im(\sqrt{-1}^m \tilde{\vf}^m \tilde{\varphi})(p)
  \vphantom{\dfrac1{2}}
  \end{array}
  \right.
$$
holds if and only if $m=2$.
\end{enumerate}
\end{theorem}

\begin{proof}
We set $i=\sqrt{-1}$.
Multiplying the both sides of the identity 
in Lemma \ref{lem:invariants} by $i^m$, 
we have
$
  |a\,g(p)+b|^{2(m-2)} i^m \vf^m \varphi(p)=i^m \tilde{\vf}^{m} \tilde{\varphi}(p) 
  +\sum_{j=0}^{m-1}  c_{m-j} i^j \tilde{\vf}^{j} \tilde{\varphi}(p).
$
Hence, 
\begin{align*}
  &|a\,g(p)+b|^{2(m-2)} \Re (i^m \vf^m \varphi)(p)
  =  \Re (i^m \tilde{\vf}^{m} \tilde{\varphi})(p) 
  +\sum_{j=0}^{m-1}  c_{m-j} \Re (i^j \tilde{\vf}^{j} \tilde{\varphi})(p),\\
  &|a\,g(p)+b|^{2(m-2)} \Im (i^m \vf^m \varphi)(p)
  =  \Im (i^m \tilde{\vf}^{m} \tilde{\varphi})(p) 
  +\sum_{j=0}^{m-1}  c_{m-j} \Im (i^j \tilde{\vf}^{j} \tilde{\varphi})(p)
\end{align*}
hold, which yield the desired results.
\end{proof}

By Theorem \ref{thm:invariants}, under the condition
\begin{equation}\label{eq:condition-A}
\tag{$\mathbb{A}$}
  \Im (\varphi(p)) =\Re (\vf \varphi(p))=0,
\end{equation}
$\Im (\vf^2\varphi(p))$ does not depend on the choice of 
the Weierstrass data.
Similarly, under the condition
\begin{equation}\label{eq:condition-S}
\tag{$\mathbb{S}$}
  \Re (\varphi(p)) =\Im (\vf \varphi(p))=0,
\end{equation}
$\Re (\vf^2\varphi(p))$ does not depend on the choice of 
the Weierstrass data.

\begin{definition}
\label{def:invariants}
Let $f:D\to S^3_1$ be a CMC $1$ face,
$p\in \Sigma(f)$ be a non-degenerate singular point,
and $(\varphi, \vf)$ be the characteristic pair 
associated with a Weierstrass data $(g,\omega)$.
If $p$ satisfies the condition \eqref{eq:condition-A},
then 
$$
  \alpha(f,p):=\Im (\vf^2\varphi(p))
$$
does not depend on the choice of the Weierstrass data $(g,\omega)$.
The quantity $\alpha(f,p)$ is called the {\it $\alpha$-invariant} 
of $f$ at $p\in \Sigma(f)$.
Similarly, if 
$p$ satisfies the condition \eqref{eq:condition-S},
then 
$$
  \sigma(f,p):=\Re (\vf^2\varphi(p))
$$
does not depend on the choice of the Weierstrass data $(g,\omega)$.
The quantity $\sigma(f,p)$ 
is called the {\it $\sigma$-invariant} of $f$ at $p\in \Sigma(f)$.
\end{definition}

Let $k\geq 1$.
Since an $A_{k+3}$-type singular point satisfies the condition \eqref{eq:condition-A}
by Theorem \ref{thm:Ak-criteria-CMC1},
the $\alpha$-invariant is defined for $A_{k+3}$ singularities on CMC $1$ faces.
Similarly, a cuspidal $S_k$ singular point 
satisfies the condition \eqref{eq:condition-S} by
Theorem \ref{thm:Sk-criteria-CMC1},
the $\sigma$-invariant is defined for $cS_k$ singularities on CMC $1$ faces.

\subsection{Higher order derivatives}

For latter discussion, we will prove the following.

\begin{proposition}\label{prop:Sk-induction}
Let $k$ be a positive integer,
$(\varphi,V)$ be the characteristic pair 
associated with a Weierstrass data $(g,\omega)$,
and $\xi$ be the singular directional vector field 
given by \eqref{eq:xi-eta}.
Then, there exist real valued functions
$\tau_1^k,\dots, \tau_{k-1}^k$,
$\rho_1^k,\dots, \rho_{k-1}^k$
such that
\begin{align}
\label{eq:Re-xk}
\xi^k(\Re \varphi)
&= \sum_{\ell=1}^{k} \tau_{\ell}^k \Re (\sqrt{-1}^\ell \vf^{\ell} \varphi),\\
\label{eq:Im-xk}
\xi^k(\Im \varphi)
&= \sum_{\ell=1}^{k} \rho_{\ell}^k \Im (\sqrt{-1}^\ell \vf^{\ell} \varphi),
\end{align}
and $\tau_k^k=\rho_k^k=|g\homega \varphi|^{2k}$ hold,
where $\omega=h\,dz$.
\end{proposition}

For the proof of Proposition \ref{prop:Sk-induction},
we prepare two lemmas (Lemmas \ref{lem:Sk-1} and \ref{lem:Sk-2}).

\begin{lemma}\label{lem:Sk-1}
Under the setting in Proposition \ref{prop:Sk-induction},
it holds that
\begin{align}
\label{eq:Re-xi-k}
  \xi\left\{ \Re(\vf^k\varphi) \right\}
  &=- |g\homega \varphi|^2 \Im(\vf^{k+1}\varphi),\\
\label{eq:Im-xi-k}
  \xi\left\{ \Im(\vf^k\varphi) \right\}
  &=|g\homega \varphi|^2 \Re(\vf^{k+1}\varphi).
\end{align}
\end{lemma}

\begin{proof}
We set $i=\sqrt{-1}$.
Since $\xi=i\overline{(g\homega \varphi)}\partial/\partial z -ig\homega \varphi\, \partial/\partial \bar{z}$,
we have
\begin{align*}
\xi\{\Im(\vf^k \varphi)\}
&=i\overline{(g\homega \varphi)} \{\Im(\vf^k \varphi)\}_z-i(g\homega \varphi)\{\Im(\vf^k \varphi)\}_{\bar{z}} \\
&=\dfrac{1}{2}\{\overline{(g\homega \varphi)}(\vf^k\varphi)_z +g\homega \varphi \overline{(\vf^k\varphi)_z}\} 
=|g\homega \varphi|^2\dfrac{\vf^{k+1}\varphi +\overline{\vf^{k+1}\varphi}}{2},
\end{align*}
which yields \eqref{eq:Re-xi-k}.
By a similar calculation, we have \eqref{eq:Im-xi-k}.
\end{proof}

\begin{lemma}\label{lem:Sk-2}
Under the setting in Proposition \ref{prop:Sk-induction},
it holds that
\begin{align}
\label{eq:Re-ik}
  \xi\left\{ \Re(i^k\vf^k\varphi) \right\}
  &=|g\homega \varphi|^2 \Im(i^{k+1}\vf^{k+1}\varphi),\\
\label{eq:Im-ik}
  \xi\left\{ \Im(i^k\vf^k\varphi) \right\}
  &=|g\homega \varphi|^2 \Re(i^{k+1}\vf^{k+1}\varphi).
\end{align}
\end{lemma}

\begin{proof}
We set $i=\sqrt{-1}$.
In the case of $k\equiv 0 \mod 4$,
we have $i^k=1$.
By Lemma \ref{lem:Sk-1},  
\begin{align*}
  \xi\left\{ \Re(i^k\vf^k\varphi) \right\}
  =\xi\left\{ \Re(\vf^k\varphi) \right\}
  &=- |g\homega \varphi|^2 \Im(\vf^{k+1}\varphi)\\
  &=|g\homega \varphi|^2 \Re(i \vf^{k+1}\varphi)
  =|g\homega \varphi|^2 \Im(i^{k+1} \vf^{k+1}\varphi)
\end{align*}
holds, and hence we have \eqref{eq:Re-ik}.
In the cases of $k\equiv 1,2,3 \mod 4$, 
we can prove \eqref{eq:Re-ik} in a similar way.
An analogous calculation shows \eqref{eq:Im-ik}.
\end{proof}

\begin{proof}[Proof of Proposition \ref{prop:Sk-induction}]
We set $i=\sqrt{-1}$.
We prove \eqref{eq:Re-xk} by induction.
First, in the case of $k=1$,
substituting $k=0$ in Lemma \ref{lem:Sk-1},
we have
$$
  \xi( \Re \varphi )
  =- |g\homega \varphi|^2 \Im(\vf \varphi)
  =|g\homega \varphi|^2 \Re(i\vf \varphi).
$$
Next, for a positive integer $k$,
{\allowdisplaybreaks
we suppose that \eqref{eq:Re-xk} holds.
Since
\begin{align*}
&\xi^{k+1}(\Re \varphi)
= \sum_{\ell=1}^{k} \xi(\tau_{\ell}^k) \Re (i^\ell \vf^{\ell} \varphi)
  +\sum_{\ell=1}^{k} \tau_{\ell}^k \xi\{\Re (i^\ell \vf^{\ell} \varphi)\}\\
&= \sum_{\ell=1}^{k} \xi(\tau_{\ell}^k) \Re (i^\ell \vf^{\ell} \varphi)
  +\sum_{\ell=1}^{k} \tau_{\ell}^k |g\homega \varphi|^2 
  \Re(i^{\ell+1}\vf^{\ell+1}\varphi)\\
&= \xi(\tau_1^k) \Re (i \vf \varphi)
  +\sum_{\ell=2}^{k}  \left\{
    \xi(\tau_{\ell}^k) +\tau_{\ell-1}^k |g\homega \varphi|^2 
  \right\} \Re(i^{\ell}\vf^{\ell}\varphi)\\
&\hspace{60mm}  +|g\homega \varphi|^{2k+2} 
  \Re(i^{k+1}\vf^{k+1}\varphi),
\end{align*}
we}
obtain
$
\xi^{k+1}(\Re \varphi)
=\sum_{\ell=1}^{k+1} \tau_{\ell}^{k+1} \Re (i^\ell \vf^{\ell} \varphi),
$
where we set
$ \tau_{\ell}^{k+1}$ $(\ell=1,\dots,k+1)$ as
$$
  \tau_{\ell}^{k+1} 
  =\left\{
  \begin{array}{cl}
  \xi (\tau_1^k) 
    \vphantom{\dfrac1{6}} & (\text{if} ~\ell=1),\\
  \xi(\tau_{\ell}^k) +\tau_{\ell-1}^k |g\homega \varphi|^2 
    \vphantom{\dfrac1{6}} & (\text{if} ~\ell=2,\dots,k),\\
  |g\homega \varphi|^{2(k+1)}
    \vphantom{\dfrac1{6}} & (\text{if} ~\ell=k+1).
  \end{array}
  \right. 
$$
We can prove \eqref{eq:Im-xk} in a similar way
\end{proof}

\section{Duality between conelike singularities and $5/2$-cuspidal edges}
\label{sec:rhamphoid}

In this section, we derive a criterion for 
$5/2$-cuspidal edges on CMC $1$ faces in $S^3_1$
(Theorem \ref{thm:to-prove}),
which yields the duality between
generalized conelike singularities and $5/2$-cuspidal edges
(Theorem \ref{thm:A}).

\subsection{Generalized conelike singularities}

We set a smooth map
$\cone : \R^2 \to \R^3$ as
$\cone(u,v) := v(\cos u,\sin u,1)$.
Let $f:D\to M^3$ be a smooth map into a $3$-manifold $M^3$.
If the map germ $f$ at a point $p\in D$ 
is $\mathcal{A}$-equivalent to 
$\cone$ at the origin,
then $f$ is said to have a {\it conelike singular point} at $p\in D$.
In such a case, the following $({\rm c}_1)$--$({\rm c}_3)$ hold.
\begin{itemize}
\item[{\rm (${\rm c}_1$)}]
$f$ is a wave front on a neighborhood of $p$.
\item[{\rm (${\rm c}_2$)}]
$p\in \Sigma(f)$ is a non-degenerate singular point.
\item[{\rm (${\rm c}_3$)}]
Let $\gamma(t)$ $(|t|<\varepsilon)$
be a singular curve passing through $p=\gamma(0)$
where $\varepsilon>0$,
and $\eta(t)$ be a null vector field along $\gamma(t)$.
Then, there exists $\delta>0$ such that
$\det(\gamma'(t),\eta(t))=0$
holds for $t\in (-\delta,\delta)$.
\end{itemize}
The condition (${\rm c}_3$) implies that the image of the singular curve
consists of a single point.
By these conditions,
we give the following definition.

\begin{definition}
Let $f:D\to M^3$ be a frontal and $p\in \Sigma(f)$ be a singular point.
If the conditions $({\rm c}_1)$, $({\rm c}_2)$ and $({\rm c}_3)$
hold, then $f$ is said to have a {\it generalized conelike singular point}
at $p\in \Sigma(f)$.
\end{definition}

\begin{proposition}[{cf.\ \cite[Lemma 2.3]{FRUYY}}]
\label{prop:conelike-CMC1}
Let $f:D\to S^3_1$ be a CMC $1$ face with Weierstrass data $(g,\omega)$,
and $p\in D$ be a singular point.
We set $\varphi=dg/(g^2\omega)$.
Then, $f$ has a generalized conelike singular point at $p$
if and only if
$\Re \varphi(p)\ne0$
and
$\Im \varphi=0$
holds along a singular curve passing through $p$.
\end{proposition}

\begin{proof}
By Fact \ref{fact:singular-FSUY},
the condition (${\rm c}_1$) is equivalent to $\Re \varphi(p) \neq 0$.
Then the condition (${\rm c}_2$) is automatically satisfied.
Let $\xi$, $\eta$ be the vector fields defined by \eqref{eq:xi-eta}.
Then, we can choose a singular curve $\gamma(t)$
passing through $p=\gamma(0)$ so that 
$\gamma'(t)=\xi_{\gamma(t)}$
holds along $\gamma(t)$.
Then, we have
\begin{equation}\label{eq:Im-phi}
\det(\gamma'(t),\eta(t))=\Im \varphi
\end{equation}
along $\gamma(t)$
(cf.\ \cite[Theorem 2.4]{FSUY}).
Hence the condition (${\rm c}_3$)
holds if and only if $\Im \varphi=0$ along $\gamma(t)$.
\end{proof}

\subsection{$5/2$-cuspidal edges}

Let $\rce: \R^2\rightarrow \R^3$ be a map 
defined by 
$\rce(u,v) := (u,v^2,v^5)$,
which we call the {\it standard $5/2$-cuspidal edge}.
Let $f:D\to S^3_1$ be a smooth map.
We say that $f$ has {\it $5/2$-cuspidal edge}
(or {\it rhamphoid cuspidal edge}\/) at $p\in D$,
if the map germ $f$ at $p$ is $\mathcal{A}$-equivalent to 
$\rce$ at the origin.

\begin{figure}[htb]
\begin{center} 
 \begin{tabular}{{c@{\hspace{10mm}}c@{\hspace{8mm}}c}}
  \resizebox{3.2cm}{!}{\includegraphics{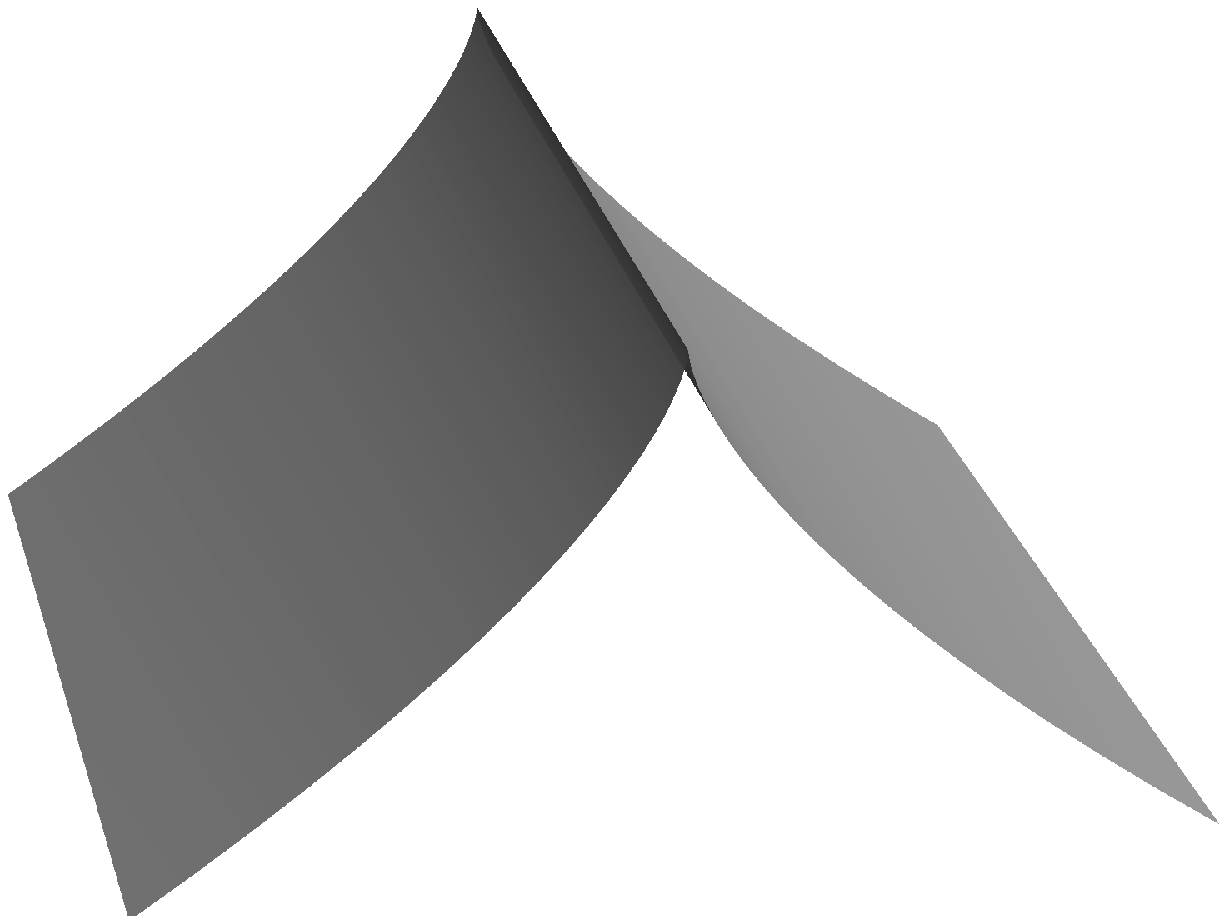}} &
  \resizebox{3.2cm}{!}{\includegraphics{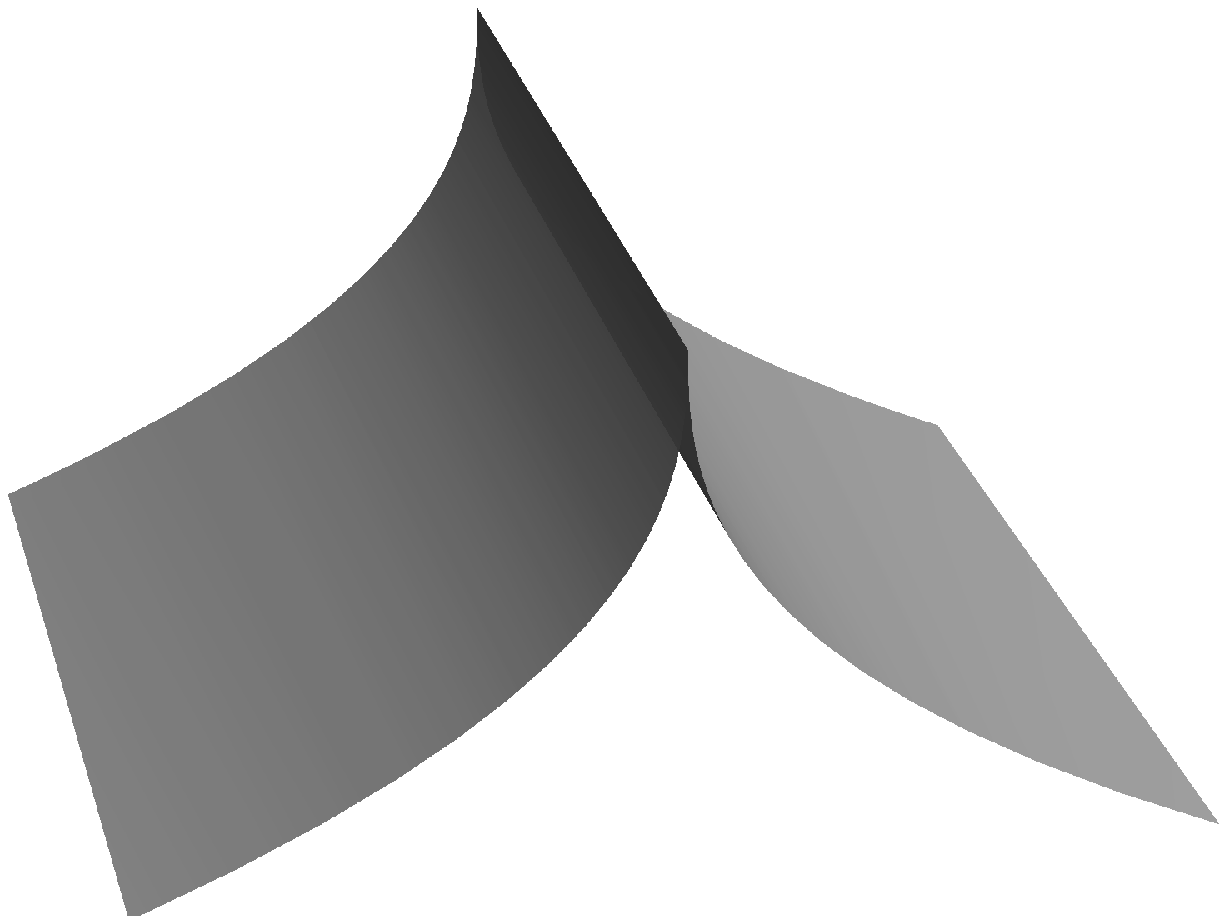}} &
  \resizebox{3cm}{!}{\includegraphics{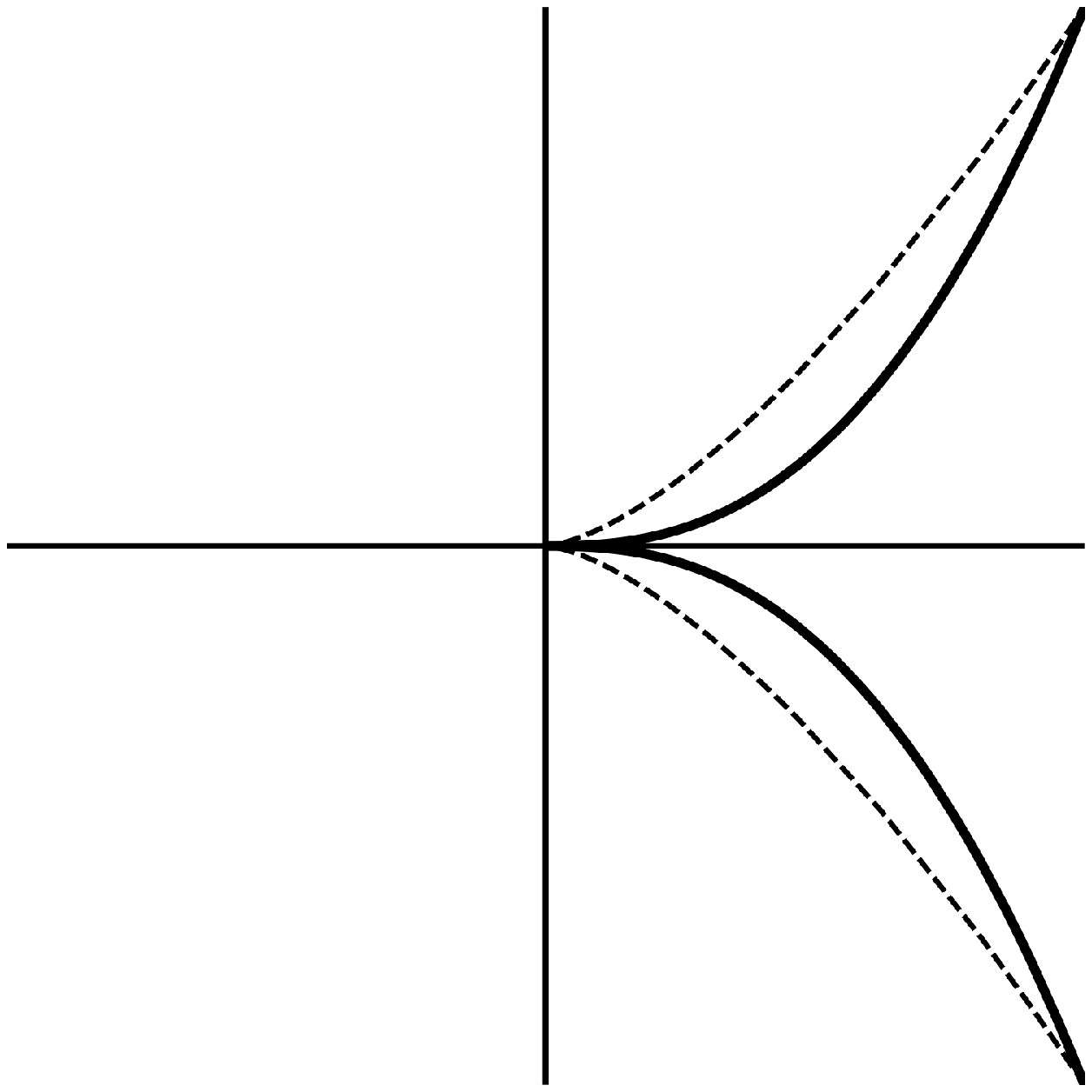}} 
 \end{tabular}
\end{center}
\caption{Left: the standard cuspidal edge 
$\ce(u,v)=(u,v^2,v^3)$.
Center:
the standard $5/2$-cuspidal edge 
$\rce(u,v)=(u,v^2,v^5)$.
Right: the images of the planar cusp $t\mapsto (t^2,t^3)$ (the dashed curve)
and the planar $5/2$-cusp $t\mapsto (t^2,t^5)$ (the bold curve).
The standard cuspidal edge and 
and standard $5/2$-cuspidal edge are 
obtained by parallel translations 
of the planar cusp and the planar $5/2$-cusp, respectively.
}
\end{figure}

We prove the following:

\begin{theorem}\label{thm:to-prove}
Let $f:D\to S^3_1$ be a CMC $1$ face with Weierstrass data $(g,\omega)$,
and $p\in D$ be a singular point.
We set $\varphi=dg/(g^2\omega)$.
Then, $f$ has $5/2$-cuspidal edge at $p$
if and only if
$\Im \varphi(p)\ne0$
and
$\Re \varphi=0$
holds along a singular curve passing through $p$.
\end{theorem}

For the proof of Theorem \ref{thm:to-prove},
we review the criterion for $5/2$-cuspidal edge.

\begin{fact}[Criterion for $5/2$-cuspidal edge {\cite[Theorem 4.1]{HKS}}]
\label{fact:rhamphoid}
Let $f:D\to \R^3$ be a frontal,
and $p\in D$ be a non-degenerate singular point.
We suppose that $p$ is of the first kind, and
\begin{equation}\label{eq:frontal}
\text{$f$ is not a front 
on a neighborhood of $p$.}
\end{equation}
Let $\xi$ be a singular directional vector field on $D$,
and $\eta$ be a null vector field which satisfies
\begin{equation}\label{eq:perp}
  \langle\xi f(p),\eta^2f(p)\rangle_{E}
  =\langle\xi f(p),\eta^3f(p)\rangle_{E}
  =0,
\end{equation}
where $\inner{~}{~}_E$ is the canonical Euclidean inner product of $\R^3$.
Then, $f$ has $5/2$-cuspidal edge at $p$ 
if and only if
$$
\det(\xi f, \eta^2f, 3\eta^5f-10c\,\eta^4f)(p)\ne0,
$$
where 
we set $\eta^j f = \eta(\eta^{j-1}f)$ for $j>1$,
and $c\in \R$ is a constant such that
$\eta^3f(p)=c\,\eta^2f(p)$.
\end{fact}

For details, see \cite[Theorem 4.1]{HKS}.
Since the criterion (Fact \ref{fact:rhamphoid})
uses the Euclidean inner product of $\R^3$,
we cannot apply it to our setting directly.
So, we will use the orthogonal projection, 
cf.\ \eqref{eq:orthogonal-proj}.

On the other hand, since the null vector field $\eta$
defined in \eqref{eq:xi-eta} 
does not satisfy \eqref{eq:perp},
we have to use another null vector field $\zeta$
satisfying \eqref{eq:perp}.
So we set
\begin{equation}\label{eq:zeta}
  \zeta
  = \left(\frac{\sqrt{-1}}{g\homega } + \const\,(1-|g|^2)^2\right)\frac{\partial}{\partial z}
  +\left(-\frac{\sqrt{-1}}{\bar{g}\bar{\homega} } + \bar{\const}\,(1-|g|^2)^2\right)
  \frac{\partial}{\partial \bar{z}},
\end{equation}
where we set $\const:=-\sqrt{-1}\,g(p)/(2g_z(p))$.

For a positive integer $j$,
we set $T_j$ so that 
$$
  \zeta^j f = F T_j F^*.
$$
In the following Lemmas \ref{lem:Tk}, \dots, \ref{lem:xi-fp},
we set $i=\sqrt{-1}$.

\begin{lemma}\label{lem:Tk}
We set 
$$
  \lambda:=1-|g|^2,
  \quad
  \Psi:=i \begin{pmatrix} 1 & -g\\ \frac1{g}&-1 \end{pmatrix},\quad
  \Omega:=\const\homega  \begin{pmatrix} g & -g^2\\1&-g \end{pmatrix}.
$$
Then, for $j>1$, it holds that
$$
T_j=U_j + (U_j)^*\quad
\left(
U_j:=
\left(\frac{i}{g\homega } + \const\,\lambda^2\right)(T_{j-1})_z
+ \left( \Psi + \lambda^2 \Omega \right)T_{j-1}
\right).
$$
\end{lemma}

\begin{proof}
By a direct calculation, we have 
\begin{equation}\label{eq:zeta-F}
  \zeta F 
  = \homega \left(\frac{i}{g\homega } + \const\,\lambda^2\right)
      F\begin{pmatrix} g & -g^2\\1&-g \end{pmatrix}
  =F\left( \Psi + \lambda^2 \Omega \right)
\end{equation}
and $\zeta F^* =\left( \Psi^*+\lambda ^2 \Omega^* \right)F^*$.
Since
$$
\zeta^j f 
= \zeta(\zeta^{j-1} f)
= (\zeta F)T_{j-1} F^*+F(\zeta T_{j-1} ) F^*+FT_{j-1} (\zeta F^*)
$$
and
$T_j = F^{-1}\,\zeta^j f \,(F^*)^{-1}$,
we can check that $T_j=U_j + (U_j)^*$.
\end{proof}

\begin{lemma}\label{lem:T2-j}
It holds that 
$T_1= \lambda T_{1,1} + \lambda^2 T_{1,2}$,
where
$$
  T_{1,1} 
  := i \begin{pmatrix} 0 & -\frac1{\bar{g}}\\ \frac1{g}&0 \end{pmatrix},
  \quad
  T_{1,2}
  := \const\homega  
  \begin{pmatrix} g & g^2\\1&g \end{pmatrix} 
  +
  \bar{\const}\bar{\homega } 
  \begin{pmatrix} \bar{g} & 1\\ \bar{g}^2&\bar{g} \end{pmatrix}.
$$
Moreover,
$$
  T_2 = \sum_{j=0}^5 \lambda^j \, T_{2,j}
  \quad
  \left(
  T_{2, j}:=U_{2, j} + (U_{2, j})^*
  \right)
$$
holds,
where we set
$U_{2,0}:=
\varphi 
\begin{pmatrix}
 0 & -g  \\
 \bar{g} & 0 \\
\end{pmatrix},
$
\begin{align*}
%
U_{2,1}&:=
\begin{pmatrix}
1 & \frac{1}{\bar{g}}\\
\frac{1+\varphi}{g} & \frac{1}{\bar{g} g}
\end{pmatrix}
-2 i g \bar{g} \varphi 
\left\{
\const \homega
\begin{pmatrix}
g & g^2\\
1 & g
\end{pmatrix}
+
\bar{\const} \bar{\homega}
\begin{pmatrix}
\bar{g} & 1\\
\bar{g}^2 & \bar{g}
\end{pmatrix}
\right\}\\
%
U_{2,2}
&:=
i \const \left\{
g \homega \varphi 
\begin{pmatrix}
1 & 3g\\
-\bar{g} &1
\end{pmatrix}
+
\frac{\homega_z}{\homega}
\begin{pmatrix}
1 & g\\
\frac{1}{g} &1
\end{pmatrix}
\right\},\\
%
U_{2,3}
&:=
i\bar{\const} \bar{\homega}\left(
\begin{array}{cc}
 \bar{g}   
 & 1 \\
 \frac{\bar{g} }{g}
 & \frac{1}{g} 
\end{array}
\right)
-\const \homega\left[
2g^2 \bar{g} \varphi
\left\{
\const \homega
\begin{pmatrix}
g & g^2\\
1 & g
\end{pmatrix}
+
\bar{\const} \bar{\homega}
\begin{pmatrix}
\bar{g} & 1\\
\bar{g}^2 & \bar{g}
\end{pmatrix}
\right\}
+i
\begin{pmatrix}
g & \frac{g}{\bar{g}}\\
1+\varphi & \frac{1}{\bar{g}}
\end{pmatrix}
\right]
\\
%
U_{2,4}
&:=
\const^2\left\{
g^2 \homega^2 \varphi
\begin{pmatrix}
 1 & 2g  \\
 0 & 1 
\end{pmatrix}
+
\homega_z
\begin{pmatrix}
 g & g^2  \\
 1 & g
\end{pmatrix}
\right\},
\quad
%
U_{2,5}
:=
\const\bar{\const}|\homega|^2
\begin{pmatrix}
 g\bar{g} & g  \\
 \bar{g} & 1
\end{pmatrix}.
\end{align*}
\end{lemma}

\begin{proof}
Substituting \eqref{eq:zeta-F} into
$\zeta f=(\zeta F)  e_3 F^* + F  e_3 (\zeta F^*)$,
we have
$$
\zeta f
= F\left[
   i \lambda \begin{pmatrix} 0 & -\frac1{\bar{g}}\\ \frac1{g}&0 \end{pmatrix}
   + \lambda^2 \left\{
       \const\homega  
       \begin{pmatrix} g & g^2\\1&g \end{pmatrix} 
       +
       \bar{\const}\homega  
       \begin{pmatrix} \bar{g} & 1\\\bar{g}^2&\bar{g} \end{pmatrix} 
  \right\}
  \right] F^*.
$$
With respect to $T_2$,
substituting $T_1 = \lambda \, T_{1,1} +  \lambda^2 \, T_{1,2}$
into 
$U_2 = \left( \Psi + \lambda^2 \Omega \right)T_{1}
+(\frac{i}{g\homega } + \const\,\lambda^2)\,(T_{1})_z$,
we have 
$U_2= \sum_{j=0}^4\redl^j {U_2}^j$,
where we set
${U_2}^0= \frac{i}{g\homega }\lambda_z T_{1,1}$,
\begin{gather*}
{U_2}^1= \frac{i}{g\homega }(2\lambda_z T_{1,2}+(T_{1,1})_z)
+ \Psi T_{1,1},\qquad
{U_2}^2 = \frac{i}{g\homega }(T_{1,2})_z
+ \const\, \lambda_z T_{1,1} +  \Psi T_{1,2}\\
{U_2}^3
= \vphantom{\frac1{6}}
\const\,(2\lambda_z T_{1,2}+(T_{1,1})_z)
+ \Omega T_{1,1},\qquad
{U_2}^4
= \vphantom{\frac1{6}}
\const\,(T_{1,2})_z+ \Omega T_{1,2}.
\end{gather*}
By a direct calculation,
we have
${U_2}^0=U_{2,0}$,
${U_2}^1=U_{2,1}$, 
\begin{gather*}
  {U_2}^2=U_{2,2}+\lambda U_{2,2*},
  \quad
  {U_2}^3=U_{2,3}-U_{2,2*},
  \quad
  {U_2}^4 = U_{2,4} + \lambda U_{2,5},
\end{gather*}
where 
we set
$U_{2,2*}:=i \bar{\const} \bar{\homega} 
\begin{pmatrix}
\bar{g} & 1\\
\frac{\bar{g}}{g} &\frac1{g}
\end{pmatrix}$.
Then we can check that
$U_2 =\sum_{j=0}^5 \redl^j U_{2,j}$
holds.
Together with $T_2 = U_2+(U_2)^*$ (Lemma \ref{lem:Tk}),
we obtain the desired result.
\end{proof}

\begin{lemma}\label{lem:T2pT3p}
Suppose that $p\in \Sigma(f)$
is a non-degenerate singular point,
and that $f$ is not a front at $p$.
Then, $T_2(p)$ and $T_3(p)$ are written as
$$
T_2(p)
= -2i \Im \varphi
\begin{pmatrix}0&g\\ -\bar{g}&0\end{pmatrix},
\quad
T_3(p)
= 2i \Im \varphi \Im (\vf \varphi)
\begin{pmatrix}
0 & g\\ -\bar{g} &0
\end{pmatrix},
$$
where the right hand side is evaluated at $p$.
\end{lemma}

\begin{proof}
Since $T_2(p)=T_{2,0}(p)=U_{2,0}(p)+(U_{2,0}(p))^*$,
we obtain $T_2(p)$ by Lemma \ref{lem:T2-j}.
By Lemma \ref{lem:Tk},
we have $T_3(p) = U_3(p)+U_3(p)^*$,
where
$U_3(p) 
  = \frac{i}{g\homega}(p)(T_{2})_z(p)
  + \Psi(p)T_{2}(p)$.
By Lemma \ref{lem:T2-j},
\begin{align}
\nonumber
  U_3(p)&=\frac{i}{g\homega }(T_{2,0})_z
  + \frac{i}{g\homega }\lambda_z T_{2,1}
  + \Psi T_{2,0}\\
\label{eq:U3}
  &=
2i(\Im \varphi)^2
\begin{pmatrix}
0 & 0\\ \bar{g} &0
\end{pmatrix}
+
\left(
i(\Im \varphi)^2
+
(\Im \varphi ) \vf \varphi
\right)
\begin{pmatrix}
0 & g\\ -\bar{g} &0
\end{pmatrix}.
\end{align}
Here, the right hand sides are evaluated at $p$.
Substituting \eqref{eq:U3} into $T_3(p) = U_3(p)+ U_3(p)^*$,
we obtain the desired result.
\end{proof}

We prove the following (Lemma \ref{lem:T5p-body}) in the appendix.

\begin{lemma}\label{lem:T5p-body}
Let $f:D\to S^3_1$ be a CMC $1$ face
with Weierstrass data $(g,\omega)$,
and $(\varphi,\vf)$ be the characteristic pair.
Suppose that 
$p\in D$ is a non-degenerate singular point
satisfying
$\Re (\varphi(p)) = \Im (\vf \varphi(p)) =0$.
Then, there exist $a_1,a_2,a_3\in\R$ such that
$$
  \zeta^5f(p)=
  a_1\,f+a_2\,\xi f+a_3\, \zeta^2 f
  + 4(12-\Re \vf^2\varphi)(\Im \varphi)^3 
   F\begin{pmatrix}
     1 & -g \\
     -\bar{g} & 1
   \end{pmatrix}F^*,
$$
where the right hand side is evaluated at $p$.
\end{lemma}

On the other hand, 
by a direct calculation,
we have the following.

\begin{lemma}\label{lem:xi-fp}
Let $\xi$ be the singular directional vector field
defined by \eqref{eq:xi-eta},
and $p$ be a non-degenerate singular point.
Then
$$
  \xi f(p)
  = 2 |\homega|^2 \Im (\varphi) 
  F
  \begin{pmatrix} 1 & g\\ \bar{g}&1 \end{pmatrix}
  F^*,
$$
where the right hand side is evaluated at $p$.
\end{lemma}

\begin{proof}[Proof of Theorem \ref{thm:to-prove}]
We set $i=\sqrt{-1}$.
Without loss of generality,
we may assume that 
the holomorphic null lift $F$ satisfies
$
F(p)= e_0
$
by an isometry \eqref{eq:action},
where $ e_0$ is the identity matrix \eqref{eq:pauli}.
We set
\begin{equation}\label{eq:orthogonal-proj}
  {\rm pr} : \Herm(2) \ni 
  \begin{pmatrix}
    x_0+x_3 & x_1+i x_2 \\
    x_1-i x_2 & x_0-x_3 
  \end{pmatrix}
  \longmapsto 
  (x_0,x_1,x_2)^T \in \R^3.
\end{equation}
The restriction 
${\rm pr}|_{S^3_1}:S^3_1\rightarrow \R^3$
gives a local diffeomorphism 
at $f(p)= e_3\in S^3_1$.
In particular, $f:D\to S^3_1$ has $5/2$-cuspidal edge
at $p\in D$ if and only if
so does 
$\tilde{f}:={\rm pr}\circ f : D\to \R^3$.
We apply the criteria (Fact \ref{fact:rhamphoid})
to $\tilde{f}$.
By Fact \ref{fact:singular-FSUY},
$p\in D$ is a singular point of the first kind
if and only if $\Im \varphi(p)\ne0$.
Also by Fact \ref{fact:singular-FSUY},
the condition \eqref{eq:frontal} holds
if and only if
$\Re \varphi=0$ holds along the singular curve passing through $p$.
Thus, $\xi^k\Re \varphi=0$ for arbitrary non-positive integer $k$.
By Proposition \ref{prop:Sk-induction},
we have 
\begin{equation}\label{eq:V2phi}
  \Re (\varphi(p))=
  \Im (\vf \varphi(p))=
  \Re (\vf^2\varphi(p))=0.
\end{equation}
Lemmas \ref{lem:T2pT3p}, \ref{lem:xi-fp}
yield that
$$
  \xi \tilde{f}(p)=
  2|\homega(p) |^2 (\Im \varphi(p))
  \begin{pmatrix}
  1\\   \Re g(p)\\   \Im g(p)
  \end{pmatrix},
  \qquad
  \zeta^2 \tilde{f}(p)=
  2\Im \varphi(p)
  \begin{pmatrix}
  0\\   \Im g(p)\\   -\Re g(p)
  \end{pmatrix},
$$
and $\zeta^3 \tilde{f}(p)=\vect{0}$.
Hence, the constant $c$ in the criteria 
(Fact \ref{fact:rhamphoid}) is $c=0$,
and 
$\inner{\xi \tilde{f}(p)}{\eta^2 \tilde{f}(p)}_E
=\inner{\xi \tilde{f}(p)}{\eta^3 \tilde{f}(p)}_E
=0$ holds,
namely, $\zeta$ is a null vector field 
satisfying \eqref{eq:perp}.
By Lemma \ref{lem:T5p-body},
we have
$$
\zeta^5\tilde{f}(p)
={\rm pr}(\zeta^5f(p))
= a_2 \xi \tilde{f}(p) + a_3 \zeta^2 \tilde{f}(p) +
  48 (\Im \varphi(p))^3 
  \begin{pmatrix}
  1\\   -\Re g(p)\\  -\Im g(p)
  \end{pmatrix},
$$
and hence
$
\det(\xi \tilde{f},\,\eta^2 \tilde{f},\,\eta^5 \tilde{f})(p)
= -384 |\homega(p) |^2 (\Im \varphi(p))^5\ne0.
$
Therefore, we have the assertion.
\end{proof}

Theorem \ref{thm:A} is a direct conclusion of the following.

\begin{theorem}
Let $f:D\to S^3_1$ be a CMC $1$ face
defined on a simply connected domain $D$ of $\C$,
and $p\in D$ be a singular point.
Then,
$f$ has generalized conelike singularity at $p$
if and only if 
$f^\sharp$ has $5/2$-cuspidal edge at $p$.
\end{theorem}

\begin{proof}
By Proposition \ref{prop:conelike-CMC1} and
Theorem \ref{thm:to-prove},
we have the desired result.
\end{proof}

\begin{remark}
\label{rem:rhamphoid-invariant}
By \eqref{eq:V2phi},
every $5/2$-cuspidal edge singular points 
satisfies the condition \eqref{eq:condition-S},
and hence, the $\sigma$-invariant can be defined.
However, \eqref{eq:V2phi} yields that
the $\sigma$-invariant vanishes along $5/2$-cuspidal edge singular points.
\end{remark}

\begin{example}
\label{ex:catenoid-relative}
Fix a constant $\constant\in\C\setminus\{0\}$.
Let $f:\C\to S^3_1$ be a CMC $1$ face
given by the Weierstrass data
$(g,\omega)=(e^z,\,\constant\, e^{-z}dz)$.
The singular set $\Sigma(f)$ coincides with the imaginary axis $\Re z=0$.
Since $\varphi=dg/(g^2\omega)=1/\constant$,
a singular point $z\in \Sigma(f)$ is
\begin{itemize}
\item cuspidal edge if and only if 
$ \constant\in \C\setminus (\R\cup \sqrt{-1}\R) $
(Fact \ref{fact:singular-FSUY}),
\item 
a generalized conelike singular point if and only if 
$ \constant\in \R\setminus \{0\}$
(Proposition \ref{prop:conelike-CMC1}),
\item $5/2$-cuspidal edge if and only if 
$ \constant\in \sqrt{-1}\R\setminus \{0\}$
(Theorem \ref{thm:to-prove}).
\end{itemize}
Therefore, we can observe the duality
between generalized conelike singular points
and $5/2$-cuspidal edge singular points
as in Theorem \ref{thm:A}.
See Figure \ref{fig:catenoid-cousin} for 
the images of $\constant=1$ and $\constant=-\sqrt{-1}$.
In the case of $\constant\neq-1/4$,
we can find a solution of the ODE \eqref{eq:F^-1dF} as
\begin{equation*}
  F= 
    \dfrac1{\sqrt{\rho-\tau}}
    \begin{pmatrix}
          \tau\, e^{\rho z} 
          & \rho\, e^{-\tau z}\\ 
          \rho\, e^{\tau z} 
          & \tau\, e^{-\rho z} 
    \end{pmatrix} 
    \quad\left(
      \rho=\frac{-1+\sqrt{1+4\constant}}{2},\quad
  \tau=\frac{-1-\sqrt{1+4\constant}}{2}\right),
\end{equation*}
and hence the map
$f=F e_3F^*$
can be explicitly written as
\begin{equation*}
  f= 
    \dfrac1{|\rho-\tau|}
    \begin{pmatrix} 
       |\tau|^2 e^{\rho z +\bar{\rho} \bar{z}} 
       - |\rho|^2 e^{-(\tau z +\bar{\tau} \bar{z})}
       &
       \bar{\rho}\tau\, e^{\rho z +\bar{\tau} \bar{z}} 
         - \rho\bar{\tau}\, e^{-(\tau z +\bar{\rho} \bar{z})}
       \\
       \rho\bar{\tau}\, e^{\tau z +\bar{\rho} \bar{z}} 
         - \bar{\rho}{\tau}\, e^{-(\rho z +\bar{\tau} \bar{z})}
       &
       |\rho|^2 e^{\tau z +\bar{\tau} \bar{z}} 
       - |\tau|^2 e^{-(\rho z +\bar{\rho} \bar{z})}
    \end{pmatrix}.
\end{equation*}
\end{example}

\section{Duality between $A_k$ singularities 
and cuspidal $S_k$ singular points}
\label{sec:AkSk}

\subsection{Generalized $A_k$ singular points}
For an integer $k\,(\geq 2)$, we set
${f}_{A_k} : \R^2 \to \R^3$ as
$
  {f}_{A_k}(u,v):= (u, - (k+1) v^k -2 u v, k v^{k+1}+u v^2).
$
The image of ${f}_{A_k}$ coincides with
the discriminant set
$$
  \mathcal{D}_F
  =\left\{(x,y,z)\in \R^3\,;\, 
  \text{there exists}~ t\in \R ~\text{s.t.}~ F=F_t=0\right\}
$$
of the function $F(t,x,y,z):=t^{k+1}+xt^2+yt+z.$

Let $f:D \to M^3$ be a smooth map
defined on a 2-manifold $D$
into a 3-manifold $M^3$.
A singular point $p\in \Sigma(f)$ is said to be 
an {\it $A_k$-type singular point}
(or {\it $A_k$-front singular point})
if $f$ at $p$ is $\mathcal{A}$-equivalent to 
${f}_{A_k}$ at the origin.
$A_2$- (resp.\ $A_3$-) type singular points are cuspidal edges 
(resp.\ swallowtails).
An $A_4$-type singular point is called {\it cuspidal butterfly}.

\begin{figure}[htb]
\begin{center} 
 \begin{tabular}{{c@{\hspace{8mm}}c}}
  \resizebox{3cm}{!}{\includegraphics{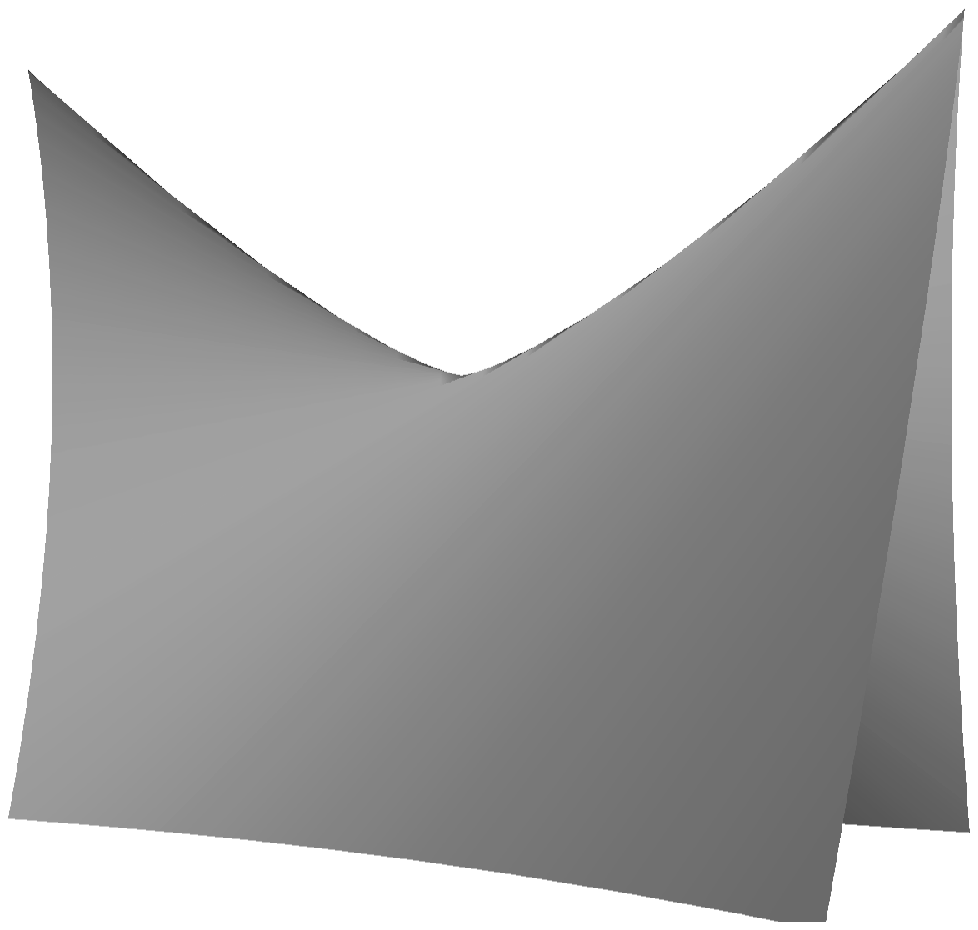}} &
  \resizebox{3cm}{!}{\includegraphics{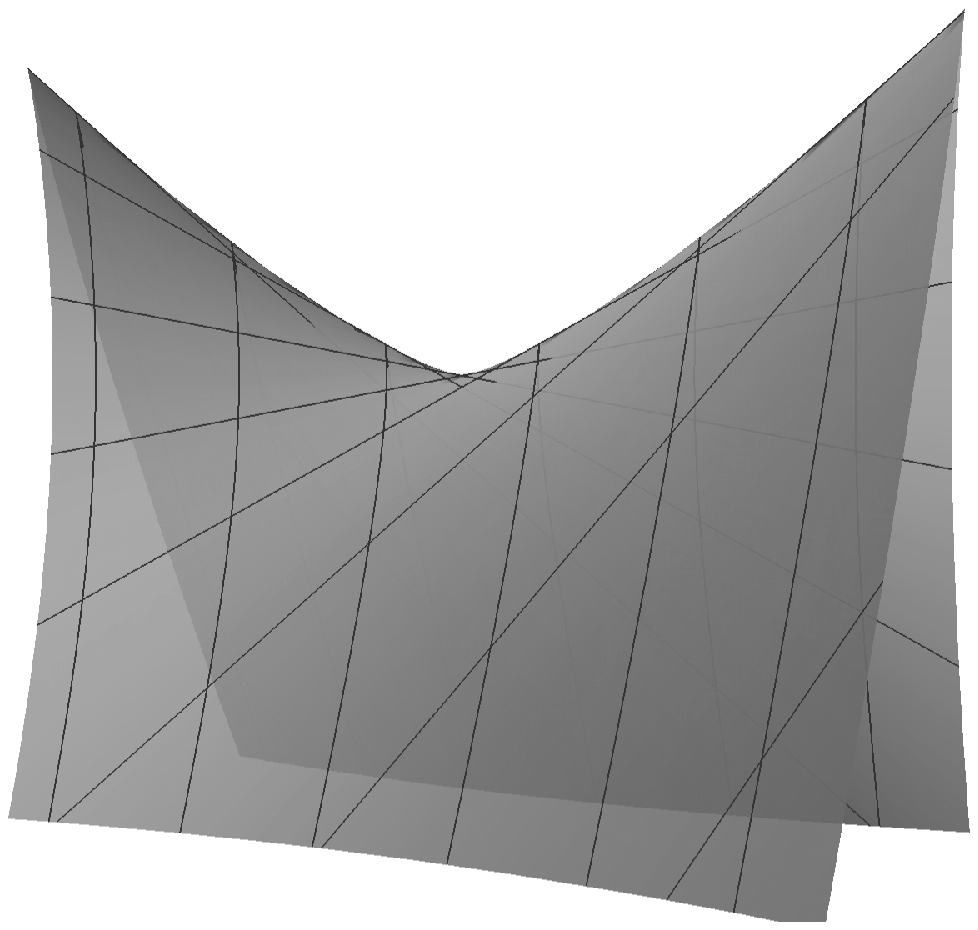}} 
 \end{tabular}
\end{center}
\caption{Cuspidal butterfly ($A_4$-type singular point).}
\end{figure}

If a smooth map $f:D \to M^3$
admits an $A_k$-type singular point $p\in D$,
then $f$ is a front on a neighborhood of $p$,
and $p$ is a non-degenerate singular point.
Moreover, if we take a singular curve $\gamma(t)$ $(|t|<\varepsilon)$ 
passing through $p=\gamma(0)$,
and a null vector field $\eta(t)$ along $\gamma(t)$,
it holds that
\begin{equation}\label{eq:Ak-like}
  \delta(0)=\delta'(0)=\dots=\delta^{(k-3)}(0)=0,
  \qquad
  \delta^{(k-2)}(0)\ne0,
\end{equation}
where we set $\delta(t):=\det(\gamma'(t),\eta(t))$.
Conversely, we have the following.

\begin{fact}[Criteria for $A_k$ singularities $(k=2,3,4)$ 
\cite{KRSUY, SUY, IS}]
\label{fact:Ak-criteria}
Let $k$ be either $2$, $3$, or $4$.
Also, let $f:D \to M^3$ be a front,
and $p\in D$ be a non-degenerate singular point.
Then, $p$ is an $A_k$-type singular point
if and only if \eqref{eq:Ak-like} holds.
\end{fact}

For $k\geq5$, 
to the best of the authors' knowledge, 
it is not known whether
\eqref{eq:Ak-like} characterizes
$A_k$-type singular points or not.
We give the following definition.

\begin{definition}[Generalized $A_k$ singularity]
\label{def:generalized-Ak}
Let $k\,(\geq2)$ be an integer.
Also, let $f:D \to M^3$ be a front,
and $p\in D$ be a non-degenerate singular point.
Then, $p$ is said to be a {\it generalized $A_k$ singular point}
if \eqref{eq:Ak-like} holds.
\end{definition}

By Fact \ref{fact:Ak-criteria},
generalized $A_k$ singular points are
$A_k$-type singular points for $k=2,3,4$.
The following holds.

\begin{theorem}
\label{thm:Ak-criteria-CMC1}
Let $f:D\to S^3_1$ be a CMC $1$ face
with Weierstrass data $(g,\omega)$,
and $p\in D$ be a singular point.
For an integer $k\,(>1)$,
$p$ is a generalized $A_k$ singular point
if and only if $\Re\varphi(p)\ne0$ and
$$
  \Im \varphi
  = \Im( \sqrt{-1} \vf \varphi)
  =\dots
  = \Im(\sqrt{-1}^{k-3} \vf^{k-3}\varphi)
  =0,\quad
  \Im(\sqrt{-1}^{k-2} \vf^{k-2}\varphi)\ne0
$$
at $p$.
Here $(\varphi,\vf)$ is the characteristic pair associated with $(g,\omega)$.
\end{theorem}

\begin{proof}
We set $i=\sqrt{-1}$.
By Fact \ref{fact:singular-FSUY},
$p\in D$ is a non-degenerate singular point
and $f$ is a front on a neighborhood of $p$
if and only if $\Re\varphi(p)\ne0$.
As in \eqref{eq:Im-phi},
it holds that
$
\delta
=\det(\gamma',\eta)
=\Im\varphi.
$
Thus, \eqref{eq:Ak-like} holds if and only if
\begin{equation}\label{eq:xi-Ak}
  \Im\varphi
  =\xi(\Im\varphi)
  =\dots 
  =\xi^{(k-3)}(\Im\varphi)
  \ne0, \quad 
  \xi^{(k-2)}(\Im\varphi)\ne0
\end{equation}
at $p$.
By Proposition \ref{prop:Sk-induction},
we have that \eqref{eq:xi-Ak} is equivalent to 
$$
  \Im \varphi
  = \Im( i \vf \varphi)
  =\dots
  = \Im(i^{k-3} \vf^{k-3}\varphi)
  =0,\quad
  \Im(i^{k-2} \vf^{k-2}\varphi)\ne0
$$
at $p$.
Hence, we have the desired result.
\end{proof}

\subsection{Cuspidal $S_k$ singular points}
For a non-negative integer $k$,
we set $\csk^\pm: \R^2\rightarrow \R^3$ as
$\csk^\pm(u,v) := (u,v^2,v^3(u^{k+1}\pm v^2))$.
Let $f:D \to M^3$ be a smooth map.
If the map germ $f$ at $p$ is $\mathcal{A}$-equivalent to 
$\csk^\pm$ at the origin,
$f$ is said to have a {\it cuspidal $S^{\pm}_k$ singular point}
(or a {\it $cS_k^\pm$ singular point}, for short) at $p\in D$.
We refer to $cS_k^+$ and $cS_k^-$ singular points collectively as 
{\it $cS_k$ singular points}.
In the case that $k$ is even,
$cS_k^+$ singular points 
are $\mathcal{A}$-equivalent to $cS_k^-$.
Moreover $cS_0$ singularity is the cuspidal cross cap.

\begin{figure}[htb]
\begin{center}
 \begin{tabular}{{c@{\hspace{8mm}}c}}
  \resizebox{3.5cm}{!}{\includegraphics{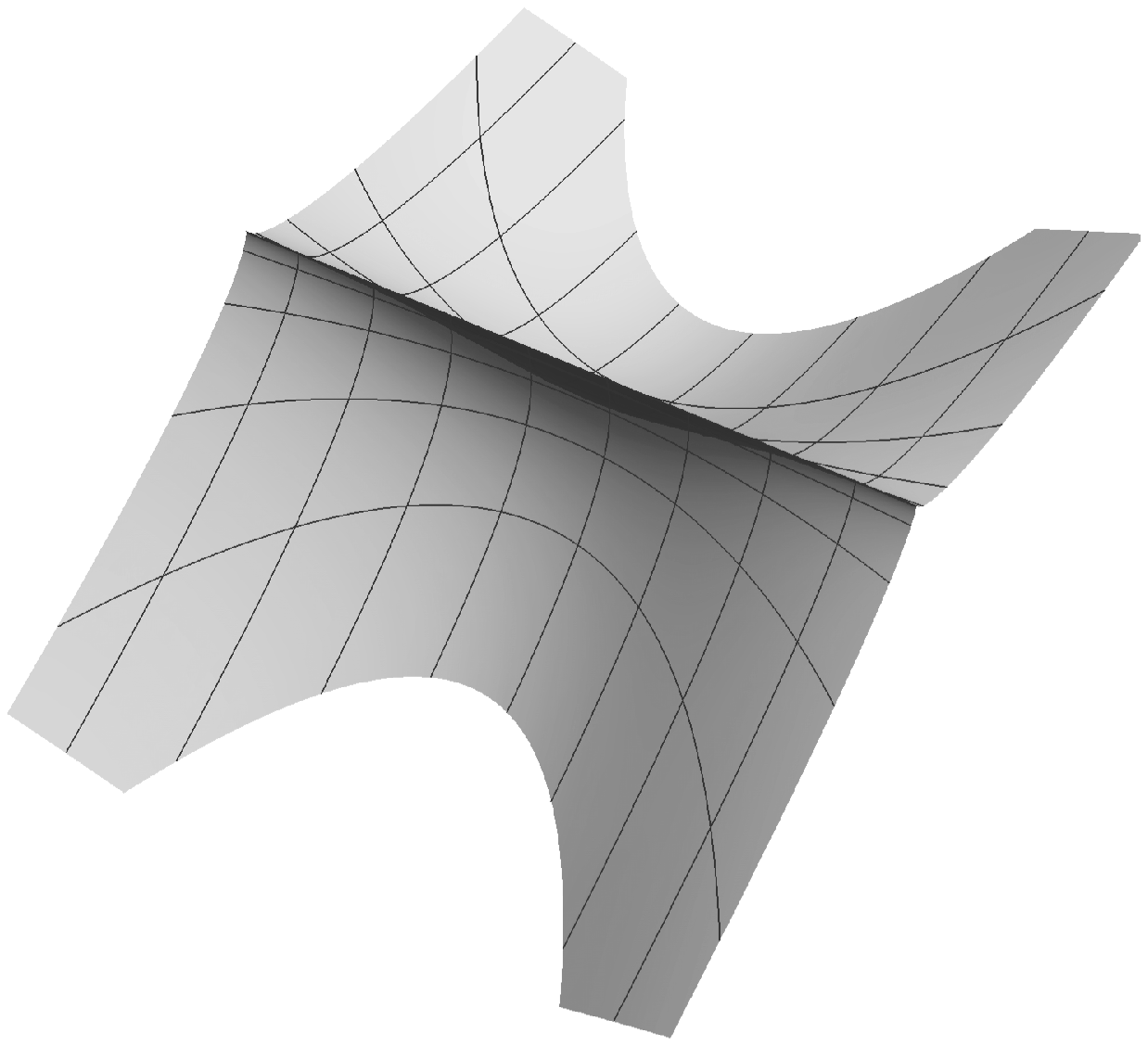}}&
  \resizebox{4cm}{!}{\includegraphics{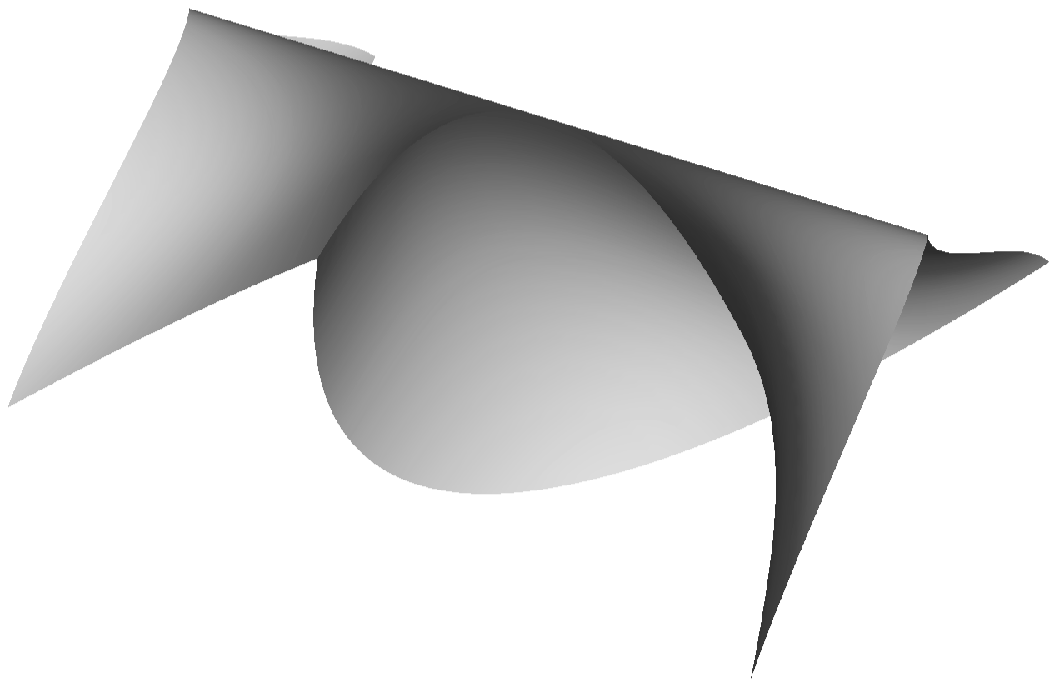}}
 \end{tabular}
\end{center}
\caption{Graphics of $cS_1$ singular points.
The left figure is the image of $\csone^+$, 
and the right one is that of $\csone^-$.}
\end{figure}

The following  is known:

\begin{fact}[Criteria for $cS_k$ singularities \cite{S, OT}]
\label{thm:S_ksing}
Let $f:D\to \R^3$ be a frontal,
$p\in D$ be a non-degenerate singular point,
and $\nu_E:D\to S^2$ be a unit normal vector field along $f$,
where $S^2$ is the unit sphere 
$S^2=\{(x,y,z)\in \R^3\,;\, x^2+y^2+z^2=1\}$.
Also let
$\gamma(t)$ $(|t|<\varepsilon)$
be a singular curve passing through $p=\gamma(0)$,
and let $\eta$ be a null vector field on $D$.
Set a function
$$
  \psi(t) := \det (\hat{\gamma}'(t),\hat{\nu}(t),d\nu(\eta(t))),
$$
where $\hat{\nu}:=\nu \circ \gamma$.
Then $f$ has $cS_k$ singularity at $p\in D$
if and only if
\begin{itemize}
\item[{\rm(1)}] 
$p$ is a singular point of the first kind,
\item[{\rm(2)}] 
$\psi(0)=\psi'(0)=\dots=\psi^{(k)}(0)=0$, $(A:=)\,\psi^{(k+1)}(0)\neq0$,
\item[{\rm(3)}] 
there exists a null vector field $\zeta$ 
defined on a neighborhood of $p$ such that
$\inner{\xi f(p)}{\zeta^2 f(p)}_E=0$,
$\zeta^3 f(p)=\vect{0}$ and 
$(B:=)\,\det(\xi f, \zeta^2f, \zeta^5 f)(p)\neq 0$,
where $\inner{~}{~}_E$ is the canonical Euclidean inner product.
\end{itemize}
Moreover, in the case that $k$ is odd,
$f$ has $cS_k^\pm$ singularity at $p\in D$
if and only if $\pm AB>0$.
\end{fact}

Applying this criteria, we prove the following.

\begin{proposition}\label{prop:Sk-criteria-CMC1}
Let $f:D\to S^3_1$ be a CMC $1$ face
with Weierstrass data $(g,\omega)$,
and $p\in D$ be a singular point.
Let us take an integer $k\,(>1)$.
Then, 
\begin{itemize}
\item 
$f$ has a $cS_1^\pm$ singular point at $p$
if and only if 
$$
  \Im\varphi\ne0,\quad
  \Re \varphi=\Im (\vf \varphi)=0,\quad
  \pm \Re(\vf^2 \varphi)\left(12-\Re(\vf^2 \varphi) \right)>0
$$
hold at $p$,
\item 
$f$ has a $cS_{2\ell}$ singular point $(\ell\geq1)$ at $p$
if and only if 
$$
  \Im\varphi\ne0,\quad
  \Re \varphi
  = \Im( \vf \varphi)
  =\dots
  = \Im (\vf^{2\ell-1}\varphi)
  = \Re (\vf^{2\ell}\varphi)
  =0,
$$
and $\Im(\vf^{2\ell+1}\varphi)\ne0$ hold at $p$,
\item 
$f$ has a $cS_{2\ell+1}^\pm$ singular point $(\ell\geq1)$ at $p$
if and only if 
$$
  \Im\varphi\ne0,\quad
  \Re \varphi
  = \Im( \vf \varphi)
  =\dots
  = \Re (\vf^{2\ell}\varphi)
  = \Im (\vf^{2\ell+1}\varphi)
  =0,
$$
and $\pm(-1)^{\ell} \Re(\vf^{2\ell+2}\varphi)>0$ hold at $p$.
\end{itemize}
Here $(\varphi,\vf)$ is the characteristic pair associated with $(g,\omega)$.
\end{proposition}

In the proof of \cite[Proposition 3.11]{MSUY},
the following assertion is proved.

\begin{fact}[\cite{MSUY}]\label{fact:MSUY}
Let $f:D\to \R^3$ be a frontal,
$\gamma(t)$ $(|t|<\varepsilon)$
be a singular curve passing through 
a non-degenerate singular point $p\in D$,
and let $\eta$ be a null vector field on $D$.
Then, there exists a positive valued function $\alpha(t)$
such that 
$$
\det (\xi f,\eta^2 f,\eta^3 f)(\gamma(t))=\alpha(t) \psi(t)
$$
holds on the singular set $\Sigma(f)$ near $p$.
\end{fact}

In particular, by Lemma \ref{fact:MSUY},
$\psi(0)=\psi'(0)=\dots=\psi^{(k)}(0)=0$, 
$\psi^{(k+1)}(0)\neq0$
holds if and only if 
$$
  \xi^\ell \det (\xi f,\eta^2 f,\eta^3 f)(p)=0
  \quad
  (\ell=0,\dots,k),
  \qquad
  \xi^{k+1} \det (\xi f,\eta^2 f,\eta^3 f)(p)\ne0.
$$
Moreover, then 
${\rm sgn}(\psi^{(k+1)}(0))
={\rm sgn}(\xi^{k+1} \det (\xi f,\eta^2 f,\eta^3 f)(p))$
holds.

\begin{proof}[Proof of Proposition \ref{prop:Sk-criteria-CMC1}]
Let $F:D \to S^3_1$ be a holomorphic null lift of $f$
which satisfies \eqref{eq:F^-1dF} with $F(p)=e_0$.
As in the proof of Theorem \ref{thm:to-prove},
we use the orthogonal projection \eqref{eq:orthogonal-proj}.
Since the restriction 
${\rm pr}|_{S^3_1} : S^3_1\to \R^3$
is an local diffeomorphism at $f(p)= e_3$,
the CMC $1$ face $f:D\to S^3_1$ has 
a $cS_k^\pm$ singular point at $p$
if and only if so does 
$\tilde{f}:={\rm pr}\circ f : D\to \R^3$.
With respect to the condition (1),
Fact \ref{fact:singular-FSUY} yields that
$p$ is a singular point of the first kind
if and only if $\Im \varphi(p)\ne0$.
Next, we consider the condition (2).
Let $\xi$ be the singular directional vector field,
and $\eta$ be the null vector field given by \eqref{eq:xi-eta}. 
Since
\begin{align*}
\det(\xi \tilde{f},\eta^2\tilde{f}, \eta^3\tilde{f})
=\det({\rm pr}(\xi f),{\rm pr}(\eta^2 f),{\rm pr}(\eta^3 f))
=-\det(\varepsilon_3,\Phi(\xi f), \Phi(\eta^2 f), \Phi(\eta^3 f)),
\end{align*}
we calculate $\xi f$, $\eta^2 f$ and $\eta^3 f$
along the singular set $\Sigma(f)$,
where $\varepsilon_3=(0,0,0,1)^T$.
By Lemma \ref{lem:xi-fp}, we have
\begin{align*}
\xi f(\gamma(t)) =
  2 |\homega|^2 (\Im \varphi)\,
  F\begin{pmatrix} 1&g \\ \bar{g}&1 \end{pmatrix}F^*,
\end{align*}
where the right hand side is evaluated at $\gamma(t)$.
On the other hand, 
{\allowdisplaybreaks
a direct calculation yields that
\begin{align*}
\eta^2 f(\gamma(t))
&= 2(\Im \varphi)F
  \begin{pmatrix}0&-\sqrt{-1} g \\ \sqrt{-1}\bar{g}&0 \end{pmatrix}F^*,\\
\eta^3 f(\gamma(t))
&=a(t)\xi f
+b(t)\eta^2 f
+ 4(\Im \varphi) (\Re \varphi) 
F\begin{pmatrix}0&g \\ \bar{g}&0 \end{pmatrix}F^*,
\end{align*}
where the right hand sides are evaluated at $\gamma(t)$,
and $a(t)$, $b(t)$ are some functions.}
Let $\check{F}:D\to \SO^+(1,3)$ be the map
defined by the holomorphic null lift $F$
via \eqref{eq:double-cover}.
Under the identification $\Phi$ as in \eqref{eq:Psi},
we have
\begin{gather*}
\Phi(\xi( \gamma(t) ))=
2 |\homega|^2 (\Im \varphi)\,
\check{F}\,
\begin{pmatrix}1 \\ \Re g \\\Im g \\0\end{pmatrix},\quad
\Phi( \eta^2f( \gamma(t) ) )
=2(\Im \varphi)\check{F}\,
\begin{pmatrix}0 \\\Im g \\-\Re g \\0\end{pmatrix},\\
\Phi( \eta^3f( \gamma(t) ) )
=a(t)\Phi(\xi f)
+b(t)\Phi(\eta^2 f)
+ 4(\Im \varphi) (\Re \varphi)
\check{F}\,
\begin{pmatrix}0 \\ \Re g \\ \Im g \\0\end{pmatrix}.
\end{gather*}
Since $\det \check{F}=1$,
{\allowdisplaybreaks
we have
\begin{align*}
-\det(\varepsilon_3,\Phi(\xi f), \Phi(\eta^2 f), \Phi(\eta^3 f))
&=-16|\homega|^2(\Im \varphi)^3 (\Re \varphi) 
\det
\begin{pmatrix}
\check{F}^{14}&1     &0&0 \\
\check{F}^{24}&\Re g&\Im g&\Re g \\
\check{F}^{34}&\Im g&-\Re g&\Im g \\
\check{F}^{44}&0&0&0
\end{pmatrix} \\
&=16|\homega|^2(\Im \varphi)^3\check{F}^{44}(\Re \varphi)
\end{align*}
along $\gamma(t)$.}
%
Since $\check{F}(p)$ is the identity matrix,
$\check{F}^{44}(p)=1$,
where $\check{F}^{-1}=(\check{F}^{jk})_{j,k=1,2,3,4}$.
In particular, 
$\check{F}^{44}>0$ holds on a neighborhood of $p$.
Hence, there exists 
a positive valued function $\rho(t)$
such that
\begin{equation}\label{eq:Re-phi}
  \det(\xi \tilde{f},\eta^2\tilde{f}, \eta^3\tilde{f})(\gamma(t))
  = {\rm sgn} (\Im \varphi) \, \rho(t) \,(\Re \varphi)(\gamma(t)).
\end{equation}
By Proposition \ref{prop:Sk-induction},
the condition (2) is equivalent to
\begin{equation}\label{eq:condition-2}
\Re \varphi = \dots =\Re (i^k\vf^k\varphi)= 0,\quad
(A:=)~{\rm sgn} (\Im \varphi)\, \Re (i^{k+1}\vf^{k+1}\varphi) \ne 0
\end{equation}
at $p$, where we set $i=\sqrt{-1}$.
In particular, 
since $k\geq1$,
$\Re \varphi(p)=\Im(\vf \varphi(p))=0$ 
holds.


Finally we consider the condition (3).
If we set $\zeta$ as in \eqref{eq:zeta},
Lemma \ref{lem:T2pT3p} yields that
$$
  \zeta^2 f(p)
  = 2 (\Im \varphi)
  \begin{pmatrix}0&-i g \\ i\bar{g}&0 \end{pmatrix},\quad
  \zeta^3 f(p)=O.
$$
Moreover, 
{\allowdisplaybreaks
by Lemma \ref{lem:T5p-body},
there exist $a_1,a_2,a_3\in\R$ such that
$$
  \zeta^5f(p)=
  a_1f(p)+a_2\xi f(p)+a_3 \zeta^2 f(p)
  + 4(12-\Re \vf^2\varphi)(\Im \varphi)^3 
   \begin{pmatrix}
     1 & -g \\
     -\bar{g} & 1
   \end{pmatrix}.
$$
Hence, we obtain
\begin{align*}
B&=
\det(\xi \tilde{f},\zeta^2\tilde{f},\zeta^5\tilde{f})(p)\\
&=-\det(\Phi(f),\Phi(\xi f),\Phi(\zeta^2f),\Phi(\zeta^5f))(p) \\
&=16|\homega|^2(\Im \varphi)^5(12-\Re \vf^2 \varphi)\det
\begin{pmatrix}
0&1&0&1 \\
0&\Re g&\Im g&-\Re g \\
0&\Im g&-\Re g&-\Im g \\
1&0&0&0
\end{pmatrix} \\
&=-32|\homega|^2(\Im \varphi)^5(12-\Re \vf^2 \varphi)(p).
\end{align*}}

If $k\geq2$,
we have $\Re \vf^2 \varphi(p)=0$ by \eqref{eq:condition-2},
and hence 
$
  B=-384|\homega|^2(\Im \varphi)^5(p)\ne0,
$
which yields the condition (3).
Then, 
$
  AB=-384|\homega|^2|\Im \varphi|^5 \Re(i^{k+1}\vf^{k+1}\varphi)(p)
$
holds.
If $k=1$,
we have 
$$
  AB=32|\homega|^2 |\Im \varphi|^5 
  (\Re \vf^2\varphi)(12-\Re \vf^2 \varphi)(p),
$$
which yields the desired result.
\end{proof}

As a direct conclusion of Proposition \ref{prop:Sk-criteria-CMC1},
we have the following.

\begin{theorem}\label{thm:Sk-criteria-CMC1}
Let $f:D\to S^3_1$ be a CMC $1$ face
with Weierstrass data $(g,\omega)$,
and $p\in D$ be a singular point.
Let us take an integer $k\,(>1)$.
Then, 
\begin{enumerate}
\item 
$f$ has a $cS_1$ singular point at $p$
if and only if 
$$
  \Im\varphi\ne0,\quad
  \Re \varphi=\Im (\vf \varphi)=0,\quad
  \Re(\vf^2 \varphi)\ne0,12
$$
hold at $p$,
\item 
$f$ has a $cS_k$ singular point at $p$
if and only if $\Im\varphi\ne0$ and
$$
  \Re \varphi
  = \Re( \sqrt{-1} \vf \varphi)
  =\dots
  = \Re(\sqrt{-1}^{k} \vf^{k}\varphi)
  =0,
  \quad
  \Re(\sqrt{-1}^{k+1} \vf^{k+1}\varphi)\ne0
$$
hold at $p$.
\end{enumerate}
Here $(\varphi,\vf)$ is the characteristic pair associated with $(g,\omega)$.
\end{theorem}

By Theorem \ref{thm:Ak-criteria-CMC1},
if a CMC $1$ face $f$ has a $cS_1$ singular point at $p$,
then the condition \eqref{eq:condition-S}, 
i.e., $\Re \varphi=\Im\vf\varphi=0$ at $p$, holds.
Hence, the $\sigma$-invariant $\sigma(f,p)$ can be defined
(cf.\ Definition \ref{def:invariants}).
As a corollary of Theorem \ref{thm:Ak-criteria-CMC1},
we have the following.

\begin{corollary}\label{cor:cS1-sigma-invariant}
Let $f:D\to S^3_1$ be a CMC $1$ face.
Suppose that $f$ has a $cS_1$ singular point at $p$.
Then the $\sigma$-invariant $\sigma(f,p)$
can be defined and $\sigma(f,p)\ne12$ holds.
\end{corollary}

By Theorem \ref{thm:Ak-criteria-CMC1},
if a CMC $1$ face $f$ has 
cuspidal butterfly ($A_4$-type singularity)
at $p$,
then the condition \eqref{eq:condition-A}, 
i.e., $\Im \varphi=\Re\vf\varphi=0$ at $p$, holds.
Hence, the $\alpha$-invariant $\alpha(f,p)$ can be defined
(cf.\ Definition \ref{def:invariants}).

\begin{definition}[Admissibility]
\label{def:generic-A4}
Suppose that a CMC $1$ face $f$ has cuspidal butterfly at $p\in D$.
If $|\alpha(f,p)|\neq 12$, then $p$ is called an {\it admissible} cuspidal butterfly.
On the other hand, 
if a CMC $1$ face $f$ has $cS_1$ singularity at $p\in D$
such that $\sigma(f,p)\neq -12$,
then $p$ is called an {\it admissible} $cS_1$ singularity.
\end{definition}

By Corollary \ref{cor:cS1-sigma-invariant},
a CMC $1$ face $f$ has an admissible $cS_1$ singular point $p\in D$
if and only if $|\sigma(f,p)|\neq 12$.
Theorem \ref{thm:B} is a direct conclusion of the following.

\begin{theorem}
Let $f:D\to S^3_1$ be a CMC $1$ face
defined on a simply connected domain $D$ of $\C$,
and $p\in D$ be a singular point.
Then, we have the following.
\begin{itemize}
\item[{\rm (1)}]
The CMC $1$ face $f$ has cuspidal $S_1$ singularity at $p$
if and only if 
the conjugate $f^\sharp$ has admissible cuspidal butterfly at $p$.
\item[{\rm (2)}] 
For $k\geq2$,
the CMC $1$ face $f$ has cuspidal $S_k$ singularity at $p$
if and only if 
the conjugate $f^\sharp$ has generalized $A_{k+3}$ singularity at $p$.
\end{itemize}
\end{theorem}

\begin{proof}
By Theorems \ref{thm:Ak-criteria-CMC1} and \ref{thm:Sk-criteria-CMC1},
we have the desired result.
\end{proof}

\begin{example}
\label{ex:cuspidal-butterfly}
Fix a constant $\constant \in \C \setminus \{0\}$.
Let $f:\C\to S^3_1$ be a CMC $1$ face
given by the Weierstrass data
$
  (g,\omega)
  =\left(-e^{z}+1/\sqrt{2},\, \constant e^{-z}dz\right).
$
Then, $z_0=(-\log 2+\sqrt{-1}\pi)/2$ 
is a singular point of $f$.
Since
$$
  \varphi(z_0)=\frac{\sqrt{-1}}{2\constant},
  \quad
  \vf \varphi(z_0)=-\frac{1}{\constant},
  \quad
  \vf^2\varphi(z_0)=\frac{1-\sqrt{-1}}{\constant},
$$
it holds that
\begin{itemize}
\item 
$f$ has $cS_1^+$ (resp.\ $cS_1^-$) singularity at $z_0$
if and only if 
$\constant\in \R$ and $\constant>1/12$
(resp.\ $k\in \R\setminus\{0\}$ and $\constant<1/12$).
In particular, 
$z_0\in \Sigma(f)$ is an admissible $cS_1$ singular point
if and only if 
$\constant\in \R\setminus\{0\}$ and $\constant\ne\pm1/12$,
\item 
$f$ has cuspidal butterfly 
if and only if 
$\constant\in \sqrt{-1}\R\setminus\{0\}$.
In particular, 
$z_0\in \Sigma(f)$ is an admissible cuspidal butterfly singular point
if and only if
$\constant\in \sqrt{-1}\R\setminus\{0\}$ and $\constant\ne\pm\sqrt{-1}/12$.
\end{itemize}
Therefore, we can observe the duality
between admissible $cS_1$ singularity
and admissible cuspidal butterfly
as in Theorem \ref{thm:B}.

To visualize the figure of the surface,
we use Small's formula \eqref{eq:Small}.
Let us set
$G(z)=\tan\left(\frac{z}{2}\sqrt{4\constant-1}\right).$
Then $G$, $g$, and 
the Hopf differential $Q=\omega dg$ satisfy \eqref{eq:QgG}.
Substituting $G,g,\omega$ into Small's formula \eqref{eq:Small},
we can write down $F$ and $f=F e_3F^*$, explicitly.
See Figures \ref{fig:cS1-plus}, \ref{fig:cS1-minus}, \ref{fig:cuspidal-butterfly} 
for $\constant=1$, $\constant=-1/50$ and $\constant=-\sqrt{-1}$.
\end{example}

\begin{figure}[htb]
\begin{center}
 \begin{tabular}{{c@{\hspace{8mm}}c}}
  \resizebox{4cm}{!}{\includegraphics{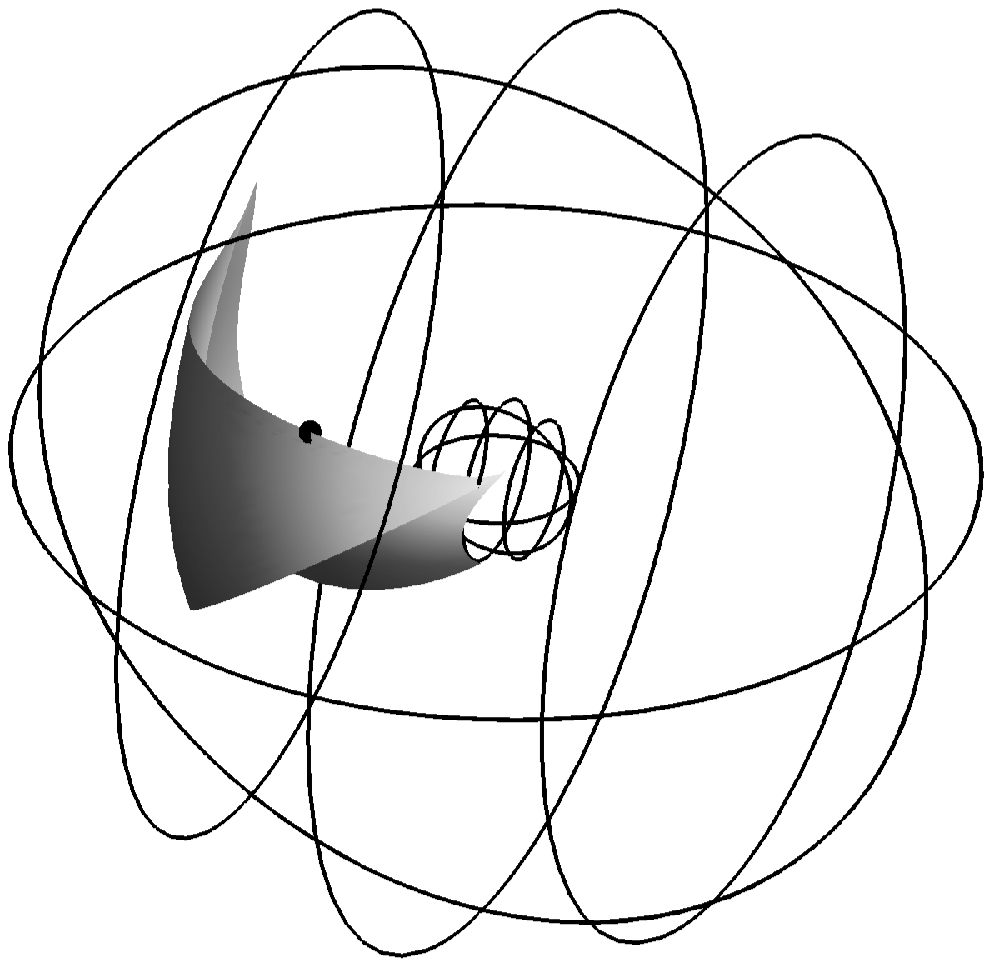}} &
  \resizebox{3cm}{!}{\includegraphics{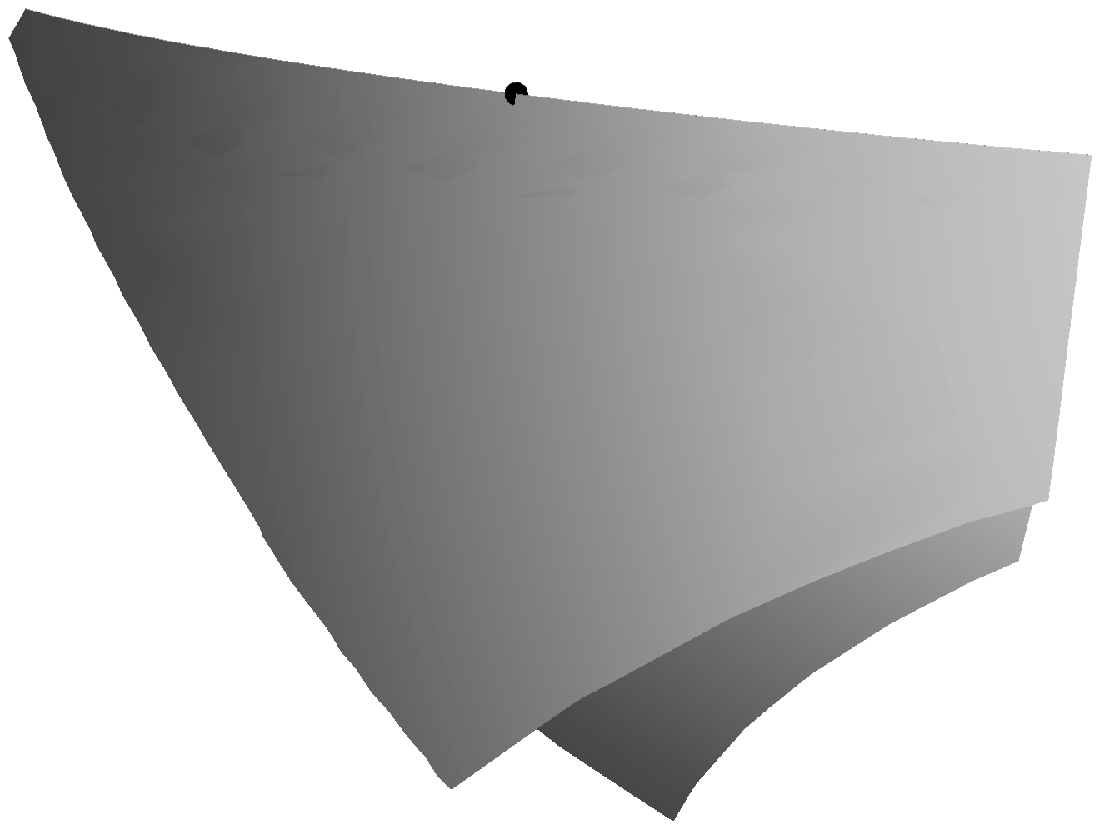}}
 \end{tabular}
\end{center}
\caption{
A CMC $1$ face with $cS_1^+$ singular point
given in Example \ref{ex:cuspidal-butterfly} with $\constant=1$.} 
\label{fig:cS1-plus}
\end{figure}

\begin{figure}[htb]
\begin{center}
 \begin{tabular}{{c@{\hspace{8mm}}c}}
  \resizebox{4cm}{!}{\includegraphics{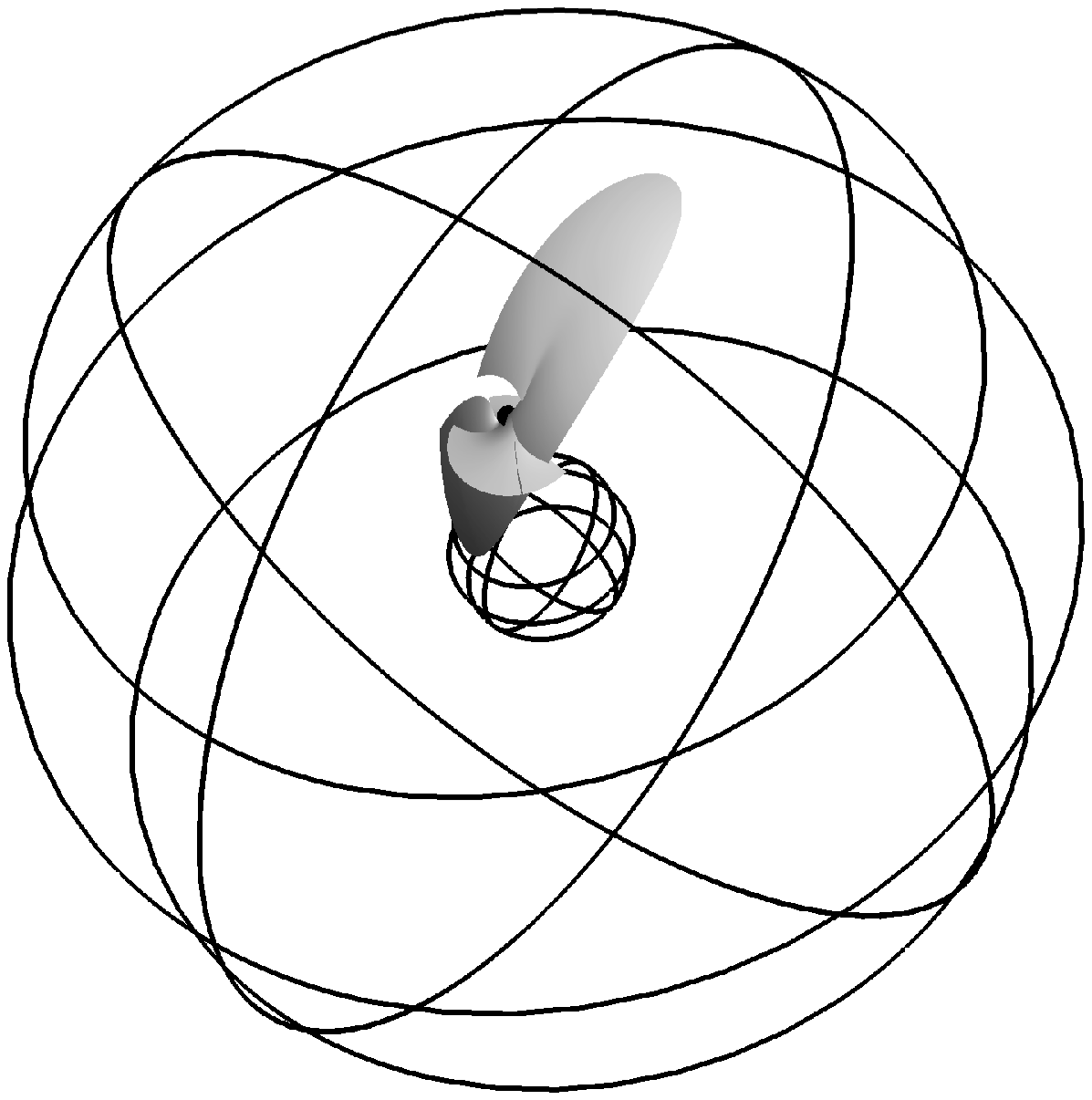}} &
  \resizebox{3cm}{!}{\includegraphics{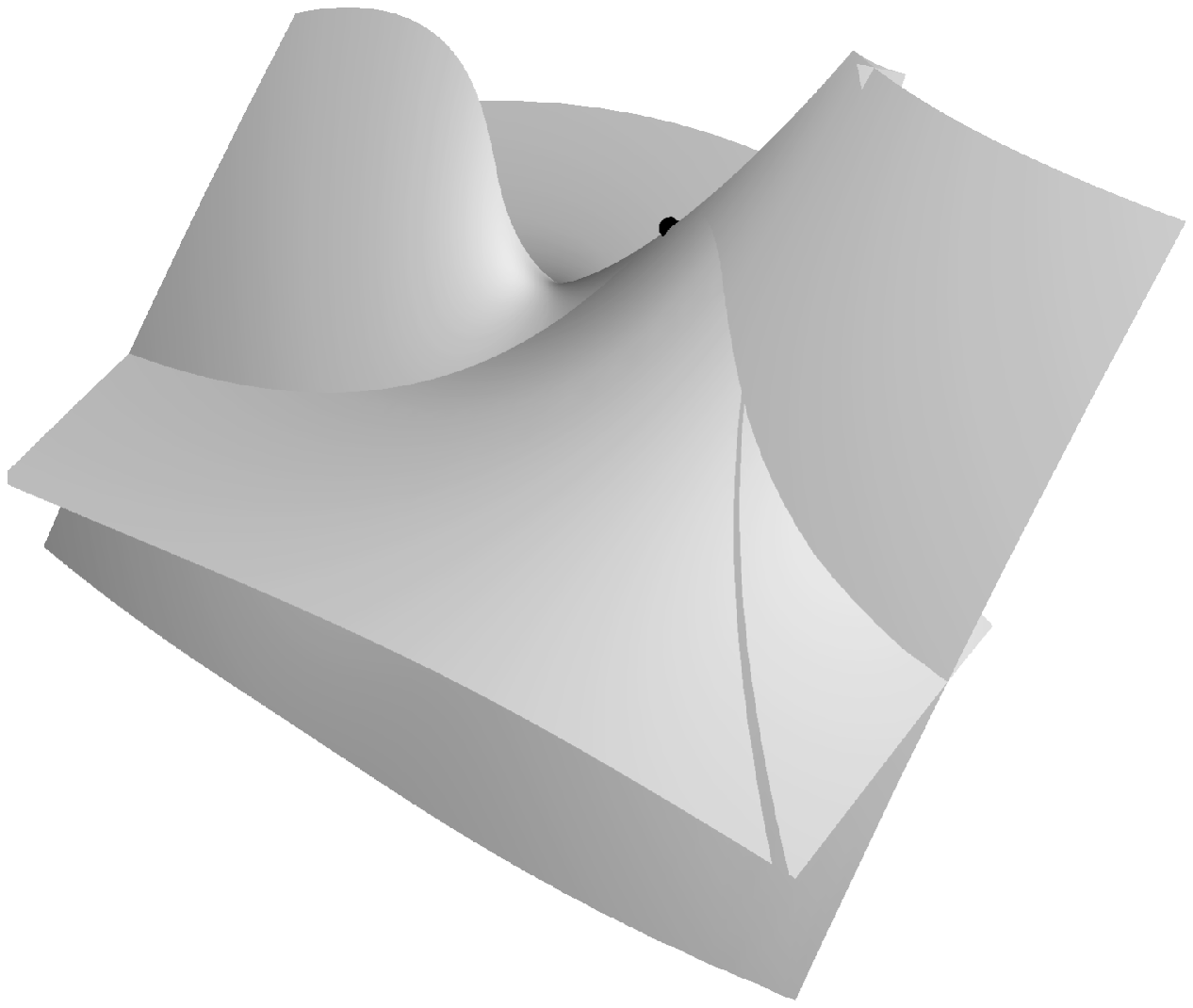}}
 \end{tabular}
\end{center}
\caption{
A CMC $1$ face with $cS_1^-$ singular point
given in Example \ref{ex:cuspidal-butterfly} with $\constant=-1/50$.} 
\label{fig:cS1-minus}
\end{figure}

\begin{figure}[htb]
\begin{center}
 \begin{tabular}{{c@{\hspace{8mm}}c}}
  \resizebox{4cm}{!}{\includegraphics{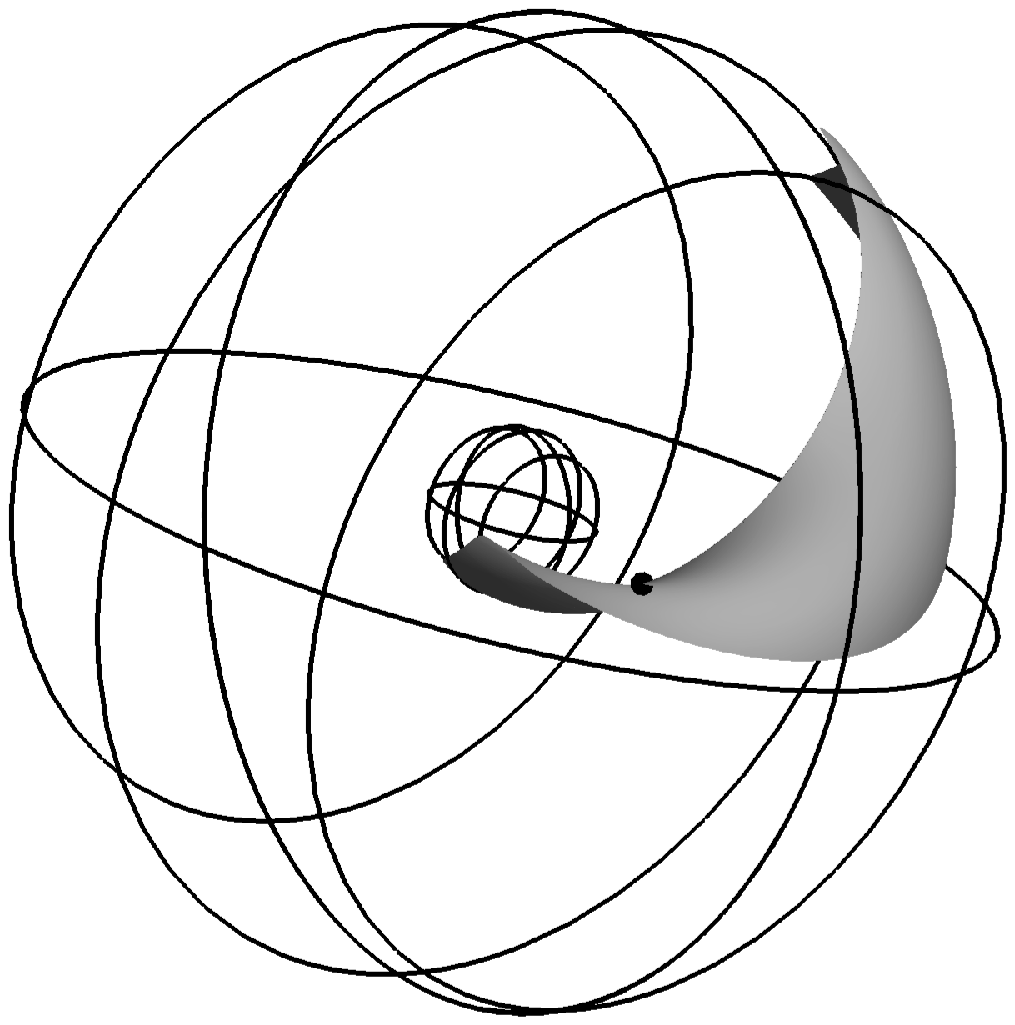}} &
  \resizebox{3cm}{!}{\includegraphics{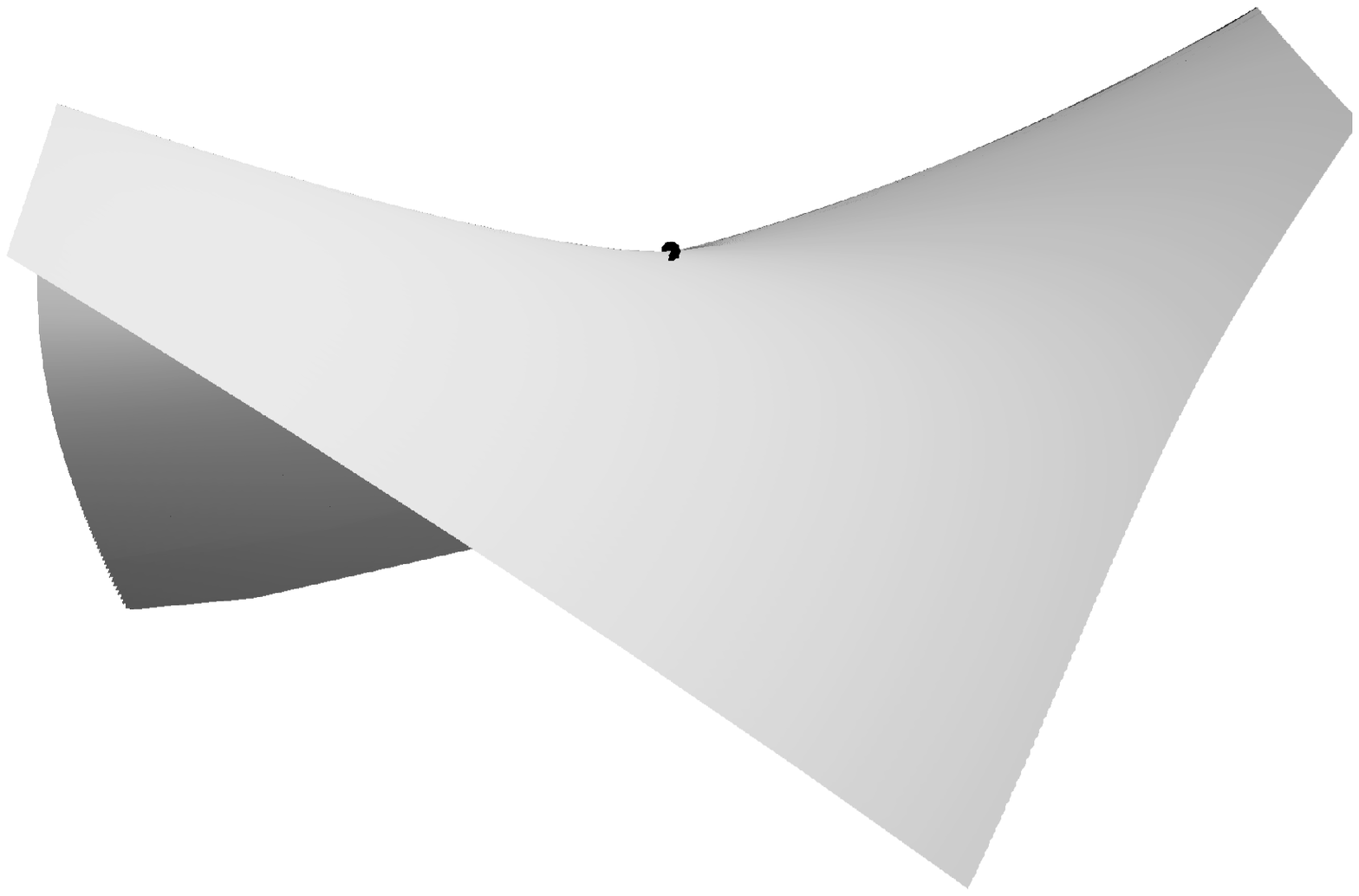}}
 \end{tabular}
\end{center}
\caption{
A CMC $1$ face with cuspidal butterfly
given in Example \ref{ex:cuspidal-butterfly} with $\constant=-\sqrt{-1}$.
This is the conjugate of the CMC $1$ face
given in Figure \ref{fig:cS1-plus}.
} 
\label{fig:cuspidal-butterfly}
\end{figure}

\section{Non-degenerate singular points on CMC $1$ faces}
\label{sec:final}

We obtained the criteria for several singular points,
such as 
$5/2$-cuspidal edge singularities
(Theorem \ref{thm:to-prove}),
cuspidal butterfly and
generalized $A_k$ singularities 
(Theorem \ref{thm:Ak-criteria-CMC1}),
and cuspidal $S_k$ singularities
(Theorem \ref{thm:Sk-criteria-CMC1}),
on CMC $1$ faces in the de Sitter $3$-space $S^3_1$
in terms of the Weierstrass data.
These criteria yield Theorems \ref{thm:A} and \ref{thm:B}.
To summarize, we conclude that
CMC $1$ faces in $S^3_1$ have the following dualities:
\begin{itemize}
\item[(i)]
The cuspidal edge singularity is self-dual.
The duality between swallowtail singularity
and cuspidal cross cap singularity
(\cite{FSUY}).
\item[(ii)]
The duality between 
generalized conelike singularity and 
$5/2$-cuspidal edge singularity
(Theorem \ref{thm:A}).
\item[(iii)]
The duality between 
admissible cuspidal butterfly singularity and 
admissible cuspidal $S_1$ singularity
(Theorem \ref{thm:B}, (1)).
The duality between 
generalized $A_{k+3}$ singularity
and 
cuspidal $S_k$ singularity,
where $k\,(\geq2)$ is an integer
(Theorem \ref{thm:B}, (2)).
\end{itemize}
These dualities (i)--(iii) yield that
CMC $1$ faces in $S^3_1$ and 
maxfaces in the Lorentz-Minkowski $3$-space $\L^3$
have the different behavior at singular points.

Furthermore, 
as a corollary of 
Theorems \ref{thm:to-prove}, \ref{thm:Ak-criteria-CMC1} and \ref{thm:Sk-criteria-CMC1},
we give a classification of non-degenerate singular points
on CMC $1$ faces in $S^3_1$ (Theorem \ref{thm:C}).
%
For the proof of Theorem \ref{thm:C},
we prepare the following.

\begin{lemma}\label{lem:wavefront-singularity}
Let $f: D\to S^3_1$ be a CMC $1$ face.
Suppose that $f$ is a wave front at a singular point $p\in D$.
Then $p$ either
cuspidal edge, swallowtail, cuspidal butterfly,
a generalized $A_k$ singular point $(k\geq5)$, or
a generalized conelike singular point.
\end{lemma}

\begin{proof}
Let $(g,\omega)$ be a Weierstrass data of $f$,
and $(\varphi,\vf)$ be the characteristic pair associated with $(g,\omega)$.
We denote by $\xi$ the singular directional vector field given by \eqref{eq:xi-eta}.
By Fact \ref{fact:singular-FSUY}, we have $\Re \varphi(p)\ne0$.
In particular, $p$ is a non-degenerate singular point.
We consider the function $\Im \varphi$ 
along the singular set $\Sigma(f)$ near $p$.
If $\Im \varphi(p)\ne0$,
then $p$ must be a cuspidal edge by 
Fact \ref{fact:singular-FSUY}.
If $\Im \varphi$ is identically zero along $\Sigma(f)$,
then $p$ must be a generalized conelike singular point by 
Proposition \ref{prop:conelike-CMC1}.
Now we assume that $\Im \varphi$ has an isolated zero at $p$.
Then there exists an integer $k\,(\geq3)$
such that 
$\xi^\ell(\Im \varphi)(p)=0$ $(\ell=0,\dots,k-3)$
and 
$\xi^{k-2}(\Im \varphi)(p)\ne0$. 
By Proposition \ref{prop:Sk-induction},
$$
\Im (\sqrt{-1}^\ell \vf^\ell \varphi)=0 \text{~for~} \ell=0,\dots,k,
\quad\text{and}\quad 
\Im (\sqrt{-1}^{k+1} \vf^{k+1} \varphi)\ne0
$$
holds at $p$.
If $k=3$, then $p$ must be a swallowtail singular point
by Fact \ref{fact:singular-FSUY}.
If $k=4$, then $p$ must be a cuspidal butterfly singular point
by Theorem \ref{thm:Ak-criteria-CMC1}.
If $k\geq5$, 
Theorem \ref{thm:Ak-criteria-CMC1} yields that 
$p$ must be a generalized $A_k$ singular point.
\end{proof}

\begin{proof}[Proof of Theorem \ref{thm:C}]
Let $f:D\to S^3_1$ be a CMC $1$ face with Weierstrass data $(g,\omega)$,
and $p\in D$ be a non-degenerate singular point.
We denote by $(\varphi,\vf)$ the characteristic pair associated with $(g,\omega)$,
and let $\xi$ be the singular directional vector field given by \eqref{eq:xi-eta}.
Since we have treated the case that $f$ is a wave front at $p$
in Lemma \ref{lem:wavefront-singularity},
we assume that $f$ is not a wave front at $p$.
Fact \ref{fact:singular-FSUY} yields $\Re \varphi(p)=0$.
Since $p\in D$ is a non-degenerate singular point,
Fact \ref{fact:singular-FSUY} yields that $\Im \varphi(p)\ne0$, 
and hence $p$ is of the first kind.
If $\Re \varphi$ is identically $0$ along a singular curve passing through $p$,
then Theorem \ref{thm:to-prove} yields that 
$p$ must be a $5/2$-cuspidal edge singular point.
Now we assume that $\Re \varphi$ has an isolated zero at $p$.
Then there exists a non-negative integer $k$
such that 
$\xi^\ell(\Re \varphi)(p)=0$ $(\ell=0,\dots,k)$
and 
$\xi^{k+1}(\Re \varphi)(p)\ne0$. 
By Proposition \ref{prop:Sk-induction},
$$
\Re (\sqrt{-1}^\ell \vf^\ell \varphi)=0 \text{~for~} \ell=0,\dots,k,
\quad\text{and}\quad 
\Re (\sqrt{-1}^{k+1} \vf^{k+1} \varphi)\ne0
$$
holds at $p$.
If $k=0$, then $p$ must be a cuspidal cross cap singular point
by Fact \ref{fact:singular-FSUY}.
If $k\geq2$, 
Theorem \ref{thm:Sk-criteria-CMC1} yields that 
$p$ must be a cuspidal $S_k$ singular point.
If $k=1$, Theorem \ref{thm:Sk-criteria-CMC1} yields that $p$ must be 
either a cuspidal $S_1$ singular point (if $\sigma(f,p)\ne12$),
or a singular point satisfying the condition \eqref{eq:condition-S}
and $\sigma(f,p)=12$.
\end{proof}

In the case of zero mean curvature surfaces (ZMC surfaces, for short) in $\L^3$, 
degenerate lightlike points are shown to be lying on lightlike lines
(\cite{Klyachin}, \cite{UY_lightlike1}, \cite{UY_lightlike2}).
However, in our case, 
degenerate singular points are isolated, 
since they occur at the branch points of the meromorphic function $g$
(cf.\ Fact \ref{fact:singular-FSUY}).

We set a smooth map
$\fold : \R^2 \to \R^3$ as $\fold(u,v) := (u,v^2,0)$.
Let $f:D\to M^3$ be a smooth map into a $3$-manifold $M^3$.
If the map germ $f$ at a point $p\in D$ 
is $\mathcal{A}$-equivalent to 
$\fold$ at the origin,
then $f$ is said to have a {\it fold singular point} at $p\in D$.
If a maxface in $\L^3$ admits fold singular points,
it can be extended to a timelike minimal surface analytically
(\cite{KKSY, FKKRSUYY}).
The resulting surface is a zero mean curvature surface of mixed type.
See \cite{FKKRUY} for details.
However, in the case of CMC $1$ faces in $S^3_1$,
the following holds.

\begin{corollary}\label{cor:CMC1-foldsing}
CMC $1$ faces in $S^3_1$
do not admit fold singular points.
\end{corollary}

In \cite[Theorem 1.1]{HKS},
a similar statement is proved in the case of  
generalized spacelike CMC surfaces in $\L^3$.
While their statements are similar,
our proof of Corollary \ref{cor:CMC1-foldsing}
is different from that of \cite[Theorem 1.1]{HKS}.

On the other hand,
We set a smooth map
${f}_{(2k+1)/2} : \R^2 \to \R^3$ as 
${f}_{(2k+1)/2}(u,v) := (u,v^2,v^{2k+1})$.
Let $f:D\to M^3$ be a smooth map into a $3$-manifold $M^3$.
If the map germ $f$ at a point $p\in D$ 
is $\mathcal{A}$-equivalent to 
$\fold$ at the origin,
then $f$ is said to have a {\it fold singular point} at $p\in D$.
In \cite[Remark 4.9]{HKS}, it was pointed out that
maxfaces in $\L^3$ do not admit $5/2$-cuspidal edge.
By an argument similar to that of \cite[Remark 4.9]{HKS},
the following holds:
\[\begin{minipage}{0.8\linewidth}
{\it 
Maxfaces in $\L^3$ do not admit $(2k+1)/2$-cuspidal edges,
other than cuspidal edges {\rm (}i.e., $k=1${\rm )}.
}\end{minipage}\]
In the case of CMC $1$ faces in $S^3_1$,
the following holds.

\begin{corollary}\label{cor:CMC1-odd-ce}
CMC $1$ faces in $S^3_1$
do not admit $(2k+1)/2$-cuspidal edges,
other than cuspidal edges {\rm (}i.e., $k=1${\rm )}
and $5/2$-cuspidal edges {\rm (}i.e., $k=2${\rm )}.
\end{corollary}

\begin{proof}[Proofs for Corollaries \ref{cor:CMC1-foldsing} and \ref{cor:CMC1-odd-ce}]
Since fold singular points and $(2k+1)/2$-cuspidal edges are non-degenerate,
Theorem \ref{thm:C} yields the desired results.
\end{proof}

\appendix
\section{Calculation of $\zeta^5f$ (Proof of Lemma \ref{lem:T5p-body})}
\label{sec:app}

In this appendix,
we give a proof of Lemma \ref{lem:T5p-body}.
Let us fix the notations.
We set $i=\sqrt{-1}$.
Let $f:D\to S^3_1$ be a CMC $1$ face
with Weierstrass data $(g,\omega)$,
and $(\varphi,\vf)$ be the characteristic pair.
Suppose that 
$p\in D$ is a non-degenerate singular point
satisfying the condition \eqref{eq:condition-S},
i.e., $\Re (\varphi(p)) = \Im (\vf \varphi(p)) =0$.
Then, there exist
$\Imphi ,\reDphi \in \R$ and $\DD \in \C$, 
such that $\Imphi \ne0$ and
$$
  \varphi(p) = i\Imphi ,\quad
  \vf \varphi(p) = \reDphi ,\quad 
  \vf^2\varphi(p) = \DD.
$$
Also, we set
$$
  \lambda:=1-g\bar{g},\qquad
  \theta := \frac{i}{g\homega } + \const\,\lambda^2,
$$
cf.\ \eqref{eq:zeta}.
By a direct calculation, we have the following.

\begin{lemma}\label{lem:lambda-diff}
It holds that
\begin{gather*}
\varphi_z
=g\homega \varphi\, \vf \varphi,
\quad
\varphi_{zz}
=g^2\homega^2 \varphi^2\, \vf^2\varphi + g\varphi\, \vf \varphi 
  \left\{g\homega^2 (\varphi +\vf \varphi)+\homega_z\right\},\\
g_{z} =g^2 \homega  \varphi,\quad
g_{zz}=
g^2 \varphi \left\{g \homega ^2 (\vf \varphi+2 \varphi)+\homega_{z}\right\},\\
g_{zzz}
=
g^2 \varphi\left[
g^2 \homega ^3 \left\{\varphi (\vf^2\varphi+6 \varphi)
+\vf \varphi(\vf \varphi + 7 \varphi )
\right\}+3 g \homega  (\vf \varphi+2 \varphi) \homega_{z}+\homega_{zz}
\right].
\end{gather*}
In particular, at the singular point $p$,
we have
\begin{gather*}
\varphi_z(p)
=i \Imphi  \reDphi  g \homega, 
\quad
\varphi_{zz}(p)
= \Imphi  g 
  \left[
   g \homega^2 \left\{ -\Imphi ( \reDphi +\DD) + i \reDphi ^2 \right\} 
   + i \reDphi  \homega_z 
  \right],\\
g_{z}(p)=i \Imphi  g^2 \homega,
\quad
g_{zz}(p)=
\Imphi  g^2 \left\{- (2 \Imphi -i \reDphi )g \homega  ^2 +i \homega _z  \right\},\\
g_{zzz}(p)
= \Imphi  g^2\left\{
\left(-\Imphi  \DD - 7 \Imphi \reDphi + i(-6  \Imphi ^2+ \reDphi ^2) \right)g^2 \homega  ^3
-3 g (2 \Imphi -i \reDphi ) \homega   \homega _z 
+i \homega _{zz}
\right\},
\end{gather*}
where the right hand sides are evaluated at $p$.
Moreover, the derivatives of 
$\lambda$ and $\theta$
are given by
\begin{gather*}
\lambda_{z}(p)
=
-i \Imphi  g \homega,\quad
\lambda_{zz}(p)
=
\Imphi  g
\left\{
(2 \Imphi -i \reDphi )g \homega  ^2 
-i \homega _z
\right\},\quad
\lambda_{z\bar{z}}(p)
=
-\Imphi ^2|\homega|^2,\\
\lambda_{zzz}(p)
=
\Imphi  g
\left[
\left\{ \Imphi (7 \reDphi + \DD )  + i (6 \Imphi ^2- \reDphi ^2)\right\}g^2 \homega ^3 
+3 g \homega  (2 \Imphi -i \reDphi ) \homega _z
-i \homega _{zz}
\right],
\\
\lambda_{zz\bar{z}}(p)
=
\Imphi ^2\bar{\homega }
\left\{
- (\reDphi +2 i \Imphi )g \homega ^2
-\homega _z
\right\},\\
\theta_{z}(p)
=\Imphi - i \frac{\homega_z}{g\homega^2},
\quad
\theta_{\bar{z}}(p)=0,
\quad
\theta_{z\bar{z}}(p)
= -\Imphi \frac{\bar{\homega}}{g},
\quad
\theta_{\bar{z}\bar{z}}(p)
=\Imphi \frac{\bar{\homega}^2}{g^3 \homega}
\\
\theta_{zz}(p)
=
\Imphi \left\{ (\reDphi+1) g \homega-\frac{\homega_z}{\homega}\right\}
+\frac{i}{g \homega ^3}\left(2 \homega_z^2-\homega \homega_{zz}\right),
\end{gather*}
where the right hand sides are evaluated at $p$.
\end{lemma}

Let $\calE$ be the $3$-dimensional real subspace of $\Herm(2)$
given by
$\calE:={\rm Span}_{\R}\{\ep_0,\ep_1,\ep_2\}$,
where
$$
  \ep_0 := 
  \begin{pmatrix} 
     1 & 0 \\ 
     0 & -1
   \end{pmatrix},\quad
  \ep_1 := 
  \begin{pmatrix} 
     1 & g(p) \\ 
     \overline{g(p)} & 1
   \end{pmatrix},\quad
  \ep_2 := i
   \begin{pmatrix}
     0 & g(p) \\
     -\overline{g(p)} & 0
   \end{pmatrix}.
$$
For $X_1,X_2\in \Herm(2)$, we define
$$
  X_1\equiv X_2
  \iff  
  X_1- X_2 \in \calE.
$$
If we set
$$
  \ep_3 := 
   \begin{pmatrix}
     1 & -g(p) \\
     -\overline{g(p)} & 1
   \end{pmatrix},\qquad
  Y:=
   \begin{pmatrix}
     0 & g(p) \\
     \overline{g(p)} & 0
   \end{pmatrix},
$$
$\ep_3\equiv -2Y$ since $\ep_3+2Y=\ep_1$.
Moreover 
${\rm Span}_{\R}\{\ep_0,\ep_1,\ep_2,\ep_3\}=\Herm(2)$.
The complexification 
$\calE^{\C}={\rm Span}_{\C}\{\ep_0,\ep_1,\ep_2\}$ of 
$\calE$ is a $3$-dimensional complex subspace of $M_2(\C)$.
For $ X_1,X_2\in M_2(\C)$, we define
$$
  X_1\sim X_2
  \iff  
  X_1- X_2 \in \calE^{\C}.
$$
We also denote by $\sim$ the equivalence relation `$\equiv$'.
We set a complex linear map
$\mu: M_2(\C) \to \C$ as
$$
  \mu(X) := \frac{1}{2}\left( \overline{g(p)}\,b+g(p)\,c -a-d \right)
  \qquad
  \left(
    X=\begin{pmatrix} a&b\\c&d\end{pmatrix}\in M_2(\C)
  \right).
$$
Then, we have $\mu(X^*)=\overline{\mu(X)}$.
Moreover, for arbitrary $X\in M_2(\C)$,
$$
  X=\frac{a-d}{2}\ep_0
  +\frac{a+d}{2}\ep_1
  +\frac{1}{2i}\left( \overline{g(p)}\,b - g(p)\, c \right)\ep_2
  +\mu(X)\,Y
$$
holds, which implies the following.

\begin{lemma}\label{lem:modulo}
For $X\in M_2(\C)$,
it holds that
$X\equiv \mu (X) \,Y$.
\end{lemma}

Let $T_{2, j}$ $(j=0,1,2,3,4,5)$ be 
the matrices defined in Lemma \ref{lem:T2-j}.
Setting $\mu_j:=\mu(T_{2,j})$,
we have
$$
\mu\left( T_2 \right)
=
\mu_0+ \lambda \mu_1+  \lambda^2 \mu_2
+  \lambda^3 \mu_3
+  \lambda^4 \mu_4
+  \lambda^5 \mu_5.
$$

\begin{lemma}\label{lem:mu-diff}
It holds that
$\mu_1(p)=\mu_2(p)=0$,
$\mu_3(p)=-1/{2}$,
and
\begin{gather*}
  (\mu_0)_{z}(p) \vphantom{\frac1{6}}
  = \Imphi ^2 g \homega,\quad
  (\mu_1)_{z}(p) \vphantom{\frac1{6}}
  = \frac{1}{2} \Imphi  (\Imphi +i \reDphi ) g \homega,\quad
  (\mu_2)_z(p) \vphantom{\frac1{6}}=\frac{1}{2} i \Imphi  g \homega,\\
  (\mu_0)_{z\bar{z}}(p) \vphantom{\frac1{6}}
  =\Imphi ^2 \reDphi  |\homega|^2,\qquad
  (\mu_1)_{z\bar{z}}(p)  \vphantom{\frac1{6}}
  =-\Imphi ^2 |\homega|^2,\\
  (\mu_0)_{zz}(p) \vphantom{\frac1{6}}
  = \Imphi ^2 g \left\{ 2 (\reDphi +i \Imphi ) g \homega^2 + \homega_z\right\},\\
%
  (\mu_1)_{zz}(p) \vphantom{\frac1{6}}
  = \frac{1}{2} \Imphi g \left[ 
        \left\{\Imphi  (2 \reDphi +4- \DD) +i \reDphi ^2\right\}g \homega^2 
        +(\Imphi +i \reDphi ) \homega_z
      \right],\\
  (\mu_0)_{zz\bar{z}}(p) \vphantom{\frac1{6}}
  =\frac{1}{2} \Imphi ^2 \bar{\homega} 
  \left[
    \left\{2 \reDphi ^2+ i \Imphi ( \DD +3 \reDphi) \right\}g \homega ^2 
    + 2 \reDphi  \homega_z
  \right],\\
  (\mu_0)_{zzz}(p) \vphantom{\frac1{6}}
  =\frac{1}{2} \Imphi ^2 g 
  \left[
    \left\{-12 \Imphi ^2+8 \reDphi ^2
    + i\Imphi ( 5  \DD +23  \reDphi) \right\}g^2 \homega ^3 
    +12 (\reDphi +i \Imphi )g \homega \homega_z +2 \homega_{zz}
  \right],
\end{gather*}
where the right hand sides are evaluated at $p$.
\end{lemma}

\begin{proof}
First of all,
we have
$
\mu_j
=\mu(T_{2,k})
=\mu(U_{2,k})+\mu((U_{2,k})^*)
=\mu(U_{2,k})+\overline{\mu(U_{2,k})}.
$
By Lemma \ref{lem:T2-j}, we obtain
\begin{multline*}
\mu(U_{2,3})=
-\const g^2 \bar{g} \homega \varphi 
\left\{ 
\const \homega  \left( \overline{g(p)} g^2 -2g +g(p) \right)
+\bar{\const} \bar{\homega} (\overline{g(p)} -2\bar{g} +g(p) \bar{g}^2)
\right\}\\
-\frac{i \const \homega}{2}
\left(
 g(p)(\varphi+1) - g 
 +\frac{ \overline{g(p)} g -1}{\bar{g}}
\right)
+\frac{i \bar{\const} \bar{\homega}}{2}
\left(
 \overline{g(p)}-\bar{g}
 +\frac{g(p)\bar{g}-1}{g}
\right).
\end{multline*}
Thus we have $\mu(U_{2,3})(p)=-1/4$,
and hence
$
\mu_3(p)
=2\Re \mu(U_{2,3})
=-1/2
$
holds.
{\allowdisplaybreaks
Next,
Lemma \ref{lem:T2-j} yields 
$\mu(U_{2,0})=\frac1{2}(g(p)\bar{g}-\overline{g(p)} g)\varphi $,
\begin{align*}
\mu(U_{2,1})&=
- i g \bar{g} \varphi 
\left\{
\const \homega \left(\overline{g(p)} g^2-2g+g(p)\right)
+\bar{\const} \bar{\homega} \left(\overline{g(p)}- 2 \bar{g} +g(p) \bar{g}^2 \right) 
\right\}\\
&{}\hspace{30mm}+\frac1{2}\left(
\frac{\overline{g(p)}}{\bar{g}}+\frac{g(p)(\varphi+1)}{g}
-1-\frac{1}{g\bar{g}}
\right),
\\
\mu(U_{2,2})&=\frac{i\const}{2}
\left\{
g \homega \varphi \left(3 \overline{g(p)} g - g(p) \bar{g} -2\right) 
+\frac{\homega_{z}}{\homega } \left(\overline{g(p)} g +\frac{g(p)}{g}-2 \right)
\right\}.
\end{align*}
Hence,
$\mu_0=\frac1{2}
(g(p)\bar{g}-\overline{g(p)} g)\left(\varphi -\bar{\varphi}\right) $,
\begin{align*}
\mu_1&=
- i g \bar{g}(\varphi -\bar{\varphi}) 
\left\{
\const \homega \left(\overline{g(p)} g^2-2g+g(p)\right)
+\bar{\const} \bar{\homega} \left(\overline{g(p)}- 2 \bar{g} +g(p) \bar{g}^2 \right) 
\right\}\\
&\hspace{8mm}
+\frac{\overline{g(p)}}{\bar{g}}
+\frac{g(p)}{g}
-1-\frac{1}{g\bar{g}}
+\frac1{2}
\left(
\frac{g(p)\varphi}{g}
+\frac{\overline{g(p)}\bar{\varphi}}{\bar{g}}
\right)
\\
\mu_2&=
\frac{i\const}{2}
\left\{
g \homega \varphi \left(3 \overline{g(p)} g - g(p) \bar{g} -2\right) 
+\frac{\homega_{z}}{\homega } \left(\overline{g(p)} g +\frac{g(p)}{g}-2 \right)
\right\}\\
&\hspace{8mm}
-
\frac{i\bar{\const}}{2}
\left\{
\bar{g} \bar{\homega} \bar{\varphi} \left(3 g(p) \bar{g} - \overline{g(p)} g -2\right) 
+\frac{\overline{\homega_{z}}}{\bar{\homega} } 
\left(g(p) \bar{g} +\frac{\overline{g(p)}}{\bar{g}}-2 \right)
\right\}.
\end{align*}
A} straightforward calculation, we have the desired results.
\end{proof}

\begin{lemma}
\label{lem:T2-diff}
The derivatives of $T_2$ satisfies
\begin{gather*}
  (T_2)_{z}(p)
  \equiv \Imphi ^2 g \homega  Y,
  \quad
  (T_2)_{zz}(p)
  \equiv 
   \Imphi ^2 g \left\{
   g \homega ^2 \left(3 \reDphi  + i \Imphi  \right)+\homega _z
   \right\}  Y,
   \quad
  (T_2)_{z\bar{z}}(p)
  \equiv O\\
  (T_2)_{zzz}(p)
  \equiv 
  \Imphi ^2 g \left[
  g^2 \homega ^3 \left\{-3 \Imphi ^2 + 7 \reDphi ^2 
  +2 i \Imphi  (5 \reDphi -6+ 2\DD)\right\}
  +3g \homega  (3 \reDphi +i \Imphi ) \homega_z
  +\homega_{zz}
  \right] Y,
\end{gather*}
and $  (T_2)_{zz\bar{z}}(p)
  \equiv 
  10 i \Imphi ^3 |\homega |^2 \homega  g 
  \, Y,
$
where the right hand sides are evaluated at $p$.
\end{lemma}

\begin{proof}
By Lemma \ref{lem:modulo},
$
\frac{\partial^{k+l}}{\partial z^j \partial \bz^l }(T_2)(p)
\equiv
\mu\left(
\frac{\partial^{k+l}}{\partial z^j \partial \bz^l }(T_2)(p)
\right)\,Y.
$
Hence it suffices to calculate
$\mu\left((T_2)_{z}(p)\right)$,
$\mu\left((T_2)_{zz}(p)\right)$,
$\mu\left((T_2)_{z\bz}(p)\right)$,
$\mu\left((T_2)_{zzz}(p)\right)$,
$\mu\left((T_2)_{zz\bz}(p)\right)$.
Since
$
\mu\left(
\frac{\partial^{k+l}}{\partial z^j \partial \bz^l }T_2
\right)
=
\frac{\partial^{k+l}}{\partial z^j \partial \bz^l }
\mu\left(
T_2
\right),
$
we have
$\mu\bigl((T_2)_z(p)\bigr)= (\mu_0)_{z} +\lambda_z \mu_1$,
\begin{align*}
  \mu\bigl((T_2)_{zz}(p)\bigr)
  &= (\mu_0)_{zz} +\lambda_{zz} \mu_1+2\lambda_{z} (\mu_1)_z
       + 2(\lambda_z)^2 \mu_2\\
  \mu\bigl((T_2)_{z\bar{z}}(p)\bigr)
  &= (\mu_0)_{z\bar{z}} +\lambda_{z\bar{z}} \mu_1
    +\lambda_{z} (\mu_1)_{\bar{z}}+\lambda_{\bar{z}} (\mu_1)_{z}
    + 2\lambda_z\lambda_{\bar{z}}\mu_2\\
  \mu\bigl((T_2)_{zzz}(p)\bigr)
  &= (\mu_0)_{zzz}
    +\lambda_{zzz} \mu_1+3\lambda_{zz} (\mu_1)_z + 3\lambda_z (\mu_1)_{zz}
    + 6\lambda_z \lambda_{zz} \mu_2 + 6 (\lambda_z)^2 (\mu_2)_z
    + 6(\lambda_z)^3 \mu_3 \\
  \mu\bigl((T_2)_{zz\bz}(p)\bigr)
  &= (\mu_0)_{zz\bz} 
    +\lambda_{zz\bz} \mu_1+\lambda_{zz} (\mu_1)_{\bz} 
    + 2\left\{\lambda_{z\bz}(\mu_1)_z + \lambda_z(\mu_1)_{z\bz} \right\}\\
    &\hspace{4mm}
    + 2(2\lambda_z \lambda_{z\bz}+\lambda_{\bz} \lambda_{zz}) \mu_2 
    + 2(\lambda_z)^2 (\mu_2)_{\bz} 
    +4\lambda_z\lambda_{\bz} (\mu_2)_z
    + 6(\lambda_z)^2\lambda_{\bz} \mu_3 
+\lambda_{\bar{z}}(\mu_1)_{zz},
\end{align*}
where the right hand sides are evaluated at $p$.
Substituting the results of Lemmas \ref{lem:lambda-diff},
\ref{lem:mu-diff} 
into the above, we obtain the assertion.
\end{proof}


\begin{lemma}\label{lem:Vp}
Let $\Psi$, $\Omega$ be the matrices
as in Lemma \ref{lem:Tk}.
We have the following.
\begin{itemize}
\item[$(1)$]
$\Psi(p) M_2(\C) \subset \calE^{\C}$,
\item[$(2)$]
$\Psi_z(p) \calE^{\C}\subset \calE^{\C}$,
\item[$(3)$]
$
\Psi_z(p)Y \equiv  - \Imphi  g(p)\homega (p) Y,
$
\item[$(4)$]
$\Omega(p) \ep_2 \equiv O$,
$\Psi_{zz}(p)\ep_2 \equiv -\Imphi ^2g(p)^2\homega (p)^2 Y$.
\end{itemize}
\end{lemma}

\begin{proof}
For any $X\in M_2(\C)$,
a direct calculation shows that
$\mu(\Psi(p) X)=0$,
and hence (1) holds.
Since
$\mu(\Psi_z(p)\ep_0) =\mu(\Psi_z(p)\ep_1) =\mu(\Psi_z(p)\ep_2) =0$,
we have (2).
By
$\mu(\Psi_z(p)Y) = - \Imphi  g(p)\homega (p)$,
(3) holds.
Finally,
since
$\mu(\Omega(p) \ep_2) = 0$,
$\mu(\Psi_{zz}(p)\ep_2) = -\Imphi ^2g(p)^2\homega (p)^2$,
we obtain (4).
\end{proof}

\begin{lemma}\label{lem:T3zz}
The derivatives of $T_3$ satisfies
\begin{align*}
  (T_3)_{z}(p)
  &\equiv
  3i \Imphi ^2 \reDphi  g\homega Y,\\
  (T_3)_{zz}(p)
  &\equiv
  \Imphi ^2g\left\{
   g \homega ^2 \left(-3 \Imphi  \reDphi +24 \Imphi- 4  \Imphi  \DD   
   +i(-  \Imphi ^2  + 7  \reDphi ^2)\right)
  +3 i \reDphi  \homega_z \right\} Y,\\
  (T_3)_{z\bar{z}}(p)
  &\equiv
  -24 \Imphi ^3 |\homega |^2 Y,
\end{align*}
where the right hand sides are evaluated at $p$.
\end{lemma}

\begin{proof}
Lemma \ref{lem:Tk} yields that
$
  T_3=
  \left( \Psi + \lambda^2 \Omega \right)T_{2}
  + \theta(T_{2})_z + \bar{\theta}(T_{2})_{\bar{z}}
  + T_{2} \left( \Psi^*+\lambda ^2 \Omega^* \right).
$
By a direct calculation,
we obtain
\begin{align*}
  (T_3)_z&=
  \left( \Psi + \lambda^2 \Omega \right)_zT_{2}
  +\left( \Psi + \lambda^2 \Omega \right)(T_{2})_z
  + \theta_z(T_{2})_z 
  + \theta(T_{2})_{zz}\\ 
  &\hspace{50mm}
  + \bar{\theta}_z(T_{2})_{\bar{z}}
  + \bar{\theta}(T_{2})_{z\bar{z}}
  + (T_{2})_z \left( \Psi^*+\lambda ^2 \Omega^* \right),\\
  (T_3)_{zz}&=
  \left( \Psi + \lambda^2 \Omega \right)_{zz}T_{2}
  +2\left( \Psi + \lambda^2 \Omega \right)_z(T_{2})_z
  +\left( \Psi + \lambda^2 \Omega \right)(T_{2})_{zz}\\ 
  &\hspace{20mm}
  + \theta_{zz}(T_{2})_{z} 
  + 2\theta_z(T_{2})_{zz}
  + \theta(T_{2})_{zzz}
  + \bar{\theta}_{zz}(T_{2})_{\bar{z}}\\ 
  &\hspace{40mm}
  + 2\bar{\theta}_z(T_{2})_{z\bar{z}}
  + \bar{\theta}(T_{2})_{zz\bar{z}}
  + (T_{2})_{zz} \left( \Psi^*+\lambda ^2 \Omega^* \right),\\
  (T_3)_{z\bar{z}}&=
  \left( \Psi + \lambda^2 \Omega \right)_z(T_{2})_{\bar{z}}
  +\left( \Psi + \lambda^2 \Omega \right)(T_{2})_{z\bar{z}}
  + \theta_{z\bar{z}}(T_{2})_z 
  + \theta_z(T_{2})_{z\bar{z}} \\
  &\hspace{10mm}
  + \theta_{\bar{z}}(T_{2})_{zz}
  + \theta(T_{2})_{zz\bar{z}}
  + \bar{\theta}_{z\bar{z}}(T_{2})_{\bar{z}}
  + \bar{\theta}_z(T_{2})_{\bar{z}\bar{z}}\\ 
  &\hspace{12mm}
  + \bar{\theta}_{\bar{z}}(T_{2})_{z\bar{z}}
  + \bar{\theta}(T_{2})_{z\bar{z}\bar{z}}
  + (T_{2})_{z\bar{z}} \left( \Psi^*+\lambda ^2 \Omega^* \right)
  + (T_{2})_z \left( \Psi^*+\lambda ^2 \Omega^* \right)_{\bar{z}}.
\end{align*}
Substituting the results of Lemmas 
\ref{lem:lambda-diff},  \ref{lem:T2-diff}, \ref{lem:Vp}
into the above,
and evaluating this at $p$, we obtain the assertion.
\end{proof}

Now,
Lemma \ref{lem:T5p-body} is a direct conclusion of the following.

\begin{proposition}\label{prop:T5p}
It holds that
$
  T_5(p)
\equiv  4(12 - \Re \DD)\Imphi ^3\ep_3.
$
\end{proposition}

\begin{proof}
By Lemma \ref{lem:Tk},
we have
$$
  U_5=
  \left( \Psi + \lambda^2 \Omega \right)T_4 + \theta(T_4)_z,
  \qquad
  U_5(p)
  = \Psi(p)T_4(p) + \frac{i}{g(p)\homega(p)} (T_4)_z(p).
$$
Lemma \ref{lem:Vp} yields that
$ U_5(p) \equiv \frac{i}{g(p)\homega(p)} (T_4)_z(p) $.
With respect to $(T_4)_z(p)$,
Lemma \ref{lem:Tk} implies
$
  T_4=
  \left( \Psi + \lambda^2 \Omega \right)T_{3}
  + \theta(T_{3})_z + \bar{\theta}(T_{3})_{\bar{z}}
  + T_{3} \left( \Psi^*+\lambda ^2 \Omega^* \right).
$
Substituting the results of Lemmas 
\ref{lem:lambda-diff}, \ref{lem:Vp} and $T_3(p)=O$
into
\begin{multline*}
  (T_4)_z=
  \left( \Psi + \lambda^2 \Omega \right)_zT_{3}
  +\left( \Psi + \lambda^2 \Omega \right)(T_{3})_z
  + \theta_z(T_{3})_z 
  + \theta(T_{3})_{zz}\\ 
  + \bar{\theta}_z(T_{3})_{\bar{z}}
  + \bar{\theta}(T_{3})_{z\bar{z}}
  + (T_{3})_z \left( \Psi^*+\lambda ^2 \Omega^* \right),
\end{multline*}
we have
\begin{equation}\label{eq:T4z}
  (T_4)_z(p)
  \equiv  \theta_z(p)(T_{3})_z(p) 
  + \theta(p)(T_{3})_{zz}(p)  
  + \bar{\theta}(p)(T_{3})_{z\bar{z}}(p).
\end{equation}
Hence by Lemmas
\ref{lem:lambda-diff} and \ref{lem:T3zz},
it holds that
\begin{align*}
  (T_4)_z(p)
=-i \Imphi^2 g(p) \homega (p)\left\{ 4 \Imphi  (\DD -12 )+ i \left(\Imphi ^2 -7 \reDphi ^2\right) \right\} Y.
\end{align*}
Therefore we obtain
$
  T_5(p)
  \equiv  8(-12 + \Re \DD) \Imphi ^3 Y  \equiv  4(12 - \Re \DD)\Imphi ^3\ep_3.
$
\end{proof}

\begin{acknowledgements}
The authors 
would like to thank 
Shintaro Akamine,
Kentaro Saji, 
Yuta Ogata,
Keisuke Teramoto,
Masaaki Umehara and Kotaro Yamada for helpful comments.
\end{acknowledgements}


\end{document}